\documentclass[twocolumn,aip,reprint,jcp %fleqn,
,cha]{revtex4-1}

\usepackage{enumitem}
\usepackage{comment}
\usepackage{amsthm,mathtools}
\usepackage{amsmath,amssymb}

\usepackage[usenames,dvipsnames,svgnames,table]{xcolor}
\usepackage[colorlinks,linkcolor={blue!50!black},citecolor={blue!50!black}]{hyperref}

\newtheorem{theorem}{Theorem}

\newtheorem{assumption}[theorem]{Assumption}
\newtheorem{remark}[theorem]{Remark}

\newtheorem{corollary}[theorem]{Corollary}
\newtheorem{lemma}[theorem]{Lemma}
\newtheorem{definition}[theorem]{Definition}
\newtheorem{proposition}[theorem]{Proposition}
\renewcommand{\cal}[1]{\mathcal{#1}}

\renewcommand{\r}{\mathbb{R}}

\newcommand{\rn}{\mathbb{R}^n}

\newcommand{\norm}[1]{\left|\left|{#1}\right|\right|}

\newcommand{\Pb}{\mathbb{P}}

\newcommand{\Pbl}[1]{\mathbb{P}_\lambda\left(#1\right)} %_\gamma
\newcommand{\Pbp}[1]{\mathbb{P}_\pi\left(#1\right)}

\newcommand{\hd}{\#_d}

\renewcommand{\rhd}{\r^{\hd}}

\newcommand{\mg}[1]{\left|{#1}\right|}

\newcommand{\s}{\mathcal{S}}

\newcommand{\n}{\mathbb{N}}
\newcommand{\nn}{\mathbb{N}^n} % REMOVE

\newcommand{\zp}{\mathbb{Z}_+}

\newcommand{\ave}[1]{\left\langle #1 \right\rangle_\pi}
\newcommand{\aver}[1]{\left\langle #1 \right\rangle_{\pi^r}}
\newcommand{\averr}[1]{\left\langle #1 \right\rangle_{\pi_{|r}}}

\newcommand{\mmag}[1]{\left| #1 \right|}

\begin{document}
\title{Rigorous bounds on the stationary distributions of the chemical master equation via mathematical programming}
\author{Juan Kuntz}
\email{juan.kuntz08@imperial.ac.uk}
\affiliation{Department of Mathematics, Imperial College London, London, United Kingdom}
\affiliation{Department of Bioengineering, Imperial College London, London, United Kingdom}

\author{Philipp Thomas}
\email{p.thomas@imperial.ac.uk}
\affiliation{Department of Mathematics, Imperial College London, London, United Kingdom}

\author{Guy-Bart Stan}
\email{Corresponding author: g.stan@imperial.ac.uk}
\affiliation{Department of Bioengineering, Imperial College London, London, United Kingdom}

\author{Mauricio Barahona}
\email{Corresponding author: m.barahona@imperial.ac.uk}
\affiliation{Department of Mathematics, Imperial College London, London, United Kingdom}

\date{\today}

\begin{abstract}
The stochastic dynamics of biochemical networks are usually modelled with the chemical master equation (CME). The stationary distributions of CMEs are seldom solvable analytically, and numerical methods typically produce estimates with uncontrolled errors.
Here, we introduce mathematical programming approaches that yield approximations of these distributions with computable error bounds which enable the verification of their accuracy.
First, we use semidefinite programming to compute increasingly tighter upper and lower bounds on the moments of the stationary distributions for networks with rational propensities. 
Second, we use these moment bounds to formulate linear programs that yield convergent upper and lower bounds on the stationary distributions themselves,  their marginals and stationary averages. 
The bounds obtained also provide a computational test for the uniqueness of the distribution. In the unique case, the bounds form an approximation of the stationary distribution with a computable bound on its error. 
In the non-unique case, our approach yields converging approximations of the ergodic distributions. We illustrate our methodology through several biochemical examples taken from the literature: Schl\"ogl's model for a chemical bifurcation, a two-dimensional toggle switch, a model for bursty gene expression, and a dimerisation model with multiple stationary distributions.

\end{abstract}

\maketitle

\section{Introduction}
Cell-to-cell variability is pervasive in cell biology. A fundamental source of this variability is the fact that biochemical reactions inside cells often involve 
only a few molecules per cell~\cite{elowitz2002,taniguchi2010,uphoff2016}. Such reactions are key components in gene regulatory and signalling networks involved in cellular adaptation and cell fate decisions~\cite{arkin1998,Dandach2010,thomas2014}. 
Mathematically, stochastic reaction networks are modelled using continuous-time Markov chains whose distributions satisfy the chemical master equation (CME). As the availability of accurate single cell measurements widens, it is crucial to develop reliable methods for the analysis of the CME that facilitate parameter inference~\cite{munsky2015,frohlich2016}, both for   the identification of molecular mechanisms~\cite{neuert2013} and for the design of synthetic cellular circuits~\cite{Strelkowa:2010,Strelkowa:2011,oyarzun2014,zechner2016,TOMAZOU2018}.

Significant effort has been devoted to investigating the stationary solutions of CMEs, which determine the long time behaviour of the stochastic process~\cite{Meyn1993a}. While exact\cite{hemberg2007,Hemberg2008} Monte Carlo methods have been developed to sample from stationary solutions of some CMEs, analytical solutions are known only in a few special cases. 
In general, the CME is considered intractable\cite{Schnoerr2017} because, aside of systems with finite state space, it consists of an infinite set of coupled equations. 

An approach to circumvent the intractability of the full CME is to compute moments of its stationary solutions. However, moment computations are only exact for networks of unimolecular reactions; in all other cases, the equations of lower moments involve higher moments, leading to an infinite system of coupled equations that cannot be solved analytically. Moment closure schemes, usually requiring assumptions about the unknown solution, are thus employed to approximate the moments~\cite{engblom2006,engblom2014,lakatos2015,schnoerr2015,vanKampen1976}. Yet few of these methods provide quantified approximation errors~\cite{engblom2014}. 
Another approach is provided by mathematical programming techniques, which have been employed to compute bounds on the moments of Markov processes in various contexts. In such schemes, the finite set of moment equations is supplemented by moment inequalities and the moments are bounded by solving linear programs (LPs)~\cite{Schwerer1996,Helmes2003} or semidefinite programs (SDPs)~\cite{Hernandez-Lerma2003,Kuntz2016}. Alternatively, when the CME has a unique stationary solution, several truncation-based schemes\cite{Hart2012,Gupta2017,Dayar2011,Spieler2014,Dayar2011a} have been proposed to approximate the solution, 
although in most cases they do not provide estimates of the error they introduce.

Here, we present two different mathematical programming approaches that yield bounds on, and approximations of, the stationary solutions of the CME. 
Our first approach builds on our previous work~\cite{Kuntz2016,Kuntz2017,Kuntzthe},
and uses semidefinite programming to obtain upper and lower bounds on the moments of stationary solutions of networks with polynomial and rational propensities. 
The scheme constrains the possible solutions of a truncated, underdetemined set of moment equations by appending semidefinite inequalities that are satisfied by all probability distributions on the state space. Independently of this work, similar approaches have been recently proposed 
for networks with polynomial propensities~\cite{Kuntz2017,Sakurai2017,Dowdy2017,Ghusinga2017a}. Here we extend the mathematical framework to rational networks of interest in biochemistry, and we state rigorous, checkable mathematical conditions for the validity of the approach.

The second approach 
employs linear programming
to obtain lower and upper bounds on stationary averages of the CME, using state space truncations guided by moment bounds computed with our first (SDP) approach. 
In the case of a unique stationary solution, we prove that the LP bounds converge to the true average as the truncation approaches the entire state space.
Because stationary averages can be tailored to bound \emph{distributions}, the LP bounds provide approximations that converge in total variation to the stationary solution and its marginals 
with a computable approximation error (see also Ref.~\cite{Kuntz2018a}). Additionally, the LP bounds provide a computational test for the uniqueness of the stationary solution,
a prerequisite for most other approximation schemes.  In the non-unique case, the scheme provides converging approximations of the ergodic~distributions.
 
The paper is organised as follows. In Section~\ref{sec:stationary}, we introduce definitions regarding stationary solutions of the CME. Section~\ref{bmom} presents the SDP method to bound moments: first, the conceptual framework is introduced analytically for a simple birth-death process that displays a chemical bifurcation, followed by the general computational approach  for multi-species networks with rational propensities using semidefinite programming.
Section~\ref{bdist} presents the LP approach to bound and approximate entire stationary solutions, their averages and marginals: first, the mathematical framework is introduced through semi-analytical expressions for birth-death processes, followed by the general computational approach for rational networks using linear programming. 
In Section~\ref{sec:numexamples}, we apply the methods to three additional examples: 
a toggle switch, a model of bursty gene expression with negative feedback, and a model with multiple stationary solutions. 
We conclude with a discussion in Section~\ref{sec:discussion}.  For completeness, Appendix~\ref{appendix1} presents theoretical results linking stationary distributions of continuous-time chains and the stationary solutions of CMEs, and Appendix~\ref{appendix2} presents a Foster-Lyapunov criterion that guarantees existence and finiteness of moments of the stationary distribution for the examples in the paper.
\section{Notation and definitions}
\label{sec:stationary}

Stochastic biochemical kinetics under well-mixed conditions are usually described by a set of $m$ reactions $R_j$ involving $n$ species $S_1,S_2,...,S_n$:
\begin{align}
\label{eq:network}
 R_j:\enskip v_{1j}^-S_1+\dots+v_{nj}^-S_n \xrightarrow{a_j} v_{1j}^+S_1+\dots+v_{nj}^+S_n \\ 
 \quad j=1,\ldots,m,  \nonumber
\end{align}
where $v_{ij}^{\pm}\in\n$ denote the stoichiometric coefficients, and $a_j:\nn\to[0,\infty)$ is the propensity of reaction $R_j$. 

Formally, the state of the system is described by the %time-varying 
random variable $X(t)=(X_1(t),\dots,X_n(t)) \in \nn$, a vector with components representing the number of molecules of each species at time $t$. 
The dynamical process is modelled with a minimal continuous-time Markov chain\cite{Norris1997} with rate matrix $Q=(q(x,y))$ defined by:
\begin{equation}
\label{eq:qmatrixsrn}
q(x,y):=\sum_{j=1}^ma_j(x)(1_{x+v_j}(y)-1_x(y)),
\end{equation}
where $v_j:=(v_{1j}^+-v_{1j}^-,\dots,v_{nj}^+-v_{nj}^-)$ denotes the stoichiometric vector containing the net changes in molecule numbers produced by reaction $R_j$, 
and $1_y$ denotes the indicator function of state $y$:
\begin{align}
\label{eq:indicator}   
%\{f_x\}_{x \in \s_r}, \enskip \text{where} \enskip
%f_x (y) := 
1_{y}(x) :=
    \begin{cases*}
      1 & if $x=y$ \\
      0 & otherwise.
    \end{cases*}
\end{align}

The state of the system takes values in a subset $\s \subseteq \n^n$, known as the \emph{state space}, 
with (possibly infinite) cardinality $\mmag{\s}$. The set $\s$ must be chosen such that 
\begin{align}
\label{eq:qmatrix}
%\begin{array} {c}
 q(x,y) & \geq0 \enskip & \forall x\neq y,\\ 
 -q(x,x) &= \sum_{y \in \s, \, 
y \neq x} q(x,y)<\infty\enskip &\forall x\in\cal{S},
\label{eq:qmatrix2}
%\end{array}
\end{align}
in which case $Q$ is said to be~\emph{totally stable}~and~\emph{conservative}.

If the Markov chain cannot leave the state space in finite time, the matrix $Q$ is said to be \emph{regular} (see Appendix~\ref{appendix1}). In this case, the collection of probabilities $p_t(x)$ of observing the chain in state $x$ at time $t\geq0$ is the only solution of the \emph{chemical master equation} (\emph{CME})
\begin{align}\label{eq:CME}
\frac{\mathrm{d}p_t(x)}{\mathrm{d}t} = p_tQ(x),\quad p_0(x)=\lambda(x)\qquad\forall x\in\s,
\end{align}
where we define the vector $p_t:=(p_t(x))_{x\in\s}$ and
\begin{align} \label{eq:innerprod}
p_tQ(x):=\sum_{y\in\s}p_t(y)q(y,x), 
\end{align}
Following standard convention, probability distributions and measures are defined throughout as row vectors.

Any probability distribution $\pi:=(\pi(x))_{x\in\s}$ %on $\s$ 
that solves the equation
\begin{equation}\label{eq:stat}\pi Q(x)=0\qquad\forall x\in\s,\end{equation}
is called a \emph{stationary solution} of the CME. By definition, each stationary solution belongs to the space 
$$\ell^1:=\left\{(\rho(x))_{x\in\s}:\sum_{x\in\s}\mmag{\rho(x)}<\infty\right\}$$
of absolutely summable real sequences indexed by states in $\s$. The set of all stationary solutions forms a convex polytope in $\ell^1$:
\begin{align}
\label{eq:cmestat}
\cal{P}:=\left\{\pi\in\ell^1 : 
\begin{array}{l}
   \pi \, Q(x)=0,\, \, \forall x\in\s, \\
 \pi(\s) :=\sum_{x\in\s}\pi(x) = 1, \\ \pi\geq0 
 \end{array}
 \right\},
 \end{align}
where
$\pi\geq0$ is shorthand for $\pi(x)\geq0, \forall x\in\s$. 
For most networks of interest 
the stationary solutions determine the long-term behaviour of the chain 
(see Appendix~\ref{appendix1}).
In many cases, we will be interested in obtaining the $\pi$-average of a real-valued function $f$ on $\s$:
$$\ave{f}:=\sum_{x\in\s}f(x)\pi(x).$$
For example, the $k$-th stationary moment is $\ave{x^k}$.

\section{Bounding the stationary moments of the CME}
\label{bmom}

As a first use of optimisation techniques, we present a systematic approach that yields bounds of increasing tightness on the stationary moments of reaction networks
with polynomial or rational propensities,
and we give rigorous sufficient conditions for its validity. 
The moment bounds obtained in this section will be used in conjunction with linear programming to bound the full stationary distributions of the CME in Sec.~\ref{bdist}.

To motivate our optimisation approach, we first present the mathematical formulation through a simple example, for which explicit analytical expressions can be obtained (Sec.~\ref{sec:schloegl_moments}).
For more complex systems, the approach can be implemented computationally in a systematic manner through  
a general semidefinite programming method (Sec.~\ref{sec:mommulti}). 
Readers interested in the computational approach (and not the theory behind it) should skip Sec.~\ref{sec:schloegl_moments} and go directly to Sec.~\ref{sec:mommulti}.
\begin{figure}
	\begin{center}
	\includegraphics[width=0.423\textwidth]{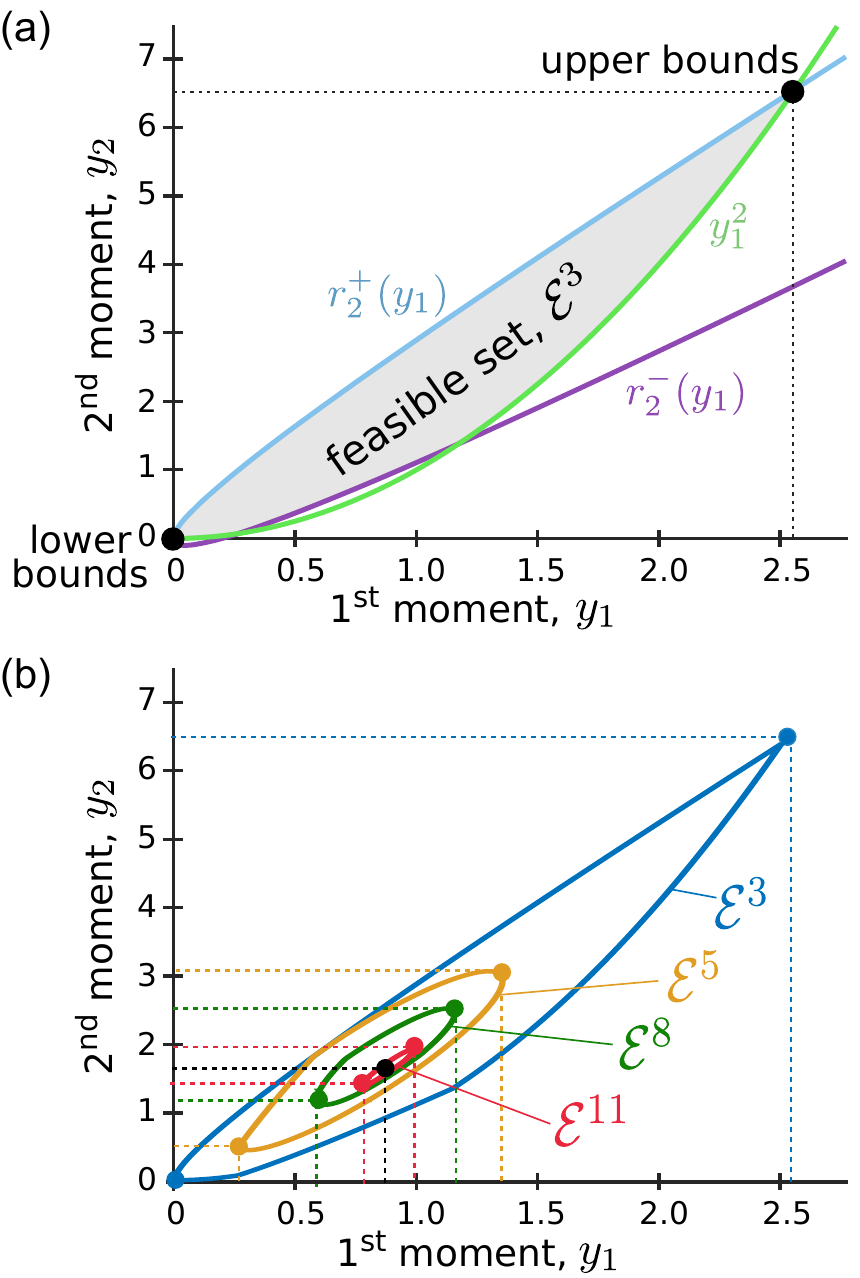}
	\vspace{-20pt}
	\end{center}
\caption{\textbf{Outer approximations and bounds for the moments of the stationary solution of Schl\"ogl's model~\eqref{eq:smdl}. } 
\textbf{(a)} Grey area: projection on the $y_1$--$y_2$ plane of $\cal{E}^3$~\eqref{eq:E3_bounds}, an outer approximation of the set of stationary moment vectors. Black dots mark the upper and lower bounds on the first and second moments~\eqref{eq:e3finboun}.
\textbf{(b)} By appending further moment equations and inequalities, the outer approximations $\cal{E}^d$~\eqref{eq:spectrahedron} can be tightened systematically as we increase the order $d$. The boundaries of the sets $\cal{E}^d$ (lines in different colours) were computed explicitly by applying Mathematica's Reduce function to~\eqref{eq:spectrahedron}.
The singleton set of stationary moment vectors (black dot) is always contained in $\cal{E}^d$. The increasingly tighter lower and upper bounds on the moments~\eqref{eq:powermoments2} (coloured dots) are computed by solving the SDPs~\eqref{eq:sdp1}--\eqref{eq:sdp2}.  Parameter values: $k_1=1, k_2=1, k_3=0.8, k_4=1$.}
\label{fig:set}
\end{figure}
\subsection{A simple analytical example: moment bounds for Schl\"ogl's model
} 
\label{sec:schloegl_moments}
To illustrate the mathematical framework, consider the classic autocatalytic network with a single species $S$ proposed by Schl\"ogl~\cite{schlogl1972} as a model for a chemical bifurcation:
\begin{equation}
\label{eq:smdl}
2S\underset{a_2}{\overset{a_1}{\rightleftarrows}}3S,
\qquad \varnothing \underset{a_4}{\overset{a_3}{\rightleftarrows}}S, 
\end{equation}
with mass-action propensities:
\begin{align}
& a_1(x):=k_1 x (x-1),
\quad a_2(x):= k_2 x (x-1) (x-2) \label{eq:schloegl_props1}\\ 
& a_3(x):= k_3, 
\qquad  a_4(x):=k_4x,
\label{eq:schloegl_props2}
\end{align}
where $k_1, k_2, k_3, k_4>0$ are rate constants.  
Here, $n=1$ and the state space is $\s = \mathbb{N}$. 

The CME of this network has a unique stationary solution $\pi$ and all of its moments are finite (see~Appendix~\ref{appendix2}). 
The reaction network, as encoded in the rate matrix $Q$, imposes certain relationships between the stationary moments $\ave{x^k}$. Such relations form an infinite system of coupled \textit{stationary moment equations}:
\begin{equation}
\label{eq:moment_hierarchy}
\ave{Q x^k }:= \sum_{x'\in\s} \sum_{x\in\s}q(x',x) \, x^k  \, \pi(x')= 0, \quad k \in \n.
\end{equation}
Except in particular instances, it is not possible to solve this coupled system exactly, and closure approximations are usually adopted by neglecting higher order moments. 

An alternative approach is that of mathematical programming, which includes additional constraints in the form of inequalities that must be fulfilled by the moments of distributions. Including such inequalities allows us to obtain feasible regions for the solutions of the system, and hence rigorous bounds for the moments. Increasing the number of inequalities considered, restricts the feasible region further and makes the bounds tighter. 

Let us consider the first moment equation for~\eqref{eq:smdl}: 
%the Schl\"ogl network: 
\begin{align}
\label{eq:schmom} 
%&b_1z_0-b_2z_1+b_3z_2-b_4z_3 = 0, 
&\ave{Q x } = b_1 \ave{1} -b_2 \ave{x} +b_3 \ave{x^2} -b_4 \ave{x^3} = 0, 
\end{align}
where  $b_1:=k_3$, $b_2:=k_1+2k_2+k_4$, $b_3:=k_1+3k_2$, $b_4:=k_2$ are positive numbers. 
Even after noting that  
\begin{align}
\label{eq:pos_moments}
%z_0=
\ave{1}=1,
\end{align}
as $\pi$ is a probability distribution, Eq.~\eqref{eq:schmom} is underdetermined. 
However, further relationships between moments can be added to constrain the system. For example, the non-negativity of the variance implies the inequality:
\begin{align}
\label{eq:variance}
\ave{x^2}-\ave{x}^2 \geq 0.    
\end{align}
Further inequalities involving higher moments can be built systematically from polynomial functions as follows.

Consider the polynomial $f(x):=f_0+f_1x$, where $\mathbf{f}:=(f_0, f_1)^T \in\r^2$ is  
a (column) vector of polynomial coefficients. 
Clearly, $f^2 (x)$ and $x \, f^2 (x)$ are  non-negative on $x \in [0,\infty)$. 
Hence it follows that
\begin{align}
\ave{f^2}&=f_0^2\ave{1}+2f_0f_1 \ave{x} +f_1^2\ave{x^2} \geq 0 
\label{eq:first_inequal}\\
\ave{x f^2}&=f_0^2\ave{x}+2f_0f_1\ave{x^2}+f_1^2\ave{x^3} \geq 0 
\label{eq:second_inequal}
\end{align}
Let us define the vector of moments 
$$z:=\begin{pmatrix} \ave{1}, \ave{x}, \ave{x^2}, \ave{x^3} 
\end{pmatrix}^T \in \r^4.$$ 
Then the inequalities~ 
\eqref{eq:first_inequal}--\eqref{eq:second_inequal} are written compactly~as
\begin{align}
\ave{f^2} %&=g_0^2\ave{1}+2g_0g_1 \ave{x} +g_1^2\ave{x^2} \nonumber\\
&=\mathbf{f}^T M_3^0(z)\mathbf{f} \, \mathbf{}\geq 0,\label{eq:nfdyuhsanshua1}\\
\ave{x f^2} %&=g_0^2\ave{x}+2g_0g_1\ave{x^2}+g_1^2\ave{x^3}\nonumber\\
&=\mathbf{f}^T M_3^1(z) \, \mathbf{f}\geq 0,\label{eq:nfdyuhsanshua2}
\end{align}
where  the matrices $M^3_0(y)$ and $M^3_1(y)$ are defined by
$$M_3^0(y)  := \begin{bmatrix}y_0 &y_1 \\ y_1 &y_2\end{bmatrix},\qquad 
M_3^1(y) := \begin{bmatrix}y_1 &y_2 \\y_2 &y_3\end{bmatrix},$$
for any four-dimensional vector $y=(y_0,y_1,y_2,y_3)^T\in\r^4$. 
Since \eqref{eq:nfdyuhsanshua1}--\eqref{eq:nfdyuhsanshua2} hold for all $\mathbf{f}\in\r^2$, we have that $M^3_{0}(z)$ and $M^3_{1}(z)$ are positive semidefinite (p.s.d.):
\begin{align}
\label{eq:schmmm}
M_3^0(z)\succeq0,\qquad M_3^1(z)\succeq0.
\end{align}
From Sylvester's criterion, \eqref{eq:schmmm} is equivalent to
\begin{align}
z_0\geq0,\quad z_1\geq0, &\quad z_2\geq0,\quad z_3\geq0,\quad \label{eq:schoin1}\\
z_0z_2-z_1^2 &\geq0,  \label{eq:schoin2}  \\
z_1z_3-z_2^2 &\geq0. \label{eq:schoin3} 
\end{align}
Hence 
\eqref{eq:schoin2} recovers~\eqref{eq:variance}, 
whereas~\eqref{eq:schoin3} gives an additional condition involving the first three moments. 

Putting \eqref{eq:schmom}--\eqref{eq:schmmm} together, we conclude that 
the singleton set $\{z\}$ of vectors  whose entries are composed of the first four moments of the stationary solutions of Schl\"ogl's model
belongs to the set
\begin{align}
\cal{E}^3=\left\{y\in\r^4:
\begin{array}{l}y_0=1 \\ 
%y_1\geq0, \quad y_2 \geq0,\quad y_3\geq0, \\
b_1y_0-b_2y_1+b_3y_2-b_4y_3 = 0 \\
M_3^0(y) \succeq 0 \\ 
M_3^1(y) \succeq 0
\end{array}\right\}.
\label{eq:E3}
\end{align}
We say that $\cal{E}^3$ is an \emph{outer approximation} of the set of stationary moment vectors. Hence the stationary moments have the following lower and upper bounds:
\begin{align}
L^3_\alpha:=\inf\{y_\alpha:y\in\cal{E}^3\}  & \le \ave{x^\alpha} \le \sup\{y_\alpha:y\in\cal{E}^3\}=:U^3_\alpha \nonumber \\
& \alpha = 0,1,2,3.
 \label{eq:schloegl_moment_bounds}
 \end{align}
 
Such bounds are usually handled computationally, but it is illustrative to obtain explicit expressions in this simple case. 
Combine~\eqref{eq:schoin3} with the equalities in~\eqref{eq:E3} to get 
$$b_1y_1+b_3y_1y_2-b_4y_2^2-b_2y_1^2\geq0.$$
Assuming $b_3\geq 2\sqrt{b_2b_4}$ (the other case is analogous), 
we use the quadratic formula 
to obtain the bound
\begin{align}
& r_2^-(y_1)  \leq y_2  \leq r_2^+(y_1) \\
%with
%
r_2^\pm(x) :=&\frac{b_3x  \pm\sqrt{4b_1b_4x+(b_3^2-4b_2b_4)x^2}}{2b_4}.
\end{align}
Hence $\cal{E}^3$ can be rewritten equivalently as:%in \eqref{eq:E3} as:
\begin{align}
\cal{E}^3=\left\{y\in\r^4:
\begin{array}{l}y_0=1 \\
y_1\geq0 \\
\max\{y_1^2,r_2^-(y_1)\}\leq y_2\leq r_2^+(y_1)\\
y_3=(b_1-b_2y_1+b_3y_2)/b_4
\end{array}\right\}.
\label{eq:E3_bounds}
\end{align}
Figure~\ref{fig:set}(a) shows the projection of  $\cal{E}^3$ onto the $y_1$-$y_2$ plane.
From~\eqref{eq:E3_bounds}, it is clear that $(1,0,0,b_1/b_4) \in \cal{E}^3$, 
and the lower bounds in~\eqref{eq:schloegl_moment_bounds} for the first two moments are trivial: $(L^3_1,L^3_2)=(0,0)$. 
The upper bounds, however, are not. 
Since $r^+_2(x) > r^-_2(x)$ for all $x>0$, 
the upper bounds in~\eqref{eq:schloegl_moment_bounds} 
are obtained by the northeasternmost intersection of $y_1^2$ and $r^+_2(y_1)$:
$(U^3_1,U^3_2)=(r_4,r_2^+(r_4))$, where $r_4$ is the rightmost root of 
$x(b_1-b_2x+b_3x^2-b_4x^3)=0$. In summary, we get:
\begin{equation}\label{eq:e3finboun}0\leq \ave{x}\leq r_4,\qquad 0\leq \ave{x^2}\leq r^+_2(r_4).\end{equation}
As seen in Fig.~\ref{fig:set}(a), the analytical bounds based on $\cal{E}^3$ are rough. However, we show in the following section how to obtain tighter bounds systematically by appending further moment equations and inequalities and solving the associated optimisations over higher order sets~$\cal{E}^d$.

\subsection{The general approach: Bounding the moments of rational CMEs by solving semidefinite programs}
\label{sec:mommulti}

The approach in the previous section can be applied to any reaction network~\eqref{eq:network} with $n$ species and state space $\s$, as long as the propensities 
of its $m$ reactions 
are rational (or polynomial) functions, i.e., they can be rewritten as
\begin{equation}
\label{eq:numden}
a_j(x) := \dfrac{b_j(x)}{s(x)}\qquad j=1,\dots,m,
\end{equation}
where $b_1,\dots,b_m$ and $s$ are polynomials on $\rn$ and the common denominator $s$ satisfies $s(x)>0, \forall x\in\s$.

To deal with multiple species, we use standard multi-index notation:
$x^\alpha :=x_1^{\alpha_1}x^{\alpha_2}\dots x^{\alpha_n}_n$, where $\alpha$ is the multi-index $(\alpha_1,\alpha_2,\dots,\alpha_n)\in\nn$ and $\mmag{\alpha}:= \alpha_1 + \alpha_2+\dots+\alpha_n$ is the degree of the monomial. Polynomial functions are expressed in terms of such monomials. For instance, the denominator of~\eqref{eq:numden} is: 
\begin{align}
\label{eq:denominator}
s(x) = &\sum_{\mmag{\beta}\leq d_s}s_\beta \, x^\beta, 
%\\
%\mathbf{s}:= &(s_\beta)_{\mmag{\beta}\leq d_s} \nonumber
%$ (using the convention that $mathbf{s}_\beta=0$ if $\mmag{\beta}>d_s$). 
\end{align}
where $d_s$ is the degree of $s$ and we define the (column) vector of coefficients
\begin{align}
\label{eq:coeff_s}
    \mathbf{s}:= (s_\beta)_{\mmag{\beta}\leq d_s}.
\end{align}
We also define the stationary rational moments
\begin{equation}
\label{eq:ch7vmom}
z_\beta:=\ave{\dfrac{x^\beta}{s}}.
\end{equation}
which are directly related to the raw moments: 
\begin{equation}
\label{eq:moment_expression}
\ave{x^\alpha}=%\sum_{\mmag{\beta}\leq d}\mathbf{s}_\beta \ave{\frac{x^{\alpha+\beta}}{s}}=
\sum_{\mmag{\beta}\leq d_s}s_\beta \, z_{\alpha+\beta}.
\end{equation}
The following checkable assumption is a \emph{sufficient condition} for our general SDP approach to apply to generic reaction networks with rational propensities.

\begin{assumption}[Order of the approximation and finiteness of moments]
\label{ass:1}
Recall that $d_s$ is the degree of the denominator~$s$ in \eqref{eq:numden}.
Let us denote the order of the approximation by an integer $d \geq d_s$, 
and let us compile the stationary rational moments~\eqref{eq:ch7vmom} up to order $d$ into the (column) vector
\begin{equation}
\label{eq:vector_z}
z:=(z_\beta)_{\mmag{\beta}\leq d}
\end{equation}
of dimension $\#_d:=\binom{n+d}{n}$.

We assume that all stationary solutions of the CME have finite rational moments up to order $d+1$, %where $d \geq d_s$, 
i.e.,  
$$z_\beta=\ave{\dfrac{x^\beta}{s}}<\infty\quad\forall\beta:\mmag{\beta}\leq d+1, \quad\forall\pi\in\cal{P}.$$
This requirement can be verified using a Foster-Lyapunov criterion as detailed in Appendix~\ref{appendix2}. 
\end{assumption}

\begin{remark}
 From~\eqref{eq:moment_expression} and Assumption~\ref{ass:1} it follows that raw moments $\ave{x^\alpha}$ with
$\mmag{\alpha}\leq d+1-d_s$ 
 are also finite.
\end{remark}

In order to write down each $\alpha$-moment, it is helpful to define the associated polynomial function $g_\alpha(x)$:
\begin{align}
s(x) \,Q x^\alpha 
&=\sum_{j=1}^m b_j(x) \,((x+v_j)^\alpha-x^\alpha) \nonumber  \\
%\qquad := g^{(\alpha)}(x)
&=\sum_{\mmag{\beta}\leq d_{g_\alpha}} {(g_\alpha)}_\beta \, x^\beta  =: g_\alpha(x), 
\label{eq:g} 
\end{align}
with degree $d_{g_\alpha} = \mmag{\alpha}+d_b -1$, where
$d_b:=\max\{d_{b_i}\}$~is~the maximum degree of the numerators %$b_i$ 
in~\eqref{eq:numden}. We also define the (column) vector of polynomial coefficients
\begin{equation}
\label{eq:g_vector}
 \mathbf{g}_\alpha  :=\left({(g_\alpha)}_\beta\right)_{\mmag{\beta}\leq d},
\end{equation}
where ${(g_\alpha)}_\beta = 0$ if $\mmag{\beta}>d_{g_\alpha}$. 

The finiteness of moments guarantees that a subset of moment equations will hold, as stated in the following lemma.

\begin{lemma}[The moment equations]\label{lem:momeqs}If Assumption~\ref{ass:1} is satisfied and $\pi \in \cal{P}$, then the $\alpha$-moment equation
\begin{equation}
\label{eq:momeqs}
\ave{Qx^\alpha}=
z^T\mathbf{g}_\alpha = 0 %\qquad\forall \alpha:\mmag{\alpha}\leq d-d_b+1.
\end{equation}
holds for every $\alpha\in\nn$ such that $\mmag{\alpha}\leq d-d_b+1$.
\end{lemma}

\begin{proof}Consider the adjoint
of \eqref{eq:stat}: 
\begin{equation}
\label{eq:steq}
\ave{Q f} %= \sum_{x\in\s}(Qf)(x)\pi(x):
= \sum_{x\in\s} \sum_{y\in\s}q(x,y) f(y) \, \pi(x)= 0,
\end{equation}
which, by Fubini's theorem, 
is valid for any $f$ as long as $\sum_{x\in\s}\mmag{q(x,x)f(x)}\pi(x)$
is finite~\cite{Glynn2008}. Because
$$
%s(x)q(x)\mmag{f(x)}=
s(x)\mmag{q(x,x)x^\alpha}=
\sum_{j=1}^mb_j(x)x^\alpha=
\sum_{j=1}^m\sum_{\mmag{\beta}\leq d_{b}}{(b_j)}_\beta \, x^{\alpha+\beta},$$
then Assumption~\ref{ass:1} implies
$$
\sum_{x\in\s}\mmag{q(x,x)x^\alpha}\pi(x)  \leq \sum_{\mmag{\beta}\leq d_b}\left(\sum_{j=1}^m\mmag{(b_j)_{\beta}}\right)z_{\beta+\alpha}<\infty,$$
for all $\mmag{\alpha}\leq d-d_b+1$. Setting $f(x):=x^\alpha$ in~\eqref{eq:steq}, we~get:
\[
0=\ave{Qx^\alpha}=\ave{\frac{g_\alpha}{s}}= 
\sum_{\mmag{\beta}\leq d} {(g_\alpha)}_\beta  \ave{\frac{x^\beta}{s}}
=z^T\mathbf{g}_\alpha.
\]
\end{proof}
In addition to the  moment equations~\eqref{eq:momeqs}, 
the moments $z$ satisfy additional constraints. 
Firstly, since $\pi$ is a probability distribution, we have:
\begin{equation}\label{eq:normalisation_cond}z^T\mathbf{s}=\sum_{\mmag{\beta}\leq d}s_\beta \ave{\frac{x^\beta}{s}}=\ave{\frac{s}{s}}=\ave{1}=1.\end{equation}
Furthermore, the rational moments 
satisfy well-known semidefinite inequalities\cite{Lasserre2009,Blekherman2013,Kuntzthe}. 
Specifically, the localising matrices are positive semidefinite:
\begin{align}
 \label{eq:moment_matrices_posdef}
 %M^0_d(y) \succeq 0  \quad 
 M^i_d(z) \succeq 0\qquad\forall i=0,\dots,n,
 \end{align}
where the 
$M^i_d(y)$
are defined by
\begin{align*}
 %\label{eq:sdpin}
 [M^0_d(y)]_{\alpha\beta}&:=y_{\alpha+\beta},     & \forall \alpha,\beta:\mmag{\alpha},\mmag{\beta}\leq \lfloor d/2\rfloor,\notag\\
 [M^i_d(y)]_{\alpha\beta}&:=y_{\alpha+\beta+e_i}, & \forall \alpha,\beta:\mmag{\alpha},\mmag{\beta}\leq \lfloor(d-1)/2\rfloor,
\end{align*}
with $e_i$ denoting the $i^{th}$ unit vector 
and $y \in \rhd$.

The inequalities~\eqref{eq:moment_matrices_posdef}
follow from the fact that for any polynomial function $f(x)$ of degree 
$\lfloor d/2\rfloor$ with (column) vector of coefficients $\mathbf{f}=(f_\beta)_{\mmag{\beta}\leq \lfloor d/2\rfloor}$, we have
\begin{align}
&\mathbf{f}^T M^0_d(z) \mathbf{f}=\sum_{\mmag{\alpha}\leq\lfloor d/2\rfloor} \sum_{\mmag{\beta}\leq\lfloor d/2\rfloor} f_\alpha f_\beta z_{\alpha+\beta}\nonumber\\
&=\ave{\frac{\left(\sum_{\mg{\alpha}\leq\lfloor d/2\rfloor}f_\alpha x^{\alpha}\right)\left(
\sum_{\mmag{\beta}\leq\lfloor d/2\rfloor}f_\beta x^{\beta}  \right)}{s}} \nonumber\\
&=\ave{\frac{f^2}{s}}\geq0.\label{eq:semial1}
\end{align}
%since $s(x)>0, \forall x\in\s$. 
Similarly, it can be shown\cite{Lasserre2009,Blekherman2013} that
\begin{equation}
\label{eq:semial2}
\mathbf{f}^T M^i_d(z) \mathbf{f}=\ave{\frac{x_i \, f^2}{s}} \geq 0
\quad\forall i=1,\dots,n,
\end{equation}
Since~\eqref{eq:semial1}--\eqref{eq:semial2} hold for any vector $\mathbf{f}$, the matrices $M^i_d(z)$ are positive semidefinite.

\begin{remark}
In the case of Schl\"ogl's model~\eqref{eq:smdl}, 
we had $s(x)=1$, $d_b=3$ and $\#_d=d+1$. Eq.~\eqref{eq:schmom} is the moment equation~\eqref{eq:momeqs} with $\alpha=1$, and the matrices in~\eqref{eq:schmmm} with $d=3$ are $M^0_3(y)$ and $M^1_3(y)$. 
\end{remark}

\subsubsection{
Bounding the moments} 

We can then  establish the following lemma regarding outer approximations of the set of rational moments.

\begin{lemma}[Outer approximations of the set of rational moments]
\label{inclu}
If Assumption~\ref{ass:1} is satisfied and $\pi\in\cal{P}$, 
the vector of rational moments $z$
belongs to the spectrahedron
\begin{align}
 \label{eq:spectrahedron}
 \cal{E}^d:=\left\{y\in \rhd:
 \begin{array}{l} 
 y^T \mathbf{g}_\alpha=0\quad\forall \mmag{\alpha}\leq d-d_b+1, \\ 
  y^T\mathbf{s}=1, \\
 M_d^i(y)\succeq0\quad \forall i=0,1,\dots,n.
 \end{array}
 \right\},
\end{align}
where $d_b$ is the maximum degree of the numerators in~\eqref{eq:numden}, and $\mathbf{s}$, $\mathbf{g_\alpha}$ and $M_d^i(y)$ are defined in \eqref{eq:denominator},\eqref{eq:g_vector}~and~\eqref{eq:moment_matrices_posdef}, respectively.
\end{lemma}
\begin{proof}The proof follows from Lemma~\ref{lem:momeqs}, \eqref{eq:normalisation_cond}--\eqref{eq:moment_matrices_posdef}.
\end{proof}
 In summary, the vectors of stationary moments of order $d$ are contained in a feasible set $\cal{E}^d$, defined by linear equalities and inequalities, which constitutes an \emph{outer approximation} to the set of moment vectors. 
 There are several implications of this lemma.

Firstly, the outer approximation property implies that extremal points of $\cal{E}^d$  provide bounds on the stationary moments. Specifically, the vector in $\cal{E}^d$ with largest (resp.~smallest) $\alpha$-entry provides an upper (resp.~lower) bound on the $\alpha$-moment. For example, Fig.~\ref{fig:set}(b) shows the projection of $\mathcal{E}^d$ with increasing $d$ onto the $y_1$--$y_2$ plane for Schl\"ogl's model. The northeasternmost (resp.~soutwesternmost) vector of these outer approximations yield upper (resp.~lower) bounds on the first two moments of Schl\"ogl's model.

Secondly, note that the moment matrix $M_d^i(y)$ is a principal submatrix of $M_{d+1}^i(y)$. Since a matrix is p.s.d. if and only if all of its principal submatrices are p.s.d. and $\cal{E}^{d+1}$ includes all moment equations in $\cal{E}^d$, then it follows that every vector in $\cal{E}^{d+1}$ (appropriately truncated) belongs to $\cal{E}^d$. As a result, the outer approximations \emph{tighten} around the set of stationary moment vectors with bounds of increasing quality as the order of the approximation $d$ is increased, as seen inFig.~\ref{fig:set}(b).
These two observations are summarised in the following corollary.

\begin{corollary}[Monotonic moment bounds]
\label{boundcor}
Suppose that Assumption~\ref{ass:1} is satisfied and $\pi\in\cal{P}$.
If $f$ is a polynomial of degree $d_f \leq d$, then
\begin{small}
\begin{align}
\label{monotone1}
L^{d_f}_f \leq L^{d_f+1}_f \leq \dots \leq L^{d}_f  \leq\ave{\frac{f}{s}}
  \leq U^{d}_f  \leq \dots  \leq U^{d_f+1}_f \leq U^{d_f}_f  \
 %\label{monotone2}
 \end{align}
   \end{small}
\begin{align}
\label{eq:sdp1}
%$L^d_f$ and $U^d_f$ are as in \eqref{eq:sdps}.
 \text{where } \qquad L^d_f  &: =\inf\{\mathbf{f}^Ty:y\in\cal{E}^d\} \\ %\enskip\text{and}\enskip 
  \quad U^d_f  &: =\sup\{\mathbf{f}^Ty:y\in\cal{E}^d\}.
 \label{eq:sdp2}
\end{align}
\end{corollary}
\begin{proof}Note that $z \in \cal{E}^d$ because
\[
\ave{\frac{f}{s}}=\ave{\frac{\sum_{\mmag{\beta}\leq d}f_\beta x^\beta}{s}}=\sum_{\mmag{\beta}\leq d}f_\beta\ave{\frac{x^\beta}{s}}=z^T \mathbf{f},
\]
Hence \eqref{monotone1} follows from~\eqref{eq:sdp1}--\eqref{eq:sdp2}. 
As explained in the main text, the monotonicity of the bounds
follows from the definition of $\cal{E}^d$ and the fact that a matrix is p.s.d. if and only if all of its principal submatrices are~p.s.d and $\cal{E}^{d+1}$ includes all moment equations in $\cal{E}^d$. 
\end{proof}

Applying these results to $\ave{x^\alpha}$, the $\alpha$-moment of the CME, is straightforward. Let $f(x):=s(x) x^\alpha$ and choose $d\ge|\alpha|+d_s$ to obtain the bounds
\begin{equation}
\label{eq:powermoments2}
L_\alpha^d:=L^d_f,\qquad U^d_\alpha:=U^d_f.
\end{equation} 
Corollary~\ref{boundcor} establishes that outer approximations $\cal{E}^d$ of increasing order can be used to compute a \emph{monotonically} increasing (resp. decreasing) sequence of lower (resp. upper) bounds for $\ave{x^\alpha}$:
$$L_\alpha^{\mmag{\alpha}+d_s} \leq \ldots \leq L_\alpha^{d} 
\leq \ave{x^\alpha} \leq U_\alpha^{d} \leq \ldots \leq U_\alpha^{\mmag{\alpha}+d_s}.$$

\begin{remark}[The sequence of moment bounds is monotonic but may not converge]
\label{rem:sdpboundconv}
The monotonicity of the bounds does not imply  
that the gap between the bounds $(U_\alpha^{d} - L_\alpha^{d})$ will converge to zero as $d \to \infty$. 
Although in our experience 
the bounds often converge numerically, 
there is no general guarantee for several reasons.
Firstly, the stationary solution may not be unique and, in that case, the lower bounds are limited by the stationary solution with the smallest moment while the upper bounds are limited by that with the largest moment. 
Even if the solution is unique, the bounds may not converge because the semidefinite conditions are tailored to distributions with support on the non-negative real orthant but not to distributions with support on discrete state spaces~\cite{Dowdy2018}, for which more stringent conditions can be produced at a higher computational cost~\cite{Lasserre2002a,Lasserre2009,Dowdy2018}.
\end{remark}

\subsubsection{Computing moment bounds via semidefinite programming}\label{sec:sdpcomp}
Given a reaction network with rational propensities and a polynomial $f$ of degree $d_f$, the moment bounds~\eqref{eq:sdp1}--\eqref{eq:sdp2} are the extreme points of the linear functional $y\mapsto \mathbf{f}^Ty$ over the set $\cal{E}^d$, which is defined by linear equalities and semidefinite inequalities. Hence, computing the bounds amounts to solving a \textit{semidefinite program} (SDP), a convex optimisation problem for which there exist efficient computational tools. 
Therefore, instead of \textit{ad hoc} analytical manipulations, like those leading to~\eqref{eq:e3finboun},
a general procedure by constructing and solving the SDPs systematically is implemented as follows: 
\begin{enumerate}
\item Rewrite the reaction propensities in the form~\eqref{eq:numden} removing all common factors and setting $s$ to be the lowest common denominator.
\item Verify the existence of stationary solutions $\pi$ and choose the order of the approximation $d$, an integer $d\geq d_f$ for which the $d+1$ stationary moments are  finite (Assumption~\ref{ass:1}) using a Foster-Lyapunov criterion (Theorem \ref{lyap} in Appendix~\ref{appendix2}).
\item Compute the bounds $L_f^d$ and $U_f^d$ by solving the two SDPs~\eqref{eq:sdp1}--\eqref{eq:sdp2}. We set up the SDPs using the modelling package YALMIP\cite{Lofberg2004}, and solve them using the multi-precision solver SDPA-GMP\cite{Nakata2010} with the interface mpYALMIP\cite{Fantuzzi2016}. 
Examples of computation times are given in the figure captions.

SDPs involving high order moments can be numerically 
ill-conditioned~\cite{Kuntz2016,Fantuzzi2016a,Dowdy2018}. 
Although the origin of this numerical instability remains an open problem, our computations suggest that it could be the result 
of the rapid growth of moments, which leads to ill-conditioned moment matrices.
Such disparity is problematic for standard double-precision SDP solvers but we have mitigated it with the multi-precision solver SDPA-GMP\cite{Nakata2010} as in Ref.~\cite{Fantuzzi2016a}. Alternatively, one can scale the moments \cite{Kuntz2016,Fantuzzi2016a,Dowdy2018}, or adapt recently developed specialised solvers~\cite{Papp2017,Zheng2018}.
\item Evaluate the error of the bounds by computing the gap $U^d_f-L^d_f$. If the gap is unsatisfactorily large, increase the order of the approximation and return to Step 2 to compute new bounds. 

Corollary~\ref{boundcor} guarantees that the bounds will not loosen as we increase $d$, yet the bounds may stagnate (Remark~\ref{rem:sdpboundconv}). In this case, we recommend breaking the impasse by employing the LP approach of the next section (see Fig.~\ref{fig:gene}(b)).

\end{enumerate}

As an example of this procedure, Fig.~\ref{fig:schomom}(a) shows how the computed upper and lower bounds $(L^d_\alpha, U^d_\alpha)$ for the first three stationary moments of Schl\"ogl's model become tighter as we increase $d$, the order of the approximation.  
Fig.~\ref{fig:schomom}(b) combines the moment bounds to obtain bounds on commonly used statistics, e.g., variance, coefficient of variation, and skewness.

\section{Bounding and approximating the stationary solutions of the CME}\label{bdist}
\begin{figure*}
	\begin{center}
	\includegraphics[width=0.9\textwidth]{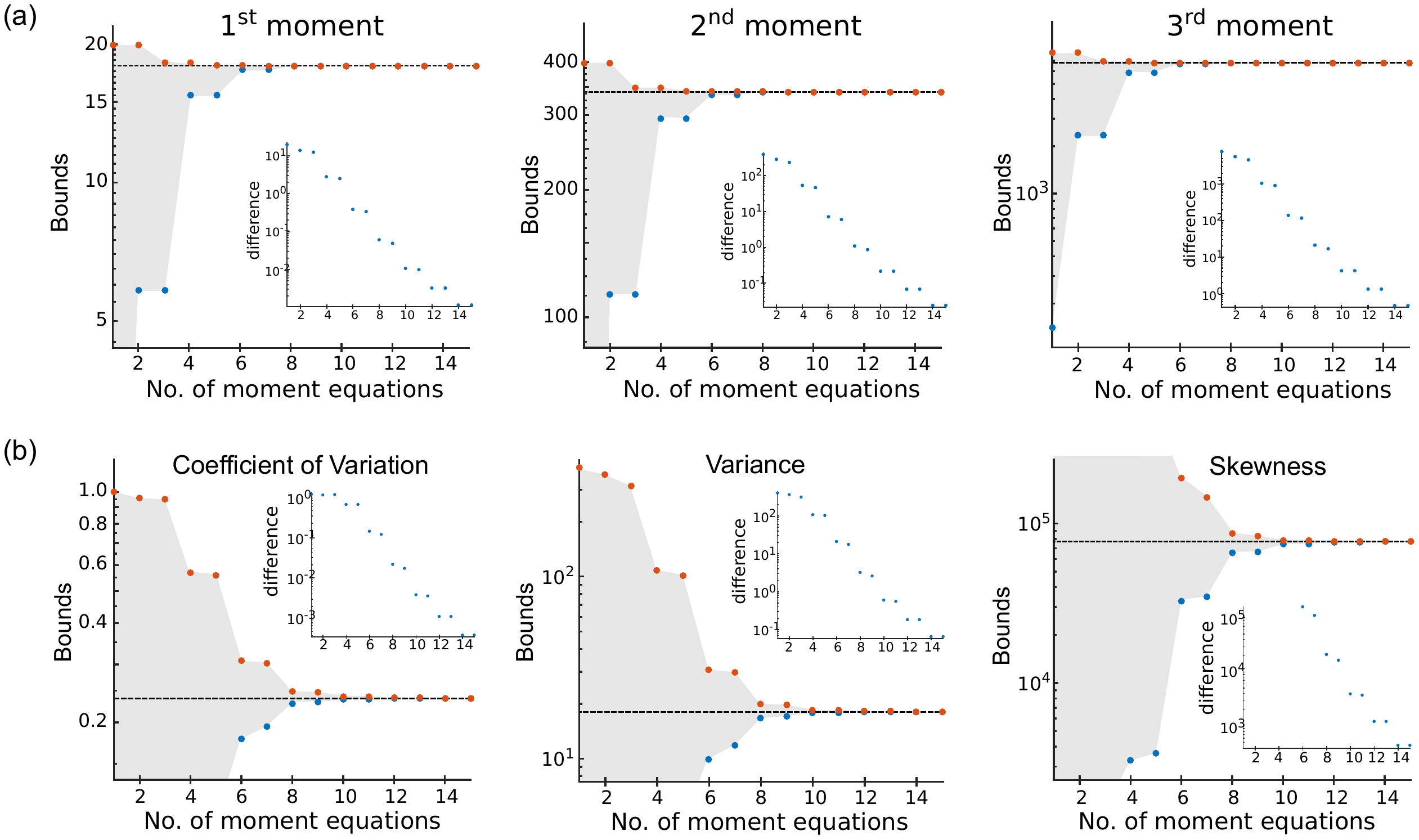}
	\vspace{-20pt}
	\end{center}
\caption{\textbf{Moment bounds for Schl\"ogl's model~\eqref{eq:smdl} using SDPs.} \textbf{(a)} Upper bounds ($U^d_\alpha$, red circles) and lower bounds ($L^d_\alpha$, blue circles) for the first three moments ($\ave{x^\alpha}, \, \alpha = 1, 2, 3$) 
computed for increasing order of the approximation, $d$
(No. of moment equations$= d-2$). The bounds for the first two moments computed by solving the SDPs with $d=3$ coincide with the analytical expressions~\eqref{eq:e3finboun}.
The bounds approach the true moments (dashed lines).
Inset: the gap between upper and lower bounds decreases to zero as $d$ increases. 
\textbf{(b)} The moment bounds in (a) are used to obtain bounds of three typical statistics: coefficient of variation, variance, and skewness. Insets: the gap between lower and upper bounds also decreases to zero. 
For each moment, we computed $30$ bounds (upper and lower, $d=3,\ldots,17$) for a solver time of $29$ seconds ($1$ second per bound).
Parameters: $k_1=6$, $k_2=1/3$, $k_3=50$, $k_4=3$, which correspond to a unimodal stationary solution.
}   
\label{fig:schomom}
\end{figure*}
The semidefinite programming scheme in the previous section allows  us to compute bounds of the stationary moments of the CME by constructing and optimising over outer approximations of the set of stationary moments. In this section, we go further and introduce a linear programming scheme that yields bounds on the \emph{full} stationary solutions of the CME, their marginals and their averages by constructing and optimising over outer approximations of the set of stationary solutions of the CME. To do so, we use a moment bound obtained in the previous section.
Note that 
the approximations in this section only involve linear inequalities (instead of semidefinite ones). Hence the optimisations to be solved are linear programs (instead of SDPs), a simpler subclass of convex optimisation problems for which mature, industrial solvers~\cite{CPLEX} are available. 

As for the SDP scheme above, we introduce the mathematical approach through a simple semi-analytic example (i.e., birth-death processes) in Sec.~\ref{sec:bd}, and then present the computational framework for general CMEs in Sec.~\ref{lpsec}. Readers interested in the computational implementation (and not the mathematical background) should skip Sec.~\ref{sec:bd} and go directly to Sec.~\ref{lpsec}. 
\subsection{A simple example: bounding the stationary solution of birth-death processes}\label{sec:bd}
Birth-death processes are one-species reaction networks ($n=1$)  
with state space $\s=\n$ whose value $x \in \n$ changes by $\pm1$ in each reaction:
\begin{equation}
\label{eq:one_step}
\varnothing \xrightarrow{a_+}S\xrightarrow{a_-}\varnothing. 
\end{equation}
The specific birth-death process is defined by the functional form of the given propensities $a_+(x)$ and $a_-(x)$.

The stationary equations $\pi Q=0$ then read
\begin{align}
\label{eq:introeqs1} 
&a_-(1) \pi(1)-a_+(0)\pi(0)=0,\\ &a_-(x+1)\pi(x+1)-(a_+(x)+a_-(x))\pi(x)\nonumber
\\&  \quad +a_+(x-1)\pi(x-1)=0,   \quad \qquad  x \geq 1. \label{eq:introeqs2}
\end{align}
%for all $x=1,2\dots$ 
Assuming non-vanishing death rates
\begin{equation}
\label{eq:drpos}
a_-(x)>0, \qquad  \forall x \geq 1,
\end{equation} 
it is well known that the unique stationary solution is:\cite{Gardiner2009}
\begin{align}
\label{eq:1deg}
&\pi(x) = \left [ \prod_{z=1}^x\frac{a_+(z-1)}{a_-(z)} \right ] \pi(0) =: \gamma(x)  \, \pi(0), \quad \forall x  \geq 0,
\end{align}
%for $x =1,2,\dots$. 
with $\pi(0)$ given by the normalisation condition: 
\begin{align}
 \label{eq:normalising}
&\pi(0) = \frac{1}{\sum_{x=0}^\infty \gamma(x)} =: \frac{1}{\gamma(\s)},
\end{align}
where we have introduced the notation for sums over sets:
\begin{align}
\label{eq:sum_notation}
\sum_{x \in \s} \gamma(x) =: \gamma(\s).     
\end{align}
Hence birth-death processes have at most one stationary solution, which exists if and only if $\gamma(\s)$ is finite.

\subsubsection{Semi-analytical approach for bounds and approximation}

For most birth-death processes, no closed-form expression for 
$\gamma(\s)$ is known and, consequently, the stationary solution cannot be computed exactly. However, we can obtain
upper and lower bounds for the distribution, as follows.

Let us consider a \emph{state space truncation}
\begin{equation}
\label{eq:momtrunc}
\s_r:=\{x\in\n:x^\alpha<r\}=\{0,1,\dots,\lceil r^{1/\alpha}\rceil-1\},
\end{equation}
with size $\mmag{\s_r}=\lceil r^{1/\alpha}\rceil$ controlled by the parameters $\alpha \in \zp$ and $r > 0$. 
Let $\s_r^c$ denote its complement, i.e., the set of states outside of the truncation $\s_r$.

Let us assume that we have available an upper bound on the stationary $\alpha$-moment:
\begin{align}
\label{eq:sdp_bound}
\ave{x^\alpha}\leq c.      
\end{align}
Note that for rational propensities, such a bound can be computed with the SDP scheme in~Sec.~\ref{bmom}.

\paragraph*{Upper bound:}
An easy upper bound on $\pi(0)$ is obtained by
truncating the sum in \eqref{eq:normalising} to get:
$$ \frac{1}{\sum_{x\in\s_r} \gamma(x)} \geq \frac{1}{\sum_{x=0}^\infty \gamma(x)} = \pi(0),$$ 
whence it follows that
\begin{align} 
% & \pi(x)\leq \frac{\gamma(x)}{\sum_{x\in\s_r} \gamma(x)}=:u^r_x \quad \forall x\in\s_r. 
& \pi(x)\leq \frac{\gamma(x)}{\gamma(\s_r)}=:u^r_x \quad \forall x\in\s_r.
  \label{eq:upperbounds}
\end{align}

\paragraph*{Lower bound:} %
Using~\eqref{eq:sdp_bound}, we obtain a bound on the probability mass $m_r$ outside of the truncation:
\begin{align}
 m_r :&= \sum_{x\not\in\s_r} \pi(x)\leq  \frac{1}{r}\sum_{x\not\in\s_r} x^\alpha\pi(x)\leq  \frac{\ave{x^\alpha}}{r}\leq \frac{c}{r}=:\varepsilon_r.
 \label{eq:Markov1}
\end{align}
We say that $\varepsilon_r$ is a \emph{tail bound}. 

A lower bound on $\pi(0)$ then follows from \eqref{eq:1deg}--\eqref{eq:Markov1}:
\begin{align}
\pi(0)=& \frac{1-m_r}{\sum_{x\in\s_r}\gamma(x)}\geq \frac{1- \varepsilon_r}
{\gamma(\s_r)}
%{\sum_{x\in\s_r}\gamma(x)},
\end{align}
whence we obtain a lower bound for $\pi(x)$:
\begin{align}
& \pi(x)\geq u^r_x (1-\varepsilon_r)=:l^r_x, \quad \forall x\in\s_r
\label{eq:lowerbounds}
\end{align}

\paragraph*{Convergent bounds:} 
We have thus shown that
\begin{align}
l^r_x  \leq \pi(x) \leq u^r_x, \quad  \forall x \in\s_r,
\label{eq:upper_lower_bounds}
\end{align}
and it is easy to see that
both bounds converge to the stationary solution as the size of the truncation grows: 
as $r \to \infty$, both $u^r_x \to \pi(x)$ and $l^r_x \to \pi(x)$.
This follows from \eqref{eq:1deg}--\eqref{eq:normalising} and $\varepsilon_r \to 0$. 

\paragraph*{\textbf{Approximating the distribution and the approximation error:}\\} %and error bounds}\\}

Motivated by these facts, we define the two following measures (lower and upper bounds padded with zeros),
\begin{align}
\label{eq:approx_lower}
l^r :=(l^r(x))_{x\in\s}, \quad l^r(x)&:=\left\{\begin{array}{ll}l^r_{x}&\text{if }x \in\s_r\\0&\text{if }x \not\in\s_r\end{array}\right. \\ 
u^r :=(u^r(x))_{x\in\s}, \quad u^r(x)&:=\left\{\begin{array}{ll}u^r_{x}&\text{if }x \in\s_r\\0&\text{if }x \not\in\s_r\end{array}\right.
\label{eq:approx_upper}
\end{align}
and introduce them as \emph{approximations} for $\pi(x)$: 
$$\widetilde{\pi} \simeq \pi, \text{ where } \widetilde{\pi}=l^r \text{ or } \widetilde{\pi} = u^r.$$

We quantify the \emph{approximation error} of $\widetilde{\pi}$ with the total variation norm: 
\begin{equation}
\label{eq:tvnorm} 
\norm{\pi-\widetilde{\pi}}=\sup_{A\subseteq\s}\mmag{\pi(A)-\widetilde{\pi}(A)},
\end{equation}
where $\pi(A)$ and $\widetilde{\pi}(A)$ are sums over sets, defined~in~\eqref{eq:sum_notation}.

For $u^r$ and $l^r$, the approximation error can be chararacterised further. Using~\eqref{eq:upperbounds}--\eqref{eq:approx_upper}, we have:
\begin{align}
\norm{\pi-l^r}&=
%\sup_{A\subseteq\s}\left(\pi(A)-\widetilde{\pi}(A)\right)=
\pi(\s)-l^r(\s) =1-\sum_{x\in\s_r}l^r_x\label{eq:low}\\
&=1-(1-\varepsilon_r)\sum_{x\in\s_r}u^r_x=\varepsilon_r,\nonumber\\
\norm{\pi-u^r}&=\max\left\{u^r(\s_r)-\pi(\s_r),\pi(\s_r^c)\right\}\label{eq:up}\\
&=\max\{1-(1-m_r),m_r\}=m_r\leq\varepsilon_r. \nonumber
\end{align}
Since $\varepsilon_r=c/r \to 0$ as $r\to \infty$, it thus follows that both 
$l^r$ and $u^r$ converge in total variation to $\pi$.

We summarise these findings in the following theorem: 
\begin{theorem}[Bounds and approximations of the stationary solution of birth-death processes]
\label{birthdeathth}
Consider any birth-death process~\eqref{eq:one_step} with non-vanishing decay rates~\eqref{eq:drpos} and finite sum $\gamma(\s)$~\eqref{eq:normalising}, such that it has a unique stationary solution $\pi$~\eqref{eq:1deg}. 
Suppose that $\pi$ satisfies the moment bound~\eqref{eq:sdp_bound} and let $\s_r \subseteq \s$ be the truncation~\eqref{eq:momtrunc} of the state space 
controlled by the parameters $r, \alpha \in \zp$, with tail bound~
$m_r = \pi(\s_r^c)\leq c/r=\varepsilon_r$.
Then the following hold:
\begin{enumerate}[label=(\roman*)]
\item \label{birtdeathth_i} The values of the distribution over the truncation are bounded above and below:
$$ l^r_x \leq \pi(x) \leq u^r_x, \quad  \forall x \in\s_r,$$
where $u^r_x=\gamma(x)/\gamma(\s_r)$ and $l^r_x=u^r_x (1-\varepsilon_r)$.
\item The measures 
$l^r =(l^r(x))_{x\in\s}$~and~$u^r =(u^r(x))_{x\in\s}$ defined in~\eqref{eq:approx_lower}--\eqref{eq:approx_upper} 
approximate the solution with approximation errors:
$$ \norm{\pi-l^r} = \varepsilon_r \text{ and } \norm{\pi-u^r} = m_r.$$ 
\item
The bounds vary monotonically with $r$:
$$l^r_x\leq l^{r+1}_x\leq \dots\leq\pi(x)\leq \dots \leq u^{r+1}_x \leq u^r_x, \quad  \forall x \in \s_r$$
and the sequences of approximations 
converge in total variation to $\pi$:
$$ \lim_{r \to \infty} \norm{\pi-l^r} = \lim_{r \to \infty} \norm{\pi-u^r} = 0.$$
\end{enumerate}
\end{theorem}
\begin{proof}This follows from~\eqref{eq:upperbounds}--\eqref{eq:up} and Corollary~\ref{boundcor}.
\end{proof}

\begin{remark}
If Assumption~\ref{ass:1} holds, then $\pi$ satisfies the moment bound \eqref{eq:sdp_bound} with  $c = U^{d}_\alpha$, where  $\alpha\in\{1,\dots,d-d_s\}$ and $U^{d}_\alpha$ is defined in \eqref{eq:sdp2}.
\end{remark}

\begin{figure}
	\begin{center}
	\includegraphics[width=0.48\textwidth]{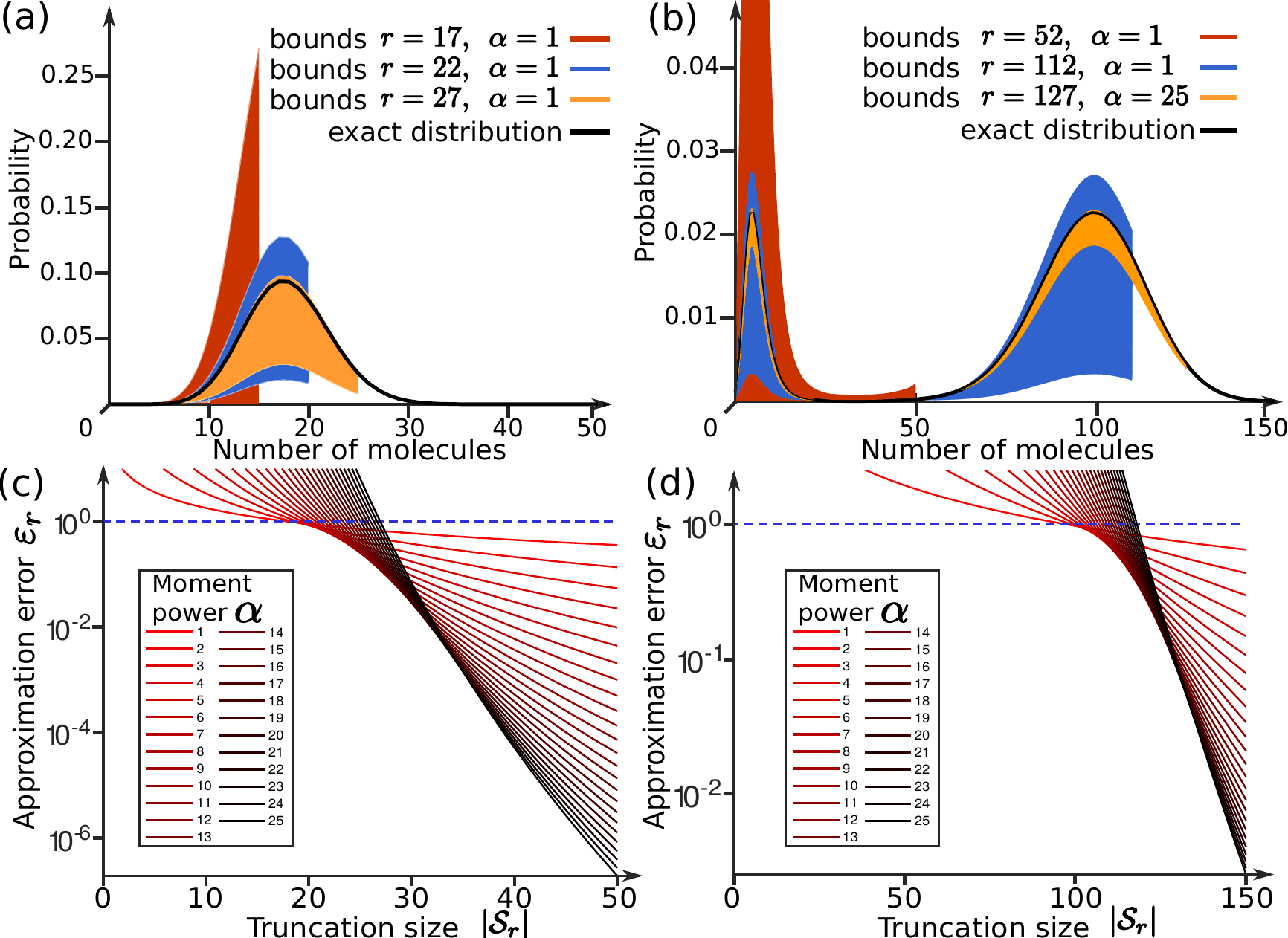}
	\vspace{-20pt}
	\end{center}
\caption{\textbf{Bounding the stationary solution of Schl\"ogl's model~\eqref{eq:smdl}.}
The stationary solution of~\eqref{eq:smdl} can be unimodal (a) or bimodal (b) depending on the parameters.
\textbf{(a)} Shadings show the tightening gap $(l_x^r,u_x^r)$ between upper and lower bounds on the stationary solution for truncations $\s_r$ of increasing size ($\mmag{\s_r}=\lceil r^{1/\alpha}\rceil$). 
We use $m_r \leq U^{25}_\alpha/r$ with $U^{25}_1 =17.5$ computed as in Sec.~\ref{bmom} (solver time = $7$ seconds).
The exact solution, given by~(\ref{eq:1deg})~and~(\ref{eq:schloegl_p0}), is shown for comparison (black line). 
\textbf{(b)} Same as (a) but for the  bimodal case, 
with tail bounds computed using
$U^{25}_{1}=98.0, U^{25}_{25}=6.37 \times 10^{51}$ (solver time = $8$ and $7$ seconds, respectively). 
\textbf{(c)}  The approximation error  of the lower bound approximation, $\varepsilon_r$, %$l_r$~\eqref{eq:low} 
for the unimodal case decreases as $\mmag{\s_r}$ increases, shown here
for various values of $\alpha$. 
\textbf{(d)} Same as (c) but for the bimodal case. 
Note that only when the truncation includes enough states (different values of $\mmag{\s_r}$ for \textbf{(c)} and \textbf{(d)}), does the error fall below the dashed line
$\varepsilon_r < 1$, so that the bounds provide information about the stationary solution. For the bimodal case, the error indicates the presence of a second mode outside of the truncation when $\mmag{\s_r}$ is too small. 
Parameters: (a): $k_1=6$, $k_2=1/3$, $k_3=50$, $k_4=3$;
(b): $k_1=1/9$, $k_2=1/1215$, $k_3=27/2$, $k_4=59/20$.
} 
\label{fig:schodist}
\end{figure}
\noindent \paragraph*{\textbf{An application of Theorem~\ref{birthdeathth}: 
Schl\"ogl's model} \\}
To illustrate our results, we apply Theorem~\ref{birthdeathth} to compute bounds on the unique stationary solution of Schl\"ogl's model~\eqref{eq:smdl}--\eqref{eq:schloegl_props2}.
This model is a birth-death process for which an explicit analytical stationary solution can be obtained, thus allowing us to test the results directly without any simulations.

Through some analytical manipulations, the solution of Schl\"ogl's model 
can be obtained explicitly in terms of:
\begin{align}
 \label{eq:schloegl_p0}
 \frac{1}{\pi(0)}= {_2F_2}\left(-\frac{c_1+1}{2},\frac{c_1-1}{2};-\frac{c_2+1}{2},\frac{c_2-1}{2};\frac{k_1}{k_2}\right),
\end{align}
where ${_2F_2}$ denotes the generalised hypergeometric function; $c_1:=\sqrt{1-4 {k_3}/{k_1}}$; and $c_2:= \sqrt{1-4 {k_4}/{k_2}}$.
The stationary solution goes from being unimodal (black line, Fig.~\ref{fig:schodist}(a)) to bimodal (black line, Fig.~\ref{fig:schodist}(b)) depending on the parameter values, analogously to a bifurcation.

In Fig.~\ref{fig:schodist}(a)--(b), we compare the approximations $l^r$ and $u^r$ given in Theorem \ref{birthdeathth} (colour shades) to the analytical solution (black lines) of the unimodal and bimodal cases. 
As the size of the truncation (controlled by the parameter $r$) is increased, the bounds tighten around the analytical solution.
In Fig.~\ref{fig:schodist}(c)--(d) we show that the approximation error tends to zero as the size of the truncation 
$\mmag{\s_r}=\lceil r^{1/\alpha}\rceil$  is increased. 
In the unimodal case (Fig.~\ref{fig:schodist}(c)), the approximation error decreases rapidly when $\mmag{\s_r}$ is larger than the mode. Furthermore, when the truncation is sufficiently large, employing bounds on higher order moments (larger $\alpha$) provides tighter tail bounds and smaller approximation errors. On the other hand, if the size of the truncation is smaller than the mode, using higher order moments does not necessarily improve the approximation error. A similar dependence of the approximation error is observed in the bimodal case (Fig.~\ref{fig:schodist}(d)), but the approximation error only decreases when the truncation size is larger than the second (larger) mode. This example shows how the ability to compute error bounds can reveal the presence of modes outside of the truncation.

\subsubsection{Reformulation of the bounds as optimisations} 
\label{sec:optim_birth_death}

The truncation method leading to Theorem~\ref{birthdeathth} relies on the detailed balanced structure of birth-death processes. However, this semi-analytic method is not generalisable to arbitrary reaction networks.
Instead, the bounds can be reformulated as an equivalent (and generalisable) optimisation problem, as follows.

Consider a truncation $\s_r$~\eqref{eq:momtrunc} controlled by the parameters $\alpha, r \in \zp$ with tail bound~$m_r = \pi(\s_r^c)\leq c/r=\varepsilon_r$.

\begin{definition}[Restriction of $\pi$ to $\s_r$]
The \emph{restriction} of $\pi$ to $\s_r$ is: 
\begin{equation}
\label{eq:restr}
\pi_{|r}(x) :=\left\{\begin{array}{ll}\pi(x)&\text{if }x \in\s_r\\0&\text{if }x \not\in\s_r\end{array}\right.\qquad \forall x\in\s.
\end{equation}
\end{definition}

The restriction $\pi_{|r}$ belongs to the convex polytope: 
\begin{equation}
\label{eq:polytopesch}
\cal{P}^r =
\left\{\pi^r\in\ell^1: 
\begin{array}{l}
%\pi^r\text{ satisfies \eqref{eq:introeqs1}, \eqref{eq:introeqs2} with $x<r-1$,}\\
\pi^r \,Q(x) =0,   \, \forall x\in\cal{N}_r\\
\pi^r(\s_r^c)=0  \\
\pi^r \geq 0 \\
1-\varepsilon_r\leq \pi^r(\s)\leq 1 \\
\aver{x^\alpha}\leq c
\end{array}
\right\},
\end{equation}
where $\cal{N}_r:=\{0,1,\dots,\lceil r^{1/\alpha}\rceil-1\}$, 
which follows directly from the definition of the restriction~\eqref{eq:restr};
the tail bound~\eqref{eq:Markov1}; and
the fact that the stationary equations~\eqref{eq:introeqs1}--\eqref{eq:introeqs2} with $x<\lceil r^{1/\alpha}\rceil-1$ only involve states inside of the truncation. 

From the definition of the polytope~\eqref{eq:polytopesch}, we can show that the bounds of the stationary solution in Theorem~\ref{birthdeathth}~\ref{birtdeathth_i} are recovered by optimising over $\cal{P}^r$, as stated in the following theorem.
\begin{theorem}[Bounds and LP formulation] 
The bounds $l^r$ and $u^r$ in \eqref{eq:approx_lower}--\eqref{eq:approx_upper} are obtained by optimising over the polytope~$\cal{P}^r$: 
\begin{align*}
l^r(x)&=\inf\{\pi^r(x):\pi^r\in \cal{P}^r \} \\ 
u^r(x)&=\sup\{\pi^r(x):\pi^r\in \cal{P}^r \} \quad\forall x\in\s.
\end{align*}
\end{theorem}

\begin{proof} We only present the argument for the upper bounds---the proof for the lower bounds is analogous. If $x\not\in\s_r$, the result is trivial. 
A distribution $\pi^r$ satisfies 
\eqref{eq:polytopesch} if and only if $\pi^r(x)=\gamma(x)\pi^r(0), \forall x\in\s_r$, where $\gamma(x)$ is given in~\eqref{eq:1deg}.
Therefore, we have
\begin{align}
\sup\{\pi^r(x): \pi^r & \in\cal{P}^r\}
=\gamma(x)(\sup\{\pi^r(0):\pi^r\in\cal{P}^r\}) \nonumber \\
\label{eq:sup_above}
%\sup\{\pi^r(x):\pi^r\in\cal{P}^r\}
&\leq \frac{\gamma(x)}{\gamma(\s_r)}=u^r(x), \enskip \forall x\in\s_r,
\end{align}
which follows from 
$\pi^r(0)=\pi^r(\s_r)/\gamma(\s_r) \leq 1/\gamma(\s_r).$
However, $u^r$ clearly satisfies all constraints in~\eqref{eq:polytopesch}, including the moment constraint:
\begin{align*}
 c \ge \langle x^\alpha\rangle_\pi &= \frac{\gamma(\s_r)}{\gamma(\s)}\sum_{x\in\s_r}\frac{\gamma(x)}{\gamma(\s_r)}x^\alpha+\frac{\gamma(\s_r^c)}{\gamma(\s)}\sum_{x\in\s_r^c}\frac{\gamma(x)}{\gamma(\s_r^c)}x^\alpha \notag\\
 &>\frac{\gamma(\s_r)}{\gamma(\s)}\langle x^\alpha\rangle_{u^r}+\frac{\gamma(\s_r^c)}{\gamma(\s)}\langle x^\alpha\rangle_{u^r}=\langle x^\alpha\rangle_{u^r} 
\end{align*}
which follows from $\pi(x)=\gamma(x)/\gamma(\s)$ and the inequality
\begin{align*}
\sum_{x\in\s_r^c}\frac{\gamma(x)}{\gamma(\s_r^c)}x^\alpha& \geq 
\sum_{x \in\s_r^c}\frac{\gamma(x)}{\gamma(\s_r^c)}r=\sum_{x\in\s_r}\frac{\gamma(x)}{\gamma(\s_r)}r
\\
&>\sum_{x\in\s_r}\frac{\gamma(x)}{\gamma(\s_r)}x^\alpha=\langle x^\alpha\rangle_{u^r}.
\end{align*}
Hence $u^r \in \cal{P}^r$ and together with~\eqref{eq:sup_above}, this completes the proof.
\end{proof}

Importantly, the definition~\eqref{eq:polytopesch} only involves linear equations and inequalities. Therefore optimising over the polytope consists of solving a \emph{linear program}, a class of optimisations for which there exist powerful computational platforms and algorithms. 
Furthermore, this optimisation reformulation can be extended seamlessly to arbitrary networks, as expanded in the next section.

\subsection{Generalisation to arbitrary networks  via linear programming} \label{lpsec}

We now generalise the optimisation approach to obtain bounds and approximations with controlled errors of the stationary solutions of arbitrary reaction networks.

Let us consider a reaction network \eqref{eq:network} with state space $\s$, rate matrix $Q$ satisfying \eqref{eq:qmatrix}--\eqref{eq:qmatrix2}, and stationary solutions $\pi \in \ell^1$ that form the polytope $\cal{P}$~\eqref{eq:cmestat}.

To characterise the solutions of the CME, we choose a \emph{norm-like} function $w$, which plays the same role as the moment function ($x^\alpha$) in Section~\ref{sec:optim_birth_death}. 
\begin{definition}[Norm-like function]
\label{def:norm_like}
A function $w: \s \to \r$  is norm-like if it is non-negative  
\begin{align}
\label{eq:positive_w}
w(x) \geq 0, \, \forall x \in \s
\end{align}
and has finite sublevel sets:
\begin{equation}
\label{sublevel}
\s_r:=\{x\in\s:w(x)<r\}.
\end{equation}
\end{definition}
Furthermore, we require that the growth of $w$ be \emph{dominated}
by the stationary solution $\pi$, so that its expectation with respect to $\pi$ is finite. We summarise these requirements in the following checkable assumption.

\begin{assumption}[Existence of CME solutions and moment bound]
\label{ass:2} 
We assume that the CME has at least one stationary solution $\pi$, and that we have available a norm-like function $w$ with sublevel sets~$\s_r$ such that every stationary solution $\pi$ satisfies 
\begin{equation}
\label{eq:momboundw}
\ave{w} = \sum_{x \in \s} w(x) \pi(x) \leq c,
\end{equation}
where $c$ is a known constant.
This inequality can be thought of as a generalisation of~\eqref{eq:sdp_bound}, hence we refer to it as a moment bound.

The existence of the stationary solutions can be verified on a case by case basis using a Foster-Lyapunov criterion (e.g., Theorem~\ref{lyap} in App.~\ref{appendix2}).

Regarding the moment bound, in the case of networks with rational propensities satisfying Assumption~\ref{ass:1}, $w$ can be chosen to be any norm-like rational function with numerator of degree $d$ and the bounding constant $c$ can then be computed using the SDP approach of Sec.~\ref{bmom}.
For general reaction networks, the moment bound can be obtained using Foster-Lyapunov criteria~\cite{Glynn2008} (see Appendix~\ref{appendix2}).

\end{assumption}

In analogy with Lyapunov theory, the sublevel sets~\eqref{sublevel} of $w$ play an important role in characterising the stationary solutions of the CME. 
Specifically, we use the sublevel sets $\s_r$ as our state space truncations, noting that~\eqref{eq:momboundw} allows us to establish a bound on the mass of the tail of the distribution outside of $\s_r$:
\begin{equation}
\label{eq:tailboundw}
m_r:=\pi(\s^c_r) \leq\frac{1}{r}\sum_{x\not\in\s_r}w(x)\pi(x)
 \leq\frac{\ave{w}}{r}\leq \frac{c}{r}:= \varepsilon_r.
\end{equation}
Just as 
in the previous section, this choice yields a sequence of increasing truncations
that approach the entire state space: 
$$\s_1\subseteq\s_2\subseteq\dots,\quad\bigcup_{r=1}^\infty\s_r=\s.$$

For each truncation, let $\cal{N}_r$ denote the set of states $x \in \s_r$ that cannot be reached in a single jump from outside of the truncation: 
\begin{equation}
\label{eq:nr}
%\cal{N}_r:=\{x\in\s_r: \{z\in\s:q(z,x)\neq 0\}\subseteq\s_r\},
\cal{N}_r:= \left \{x\in\s_r: q(z,x)=0, \enskip \forall z \not\in \s_r \right \},
\end{equation}
and the associated convex polytope: 
\begin{equation}\label{eq:br}
\cal{P}^r:= 
\left\{\pi^r\in\ell^1: \begin{array}{l}
   \pi^r \, Q(x)=0, \enskip \forall x\in\cal{N}_r\\
   \pi^r(\s_r^c)=0 \\
   \pi^r \geq 0 \\
1-\varepsilon_r\leq \pi^r(\s)\leq 1\\
\aver{w}\leq c 
\end{array}\right\},
\end{equation}
which, analogously to~\eqref{eq:polytopesch}, includes all the stationary equations that only involve states in $\s_r$. 

We then have the following lemma:
\begin{lemma}[Outer approximations of $\cal{P}$]
\label{outer} Suppose that Assumption~\ref{ass:2} holds and let $\pi_{|r}$ be the restriction of $\pi$ to $\s_r$, as defined in  \eqref{eq:restr}.
If $\pi \in \cal{P}$, 
then $\pi_{|r} \in \cal{P}^r$. 
\end{lemma}
\begin{proof} This follows directly from \eqref{eq:tailboundw} and the fact that $\pi Q(x)=\pi_{|r} \, Q(x), \, \forall x \in \cal{N}_r$.
\end{proof}

The outer approximation property 
means that optimising over $\cal{P}^r$ provides convergent bounds on the averages of functions~$f$ on the state space, as summarised in the following theorem.

\begin{theorem}[Convergent bounds of stationary averages] 
\label{bounds}
Consider a reaction network \eqref{eq:network} with state space $\s$, rate matrix $Q$ satisfying \eqref{eq:qmatrixsrn}--\eqref{eq:qmatrix2}, 
and stationary solutions $\pi$ forming the set $\cal{P}$~\eqref{eq:cmestat} and suppose that Assumption~\ref{ass:2} holds.

If $\pi \in \cal{P}$ and $f:\s \mapsto \r$ is any real-valued function, then we can bound its averages over the restrictions: 
\begin{align}
\label{th:bounds_i}
    l^r_f\leq\averr{f}\leq u^r_f, \quad \forall r \in \zp
\end{align}
where $\pi_{|r}$ is the restriction of $\pi$ to $\s_r$ defined in~\eqref{eq:restr} 
and the bounds are given by:
\begin{align}
\label{eq:lowupthe4}
l^r_f &:=\inf\{\aver{f}:\pi^r \in\cal{P}^r\} \notag\\ 
u^r_f &:=\sup\{\aver{f}:\pi^r \in\cal{P}^r\}.
\end{align}

If we have additional information on $f$, we have the following bounds on the full $\pi$-averages:
\begin{enumerate}[label=(\roman*)]
\item \label{th:bounds_ii} 
If $f(x) \geq 0, \, \forall x \not \in \s_r$, then $l^r_f\leq \ave{f}, \enskip  \forall r \in \zp$.
\item \label{th:bounds_iii}
If $f(x) \leq 0, \, \forall x \not \in \s_r$, then $\ave{f}\leq u^r_f, \enskip \forall r \in \zp$.
\item \label{th:bounds_iv}
If $\ave{\mmag{f}}<\infty$ (i.e., $f$ is $\pi$-integrable), then
\begin{equation}
\label{eq:general_error}l^r_f-c\left(\sup_{x\not\in\s_r}\frac{\mmag{f(x)}}{w(x)}\right)\leq \ave{f}\leq u^r_f+c\left(\sup_{x\not\in\s_r}\frac{\mmag{f(x)}}{w(x)}\right).
\end{equation}
\item \label{th:bounds_v}
If the growth of $f$ is stricly dominated by $w$ as the size $r$ of the truncations $\s_r$ increases, i.e.,
\begin{equation}
\label{eq:weakstar2}
\lim_{r\to\infty} \, \sup_{x \not\in\s_r}\frac{\mmag{f(x)}}{w(x)}=0, 
\end{equation}
then $f$ is $\pi$-integrable for all $\pi\in\cal{P}$
 and the sequences of bounds from below $(l^r_f)_{r\in\zp}$ and above $(u^r_f)_{r\in\zp}$~converge:
\begin{align}
&\lim_{r\to\infty}l^r_f=l_f:=\inf\{\ave{f}:\pi\in\cal{P}\},\notag\\
&\lim_{r\to\infty}u^r_f=u_f:=\sup\{\ave{f}:\pi\in\cal{P}\}.\label{eq:lowupthe5}\end{align}
\end{enumerate}
\end{theorem}
\begin{proof} 
Eq.~\eqref{th:bounds_i} follows directly from Lemma~\ref{outer}.

\textit{\ref{th:bounds_ii}}~and~\textit{\ref{th:bounds_iii}} follow 
from~\eqref{th:bounds_i}  and $\ave{f}=\averr{f}+\sum_{x\not\in\s_r}f(x)\pi(x)$. 

\textit{\ref{th:bounds_iv}} is a consequence of~\eqref{th:bounds_i}, the moment bound \eqref{eq:momboundw}, and the following inequality:
\begin{align}\sum_{x\not\in\s_r} \mmag{f(x)} \pi(x) &\leq \left(\sup_{x\not\in\s_r}\frac{\mmag{f(x)}}{w(x)}\right)\sum_{x\not\in\s_r}w(x)\pi(x)\nonumber\\
&\leq\left(\sup_{x\not\in\s_r}\frac{f(x)}{w(x)}\right)\ave{w}.\label{eq:f8a7ha3n3u9a}
\end{align}

\textit{\ref{th:bounds_v}} has two parts. First, the $\pi$-integrability of $f$ follows from 
\begin{align*}
\ave{\mmag{f}}=\averr{\mmag{f}}+\sum_{x\not\in\s_r}\mmag{f(x)}\pi(x)\\
\leq \averr{\mmag{f}}+\left(\sup_{x\not\in\s_r}\frac{f(x)}{w(x)}\right)\ave{w}<\infty.
\end{align*}
Second, the convergence of the bounds follows from the fact that every subsequence of $(l^r_f)_{r\in\zp}$ has a converging subsequence $(l^{r_k}_f)_{k\in\zp}$ with limit $\pi(f)$, where $\pi \in \cal{P}$ (see Remark~\ref{size_LPs} and Theorem~3.5 in Ref.~\cite{Kuntz2018a}). By definition, $l_f\leq \pi(f)$, hence taking limits in \eqref{eq:general_error} shows that $l_f\geq \pi(f)$ and the result follows. The case of the upper bounds is identical. For details, see Corollary~3.6$(i)$ in Ref.~\cite{Kuntz2018a}.
\end{proof}

\begin{remark}[LP computation]
\label{size_LPs}
The bounds $l^r_f$ and $u^r_f$ in~\eqref{eq:lowupthe4} are obtained by solving two linear programmes (LPs) with $\mmag{\s_r}$ variables, $\mmag{\cal{N}_r}$ equality constraints, and $\mmag{\s_r}+3$ inequality constraints. LP solvers return the optimal value $l^r_f$ (or $u^r_f$) and an optimal point $\pi^{*,r}$, such that $\langle f\rangle_{\pi^{*,r}}=l^r_f$ (or $\aver{f}=u^r_f$).
The optimal points exist because 
the LPs are optimisations of a continuous function over a compact non-empty subset of~$\r^{\mmag{\s_r}}$. 
\end{remark}

Theorem~\ref{bounds} provides a general framework to obtain bounds~\eqref{eq:lowupthe4} that can be used as approximations of stationary averages $\ave{f}$ with a quantifiable error given by~\ref{th:bounds_ii}--\ref{th:bounds_iv}; furthermore, under the conditions in~\ref{th:bounds_v}, the approximations converge to $\ave{f}$ as the truncations approach the entire state space $\s$ of the reaction network.

\subsubsection{The case of a unique distribution: bounds and approximations}
\label{sec:distributions}

Throughout this section, we assume that the CME has a \emph{unique} stationary solution $\pi$, i.e., $$\cal{P}=\{\pi\}.$$ 
In this case, the results of Theorem~\ref{bounds} can be strengthened.  and
In particular, the feasible points $\pi^r \in \cal{P}^r$ are good approximations of the stationary solution in the sense that they converge to $\pi$ in \emph{weak$^*$}, as detailed in the following corollary. 

\begin{corollary}[Convergence of bounds and feasible points for a unique solution] 
\label{optimalpointsconv}
Let us assume that the conditions of Theorem~\ref{bounds}\ref{th:bounds_v} hold and, in addition, 
that $\cal{P}=\{\pi\}$ consists of a single stationary solution $\pi$.
Then the upper and lower bounds~\eqref{eq:lowupthe4} converge to the average: 
\begin{align}
\label{eq:convergence_to_unique}
\lim_{r\to\infty}l^r_f=\lim_{r\to\infty}u^r_f=\ave{f},
\end{align}
and any sequence of feasible points  $(\pi^r)_{r\in\zp}$ 
belonging to the outer approximations $(\cal{P}^r)_{r\in\zp}$ (i.e., $\pi^r\in\cal{P}_r$, $\forall r\in\zp$)
converges to $\pi$ in weak$^*$:
\begin{equation}
\label{eq:weakstar}
\lim_{r\to\infty}\aver{g}= %\ave{f}.
\langle g\rangle_{\pi}
\end{equation}
for any function $g$ that satisfies~\eqref{eq:weakstar2}.
\end{corollary}
\begin{proof}This proof is similar to that of Theorem~\ref{bounds}\textit{\ref{th:bounds_v}}. For full details, see Corollary~3.6(iii) in Ref.~\cite{Kuntz2018a}.
\end{proof}

\begin{remark}
When $w$ is norm-like, convergence in weak* implies convergence in total variation---see Appendix~B of Ref.~\cite{Kuntz2018a} for a proof.
\end{remark}

\begin{remark}
Corollary~\ref{optimalpointsconv} shows that,
given a sufficiently large truncation, 
the feasible points $\pi^r$ %of $\cal{P}^r$ 
provide arbitrarily accurate approximations
to $\pi$, yet with no quantifiable bound on the approximation error 
$||\pi -\pi^r||$.
\end{remark}

\paragraph*{\textbf{Bounding and approximating a unique stationary solution with quantifiable error:}\\} 

Corollary~\ref{optimalpointsconv} can still be used to obtain approximations with quantified errors of the distribution $\pi$ itself. 
Specifically, by applying~\eqref{eq:convergence_to_unique} repeatedly with specific $f$s, the indicator function at each point of the truncation. This allows us to compute bounds on each value of the distribution $\pi(x)$ in $\s_r$.

Given a truncation $\s_r$, let us define the set of indicator functions $\{1_x\}_{x \in \s_r}$, one for every state in the truncation, where each $1_x$ is defined as \eqref{eq:indicator}.
For each function $1_x$ in the set (i.e., for each state in the truncation), we compute the bounds~\eqref{th:bounds_i} by solving the LPs~\eqref{eq:lowupthe4} with $f=1_x$. Hence we obtain lower and upper bounds on~$\averr{1_x} = \pi(x)$,  for all states in the truncation: 
\begin{align}
& l^r_x  \leq  \pi(x) \leq u^r_x, \quad \forall x \in \s_r \\ 
\text{where} \quad
l^r_x & = l^r_{1_x}= \inf\{\pi^r(x):\pi^r\in\cal{P}^r\}\label{eq:dn82n2y3rn8qrn1} \\
u^r_x & = u^r_{1_x}= \sup\{\pi^r(x):\pi^r\in\cal{P}^r\}.\label{eq:dn82n2y3rn8qrn2}
\end{align}
As in~\eqref{eq:approx_lower}--\eqref{eq:approx_upper}, we then collect these bounds,
pad them with zeros, and define two approximations for $\pi$:  
\begin{align}
\label{eq:lowerd}
l^r :=(l^r(x))_{x\in\s}, \quad l^r(x)&:=\left\{\begin{array}{ll}l^r_{x}&\text{if }x \in\s_r\\0&\text{if }x \not\in\s_r\end{array}\right. \\ 
u^r :=(u^r(x))_{x\in\s}, \quad u^r(x)&:=\left\{\begin{array}{ll}u^r_{x}&\text{if }x \in\s_r\\0&\text{if }x \not\in\s_r\end{array}\right.
\label{eq:upperd}
\end{align}
These approximations of $\pi$ have controlled errors, as summarised in the  following corollary.

\begin{corollary}[Upper and lower bounding approximations of the unique solution]
\label{compstatcc}
If Assumption~\ref{ass:2} holds and $\cal{P}=\{\pi\}$, then the approximations $l^r$~\eqref{eq:lowerd} and $u^r$~\eqref{eq:upperd} 
fulfill the following: %$$\in \cal{P}$.
\begin{enumerate}[label=(\roman*)]
\item \label{compstatcc_i}
%For any $\pi\in\cal{P}$ and $r=1,2\dots$, 
The approximations $l^r$ and $u^r$ 
bound $\pi$ from below and above, respectively,
\begin{align}
\label{eq:bound_pi_low}
    l^r(x) \leq & \pi(x)  \qquad \quad \quad \quad 
    \forall x \in \s   \\
 &\pi(x)  \leq u^r(x)     \quad \enskip  \forall x \in \s_r 
 \label{eq:bound_pi_up}
\end{align}
with quantified approximation errors $\varepsilon^l_r$ and $\varepsilon^u_r$:
\begin{align}
\norm{l^r-\pi}&=1-l^r(\s_r)=:\varepsilon^l_r,
\label{eq:epl}\\
\norm{u^r-\pi}&=\max\{u^r(\s_r)-1+m_r,m_r\}\nonumber\\
&\leq \max\{u^r(\s_r)-1+c/r,c/r\}=:\varepsilon^u_r,
\label{eq:epu}
\end{align}
where $\norm{\cdot}$ denotes the total variation norm~\eqref{eq:tvnorm}.

\item \label{compstatcc_ii}
As the truncation size $r$ approaches infinity (and $\s_r$ approaches $\s$), the approximation $u^r$ converges pointwise to the unique~$\pi$,
\begin{align}
\label{eq:converge_pointwise}
\lim_{r\to\infty}u^r(x)=\pi(x),\qquad\forall x\in\s,
\end{align}
and the approximation $l^r$ converges to $\pi$ in weak$^*$~\eqref{eq:weakstar}
to~$\pi$. 
\end{enumerate}
\end{corollary}
\begin{proof}The bounds~\eqref{eq:bound_pi_low}--\eqref{eq:bound_pi_up} follow directly from Theorem~\ref{bounds}. The error~\eqref{eq:epl} follows from~\eqref{eq:bound_pi_low}--\eqref{eq:bound_pi_up} and the fact that the total variation norm of an unsigned measure is its mass. Similarly, \eqref{eq:epu} follows from~\eqref{eq:tailboundw} and
\begin{align*}&\mmag{u^r(A)-\pi(A)}%%=\mmag{u^r(A)-\pi_{|r}(A)-\pi(A\cap\s_r^c)}\\&
\leq \max\{u^r(A)-\pi_{|r}(A),\pi(A\cap\s_r^c)\}\\
&\leq \max\{u^r(\s_r)-1+m_r,m_r\}\\
&=\max\{\mmag{u^r(\s_r)-\pi(\s_r)},\mmag{u^r(\s_r^c)-\pi(\s_r^c)}\},\quad\forall A\subseteq\s.\end{align*}
Theorem~\ref{bounds}\ref{th:bounds_v} shows that $l^r$ and $u^r$ converge pointwise to $\pi$. For any $f$ satisfying \eqref{eq:weakstar2} and $r,r'\in\zp$, we have that
\begin{align*}
\mmag{\pi(f)-l^r(f)}&\leq\sum_{x\in\s_{r'}}\mmag{f(x)} \, \left(\pi(x)-l^r(x)\right)\\
&+2\sum_{x\not\in\s_{r'}}\mmag{f(x)} \, \pi(x).
\end{align*}
Using the pointwise convergence of $l^r$ and~\eqref{eq:f8a7ha3n3u9a},
we can pick $r'$ such that the second sum is arbitrarily small and, subsequently, an $r$ such that the first sum is arbitrarily small. Hence the weak$^\ast$ convergence of $l^r$ follows.
See Theorem~4.1 in Ref.~\cite{Kuntz2018a} for details.
\end{proof}

Corollary~\ref{compstatcc} states that, for sufficiently large $r$, $l^r$ and $u^r$ are close to $\pi$. In contrast with the feasible points $\pi^r$,  we can answer the question ``is $r$ sufficiently large?'' by evaluating the errors $\varepsilon^l_r$~\eqref{eq:epl} and $\varepsilon^u_r$~\eqref{eq:epu}.
Since 
$$ \lim_{r \to \infty} \varepsilon^l_r =0, $$ 
we will always find an approximation $l^r$ that verifiably meets any given error tolerance by increasing $r$. 

\begin{remark}
Although we have no proof that $\varepsilon^u_r$ converges to zero (nor that $u^r$ itself converges to $\pi$ in total variation), all the examples we have encountered in practice exhibit convergence of $u^r$ and $\varepsilon^u_r$.
\end{remark}

To instead answer the question `when is $r$ too small?' (i.e., to establish how large the partition $\s_r$ must be to guarantee a given approximation error), the following proposition is of use.
\begin{proposition}[Achievable approximation errors] 
\label{prop:achievable_errors}
Under the same conditions as in Corollary~\ref{compstatcc}, the errors of the approximations $l^r$~\eqref{eq:lowerd} and $u^r$~\eqref{eq:upperd} cannot be made smaller than the tail bound or the mass of the tail, respectively, i.e.,
\begin{align}
\label{eq:error_bound_lower}
\norm{l^r-\pi} & \geq \varepsilon_r, \quad &\forall r \in \zp \\
\norm{u^r-\pi} & \geq m_r, \quad &\forall r \in \zp,
\label{eq:error_bound_upper}
%m_r:=\pi(\s^c_r)\leq\frac{1}{r}\sum_{x\not\in\s_r}w(x)\pi(x) \leq\frac{\ave{w}}{r}\leq \frac{c}{r}:= \varepsilon_r,
\end{align}
where $m_r=\pi(\s^c_r) \leq \varepsilon_r = c/r$.
%as given in \eqref{eq:tailboundw}. 
\end{proposition}
\begin{proof} 
The inequality \eqref{eq:error_bound_upper} follows directly from \eqref{eq:epu}. 
For~\eqref{eq:error_bound_lower}, 
recall that there exists at least one optimal point $\pi^{*,r}$ such that $\pi^{*,r}(x)=l^r(x)$ for any $x\in\s_r$ (Remark~\ref{size_LPs}). 
It is straightforward to verify that 
$$\frac{1-\varepsilon_r}{\pi^{*,r}(\s_r)}\pi^{*,r} \in \cal{P}^r,$$
which implies $\pi^{*,r}(\s_r)=1-\varepsilon_r$ (due to the minimality of $\pi^{*,r}(x)$).  Since $l^r$ bounds from below all feasible points of $\cal{P}^r$, it follows that $l^r(\s_r)\leq \pi^{*,r}(\s_r)=1-\varepsilon_r$. Combined with \eqref{eq:epl}, this gives~\eqref{eq:error_bound_lower}. 
\end{proof}
In other words, the approximation error of the lower bounds is no smaller than the tail bound, whereas that of the upper bounds is no smaller than the tail mass. 

\begin{remark}
The inequalities~\eqref{eq:error_bound_lower}--\eqref{eq:error_bound_upper} are sharp for birth-death processes with~$w(x)=x^\alpha$ (see~\eqref{eq:low}--\eqref{eq:up}).  
\end{remark}

\paragraph*{\textbf{Approximating marginal distributions:}\\} 
\label{sec:mdist}

For high-dimensional state spaces, we are often interested in marginal distributions rather than the full multivariate solution $\pi$ defined on $\s$. 
A marginalisation is associated with a \emph{partition} of the state space into a  collection of disjoint subsets:
$$ \{A_i\}_{i\in\cal{I}}, \quad   \cup_{i\in\cal{I}} A_i=\s, 
\quad  A_i \cap A_j=\varnothing,\enskip \forall i \neq j \in \cal{I},
$$ 
and the \emph{marginal distribution} is defined with respect to each subset: 
\begin{equation}
\label{eq:marginal2}
\hat{\pi}(i)=\pi(A_i) = \sum_{x\in A_i}\pi(x),\qquad \forall i\in\cal{I}.%\pi(\{x \in A_i\})
\end{equation}
Because $\{A_i\}_{i\in\cal{I}}$ is a partition of $\s$, $\hat{\pi}$ is a probability distribution on $\cal{I}$.

Typically, we are interested in marginalising the distribution of a reaction network with $n$ species (and state space $\s = \nn$) over a subset of species. For instance, if we are interested in the molecule counts of species $k$, we consider the following (infinite) set of subsets:
\begin{equation}\label{eq:kthm}
\{A_i\}_{i \in \n} \quad \text{where} \quad A_i:=\n^{k-1}\times\{i\}\times \n^{n-k},
\end{equation}
whose union trivially recovers the entire state space.
Associated with this set $\{A_i\}_{i\in\n}$ we then have the marginal distribution $\hat{\pi}$ 
$$\hat{\pi}(i)=\pi \left(\{x\in\nn:x_k=i\} \right), \qquad \forall i\in\n$$
which, in this case, corresponds to the (univariate) distribution describing the molecule counts of the $k^{th}$ species.

The marginal distribution $\hat{\pi}$ can also be bounded and approximated following a similar procedure to the one described above for the full distribution. Using the indicator functions $1_{A_i}$ as the functions $f$, we solve the analogous LPs:
\begin{align}
\label{eq:hat_low}
\hat{l}^r_i & = \inf\{\pi^r(A_i):\pi^r\in\cal{P}^r\}, \qquad \forall 
i \in \cal{I}_r \\
\label{eq:hat_up}
\hat{u}^r_i & = \sup\{\pi^r(A_i):\pi^r\in\cal{P}^r\}, \qquad \forall i \in \cal{I}_r,
\end{align}
for all the subsets $A_i$ that intersect with the truncation, i.e.,   $\cal{I}_r =\{i \in \cal{I} : A_i\cap\s_r \neq \varnothing \}$.

As before, we construct two approximations by padding~\eqref{eq:hat_low}--\eqref{eq:hat_up} with zeros:
\begin{align}
\label{eq:lowerdm}
\hat{l}^r :=(\hat{l}^r(i))_{i\in \cal{I}}, \quad \hat{l}^r(i)&:=
\left\{\begin{array}{ll} \hat{l}^r_{i}&\text{if } 
i \in \cal{I}_r
\\0&\text{if } i \not \in \cal{I}_r
\end{array}\right. \\ 
\hat{u}^r :=(\hat{u}^r(i))_{i\in \cal{I}}, \quad \hat{u}^r(i)&:=
\left\{\begin{array}{ll} \hat{u}^r_{i}&\text{if } 
i \in \cal{I}_r
\\0&\text{if } i \not \in \cal{I}_r
\end{array} \right. , 
\label{eq:upperdm}
\end{align}
which are the analogues for the marginal distribution of the approximations to the entire distribution~\eqref{eq:lowerd}--\eqref{eq:upperd}, and have similar (but not identical) properties, as summarised in the following two corollaries.

\begin{corollary}[Lower bounding approximation of the marginal distribution]
\label{compstatccm}

Let us assume that the conditions of Corollary~\ref{compstatcc} hold.
If  $\{A_i\}_{i\in\cal{I}}$ is a partition of $\s$,
then the associated marginal distribution $\hat{\pi}$~\eqref{eq:marginal2} is lower bounded by the approximation $\hat{l}^r$ defined in~\eqref{eq:lowerdm}: 
$$\hat{l}^r(i)\leq \hat{\pi}(i)\quad\forall i\in\cal{I},$$
with approximation error
\begin{equation}\label{eq:lower_upper_approx}||\hat{l}^r-\hat{\pi}||=1-\hat{l}^r(\cal{I}_r)=:\hat{\varepsilon}^l_r.\end{equation}
Furthermore, as the truncation size is increased ($r\to \infty$ and $\s_r$ approaches $\s$), $\hat{l}^r$ converges to $\hat{\pi}$ in total variation.
\end{corollary}
\begin{proof}This proof is analogous to that of Corollary \ref{compstatcc} except that one needs to use Corollary~4.3 in Ref.~\cite{Kuntz2018a} (with $g(x,y)=1_{A_y}(x)$ instead of Theorem 4.1 in Ref.~\cite{Kuntz2018a}. 
%for the convergence.
\end{proof}

\begin{corollary}[Convergent approximation of the marginal distribution]
\label{cor:upper_approximation}
Under the same conditions as in Corollary~\ref{compstatccm}, 
$\hat{u}^r$ defined in~\eqref{eq:upperdm} approximates the marginal distribution $\hat{\pi}$ with error bounded by
\begin{align}
\norm{\hat{u}^r-\hat{\pi}} &\leq \max\{\hat{u}^r(\cal{I}_r)-1+m_r,m_r\}
\nonumber\\
&\leq \max\{\hat{u}^r(\cal{I}_r)-1+c/r,c/r\}=:\hat{\varepsilon}^u_r.   
\label{eq:error_upper_approx}
\end{align}

Furthermore, as the truncation $\s_r$ approaches $\s$, 
$\hat{u}^r$ converges pointwise to $\hat{\pi}$:
$$\lim_{r\to\infty}\hat{u}^r(i)=\hat{\pi}(i),\qquad\forall i\in\cal{I}.$$ 
\end{corollary}

\begin{proof}
The proof is analogous to that of Corollary \ref{compstatcc}.
\end{proof}

Note that $\hat{u}^r$ does provide a controlled approximation of the marginal distribution, as it is a pointwise convergent approximation to $\hat{\pi}$ with a guaranteed, computable error bound $\hat{\varepsilon}^u_r$~\eqref{eq:error_upper_approx}.

\begin{remark}[Upper bounds for the marginal distribution]
\label{remark:bad_upper_bound}
The approximation $\hat{u}^r(i)$ bounds the marginal $\hat{\pi}(i)$ if and only if $A_i \cap \s_r^c = \varnothing$, i.e., when the set $A_i$ is fully contained inside the truncation $\s_r$. 
Hence $\hat{u}^r(i)$ does not provide an upper bound if the truncation does not include all the space of the marginalised variables.

However, using the fact that the probability mass of $A_i\cap\s_r^c$ is bounded by the mass of the tail $m_r$~\eqref{eq:tailboundw}, 
we have the following easy (but loose) upper bounds:
\begin{equation}
\label{eq:marbadbounds}
\hat{\pi}(i)\leq \hat{u}^r(i)+c/r\quad\forall i\in\cal{I}_r.\end{equation}
\end{remark}

\subsubsection{Non-uniqueness, ergodic distributions and a uniqueness~test 
}\label{nonuniquesec}\label{sec:nonunique}

Theorem \ref{bounds} shows that our LP optimisation over the polytopes $\cal{P}^r$ yields bounds on the stationary averages, even if there are multiple stationary solutions.
In the non-unique case, however, the gap between the lower bounds and the upper bounds will reflect the fact that the extreme points of $\pi\mapsto \ave{f}$ over $\cal{P}$ can be achieved by \emph{different} solutions in the polytope. 
Yet it is possible to characterise further the set of solutions and the extreme points in terms of the ergodic distributions of the CME, and use this description  
to turn our LP approach into a test of uniqueness, as we show below.

To see how multiple stationary solutions of the CME can arise, consider the simple reaction network
$$ \varnothing \xrightarrow{} 2S_1\xrightarrow{} \varnothing,
\qquad S_2\xrightarrow{}\varnothing,$$
with mass action kinetics.
%: $a_1(x)=1$, $a_2(x)=x_1(x_1-1)$, $a_3(x)=x_2$.
It is clear that its state space $\s =\n^2$ decomposes into three disjoint sets:
\begin{align*}
\s= & \{(x_1,0):x_1\in\n\text{ is odd}\}
\cup \{(x_1,0): x_1\in\n\text{ is even}\} \\
& \cup \{(x_1,x_2): x_1\in\n, x_2\in\zp \} 
=: \cal{C}_1 \cup \cal{C}_2 \cup \cal{T},
\end{align*}
where $\cal{C}_1$ and $\cal{C}_2$ are \emph{closed communicating classes} 
and $\cal{T}$ contains the remaining states. 
A set $\cal{C}\subseteq\s$ 
is a closed communicating class~\cite{Norris1997} if the chain can transit between any pair of states in $\cal{C}$ but cannot leave $\cal{C}$.
Another common source of multiple closed communicating classes are conservation laws in reaction networks\cite{vallabhajosyula2005}. For instance, the reactions $$2S_1\xrightleftharpoons[]{}S_2,$$ 
conserve the quantity $n=x_1+2x_2$. Hence there exists a different closed communicating class for every $ n \in \n$.

The closed communicating classes are intimately related to the stationary solutions, as summarised in the following theorem that compiles some facts that are broadly known in the literature. 
\begin{theorem}[Ergodic distributions and communicating classes~\cite{Meyn1993a}]
\label{doeblinc} 
Consider a reaction network~\eqref{eq:network} with rate matrix $Q$ satisfying \eqref{eq:qmatrixsrn}--\eqref{eq:qmatrix2}, assume $Q$ is regular, and decompose the state space as 
\begin{equation}
\label{eq:ddc}
\s= \left(\cup_{j} \cal{C}_j \right) \cup \cal{T},
%\left(\bigcup_{j=1}^{\mmag{\cal{J}}} \cal{C}_j \right) \cup \cal{T}.
\end{equation}
where $\cal{C}_j$ are closed communicating classes and $\cal{T}$ contains the remaining states. 
\begin{enumerate}[label=(\roman*)]
\item \label{doeblinc_i}
For each $\cal{C}_j$, there is at most one stationary solution $\pi_j$; hence $\pi_j(\cal{C}_j)=1$. Whenever it exists, $\pi_j$ is known as the ergodic distribution associated with~$\cal{C}_j$.
\item  \label{doeblinc_ii}
Let $\cal{J}$ be the set of indexes $j$ of the 
ergodic distributions $\pi_j$. 
The set of stationary solutions $\cal{P}$~\eqref{eq:cmestat} is the set of convex combinations of the ergodic distributions:
\begin{equation}
\label{eq:statconv}
\cal{P}=\left\{\sum_{j\in\cal{J}}\theta_j\pi_j: \theta_j\geq0 \enskip \forall j\in\cal{J}, \enskip \sum_{j\in\cal{J}}\theta_j=1\right\}.
\end{equation}
\end{enumerate} 
\end{theorem}
\begin{proof} 
If $Q$ is regular, the stationary solutions of the CME are the stationary distributions of the chain (Theorem~\ref{Qstateq} in Appendix~\ref{appendix1}). Using this fact, $(i)$ can be found in most books on continuous-time chains (e.g., Ref.~\cite[Th.~3.5.2]{Norris1997}), and $(ii)$ is given in~Ref.\cite[Th.~3.4]{Meyn1993a}.
\end{proof}

Theorem~\ref{doeblinc}
states that the ergodic distributions $\pi_j$ are orthogonal to each other and that they are the extreme points of the convex polytope $\cal{P}$~\eqref{eq:cmestat} of stationary solutions of the CME.
Because $\cal{P}$ 
is contained in the non-negative orthant of $\ell^1$, 
it follows that each face of the non-negative orthant  
contains \emph{at most one} of the ergodic distributions.

Using this fact, we obtain a computational test of the uniqueness of stationary solutions, 
as summarised in the following corollary.
\begin{corollary}[A uniqueness test]\label{uniqueness-test}If Assumption~\ref{ass:2} holds and $Q$ is regular, then
$$\cal{P} = \{\pi\}  \iff \exists x \in \s, \enskip r \geq 1: l^r(x)>0.$$
\end{corollary}
\begin{proof} The proof is by contradiction. Suppose that $\cal{P}$ is not a singleton. %Because $\cal{P}=\cal{P}$, 
Theorem~\ref{doeblinc}$(ii)$ implies that $\cal{P}$ contains two or more ergodic distributions, $\pi_j$, each associated with a different closed communicating class, $\cal{C}_j$.
Let us consider, e.g., a state $x \in \cal{C}_1$. Since the classes $\cal{C}_j$ are disjoint, then $\pi_j(x)=0, \forall j \neq 1$,
which contradicts the lower bound property of $l^r$ (Theorem~\ref{bounds}$(i)$). Hence, $\cal{P}$ must be a singleton. The converse follows from the convergence of the bounds in Corollary~\ref{compstatcc}~\textit{\ref{compstatcc_ii}}.
\end{proof}

In other words, if $\pi$ is unique, then the lower bound of any (and all) states in the support of $\pi$ is non-zero for sufficiently large $r$. Conversely, finding a single non-zero lower bound for any $x\in \s$ provides a proof of uniqueness of the distribution. Hence, if there is more than one ergodic distribution, all the lower bounds are zero for all states in the state space $\s$.

When $\pi$ is unique, $l^r$ or $u^r$ are good approximations of the stationary solution. However, this is not so in the non-unique case. 
Indeed, it is easy to show that, in the non-unique case, the lower and upper bounds are always loose: Corollary~\ref{uniqueness-test} shows that the lower bounds are trivially zero everywhere ($l^r =0$), whereas Theorems~\ref{bounds}$(ii)$ and~\ref{doeblinc} imply that for large $r$, the mass of $u^r$ will be no smaller than the number of ergodic distributions:
\begin{equation}\label{eq:upmasserg}\liminf_{r\to\infty}u^r(\s)\geq\mmag{\cal{J}},\end{equation}
hence the upper bound is not tight.

However, our LP framework can still be used to obtain approximations of the ergodic distributions by using sequences of feasible points $\pi^r$. To this end, we require the following generalisation of Corollary~\ref{optimalpointsconv}.

\begin{corollary}[Convergent approximations of ergodic distributions]
\label{optimalpointsconv2}
Let $f$ be any function that satisfies~\eqref{eq:weakstar2} (i.e., $f$ is dominated by the norm-like function $w$ as the size $r$ of the truncations increases), and let $(\pi^{*,r})_{r \in \zp}$ be a sequence of optimal points 
such that 
 \begin{align}
 %  \langle f\rangle_{\pi^r} & = l^r_f =\inf\{\ave{f}:\pi \in \cal{P}^r\} \notag \\ 
\pi^{*,r}\in\cal{P}^r \enskip \text{and} \enskip
\langle f\rangle_{\pi^{*,r}} & = u^r_f =\sup\{\aver{f}: \pi^r \in\cal{P}^r\},\notag
 \end{align}
for all $r\in\zp$.
\begin{enumerate}[label=(\roman*)]
\item  If there exists a unique point $\pi^* \in \cal{P}$ such that
\begin{align}
\langle f\rangle_{\pi^*} =\sup\{\ave{f}: \pi \in\cal{P}\}
%=\notag \\
%= l_f :=\inf\{\ave{f}:\pi\in\cal{P}\},\notag\\
%\lim_{r\to\infty}l^r_f
%=u_f:=\sup\{\ave{f}:\pi\in\cal{P}\}.
\end{align}
then the sequence of optimal points $\pi^{*,r}$ 
converges to $\pi^*$ in weak$^*$ as $r \to \infty$. 
 \item If $f$ is the indicator function of a state (or of a subset) that is contained in a closed communicating class $\cal{C}_j$ with associated ergodic distribution $\pi_j$, 
 then the sequence of optimal points $(\pi^{*,r})_{r \in \zp}$ converges to $\pi_j$ in  weak* as $r\to \infty$. 
 \end{enumerate}
\end{corollary}

\begin{proof} The proof of $(i)$ is similar to that of Theorem~\ref{bounds}\textit{\ref{th:bounds_v}}. See Corollary~3.6$(ii)$ and Remark~3.7 in Ref.~\cite{Kuntz2018a} for details. $(ii)$ follows from $(i)$, Theorem~\ref{doeblinc}$\ref{doeblinc_ii}$, and the fact that indicator functions satisfy \eqref{eq:weakstar2}.
\end{proof} 

Corollary~\ref{optimalpointsconv2} provides a rationale for how to use our computational framework to obtain approximations of the ergodic distributions $\pi_j$ in the non-unique case. %
Importantly, we do not need to know \textit{a priori} what the closed communicating classes are. 
Using the indicator function for a chosen state $x$, we obtain the sequence of optimal points $\pi^r$ satisfying $\pi^r(x)=u^r(x)$. 
Should $x$ belong to a closed communicating class $\cal{C}_j$ with ergodic distribution $\pi_j$, Corollary \ref{optimalpointsconv2}$(ii)$ shows that $\pi^r$ will converge to $\pi_j$ as $r$ tends to infinity. Indeed, by looking at the states for which $\pi^r(x)>0$, we can in principle deduce which communicating class $\cal{C}_j$ (if any) the state belongs. Once the class is known, we replace $\s$ with $\cal{C}_j$ and proceed as 
for the unique case to obtain bounds on~$\pi_j$. See Section~\ref{sec:ergodic_example} for an example of the application of this procedure.

\subsubsection{Computational implementation and numerical considerations}\label{sec:lpcom}

Let us consider a given reaction network with rational propensities. 
In order to obtain approximations of its stationary solutions with controlled error smaller than a tolerance $\epsilon$, we proceed as follows:
\begin{enumerate}
\item Verify the existence of stationary solutions $\pi$ and the finiteness of their moments (Assumption~\ref{ass:1}) using a Foster-Lyapunov criterion (Theorem \ref{lyap} in App.~\ref{appendix2}).
    
\item \label{step_normlike}
Choose a norm-like rational function $w$ and define the truncations $\s_r$ as the sublevel sets~\eqref{sublevel} controlled by $r$.

 We have found it best to choose functions $w$
 that define truncations that cover most of the probability mass and that tend quickly to infinity, so that the size of the truncation grows slowly with $r$.
 For example, if we take $w(x)=x^\alpha$, then higher values of $\alpha$ induce smaller truncation sizes $\mmag{\s_r}\approx \sqrt[\alpha]{r}$. To guide the selection of $w$, one can run the scheme with various $w$ to gain information about the shape of the distribution.

\item Use the SDP approach of Sec.~\ref{sec:mommulti} to find a moment bound~\eqref{eq:momboundw} with bounding constant $c$ satisfied by all stationary solutions (Assumption~\ref{ass:2}). In particular, we employ YALMIP\cite{Lofberg2004}, SDPA-GMP\cite{Nakata2010}, and mpYALMIP\cite{Fantuzzi2016} to formulate and solve the SDP~\eqref{eq:sdp2} with $f:=w$. See details in Sec.~\ref{sec:sdpcomp}.

\item  \label{step_sizetruncation}
    Choose an initial truncation size $r$ based on the achievable errors established in Proposition~\ref{prop:achievable_errors}. To guarantee an error smaller than our tolerance $\epsilon$, we must choose an initial $r > c/\epsilon$. 
    
\item \label{step_LP} Solve the LPs~\eqref{eq:dn82n2y3rn8qrn1} to obtain the upper and lower approximations $l_r$~\eqref{eq:lowerd} and $u_r$~\eqref{eq:upperd}. Here we use the dual simplex algorithm of CPLEX V12.6.3\cite{CPLEX} to solve the LPs. The tool YALMIP\cite{Lofberg2004} is convenient to formulate the LPs but, if speed is a priority, the model can be fed straight to the solver to avoid computational overheads.

For large truncations, the coefficients in the constraints of the LPs span many orders of magnitude,
leading to round-off errors in double-precision arithmetic 
and poor solver performance. One way to ameliorate this issue is to scale the decision variables;
in particular, scaling $\pi^r(x)$ by $-q(x,x)$  or $w(x)$
often significantly improves solver performance.

\item \label{step_errorLP} Evaluate the error of the approximations $l_r$ and $u_r$ using \eqref{eq:epl} and \eqref{eq:epu}, respectively. If the error is larger than our tolerance $\epsilon$, we increase the truncation size $r$ and return to the previous step.
    
\item 
In addition to approximating the full distribution, we can apply the above steps to compute other measures of interest by changing the LPs and associated errors in Steps~\ref{step_LP}--\ref{step_errorLP}:
\begin{itemize}
    \item If we want to approximate a marginal distribution, we solve the LPs~\eqref{eq:hat_low}--\eqref{eq:hat_up} and quantify the error using~\eqref{eq:lower_upper_approx}--\eqref{eq:error_upper_approx}.
    \item If we are interested in a particular stationary average $\ave{f}$, we instead solve the LPs~\eqref{eq:lowupthe4} and control the error using the bounds in Theorem~\ref{bounds}$(i)$--$(iii)$.
\end{itemize}
    
\item As the particular reaction network could have several stationary solutions, we check in Step~\ref{step_LP} for non-trivial lower bounds, $l_r(x) > 0$. If we find one such bound, the solution is unique (Corollary~\ref{uniqueness-test}). Otherwise, we investigate further the uniqueness question by increasing $r$ and recomputing the lower bounds to examine the presence of communicating classes as discussed in Corollary~\ref{optimalpointsconv2}.
\end{enumerate}

Our computations were carried out on a desktop computer with a 3.5GHz processor and 16GB of RAM.

\section{Application to biological examples}\label{sec:numexamples}

\begin{figure*}[htbp]
	\begin{center}
	\includegraphics[width=.8\textwidth]{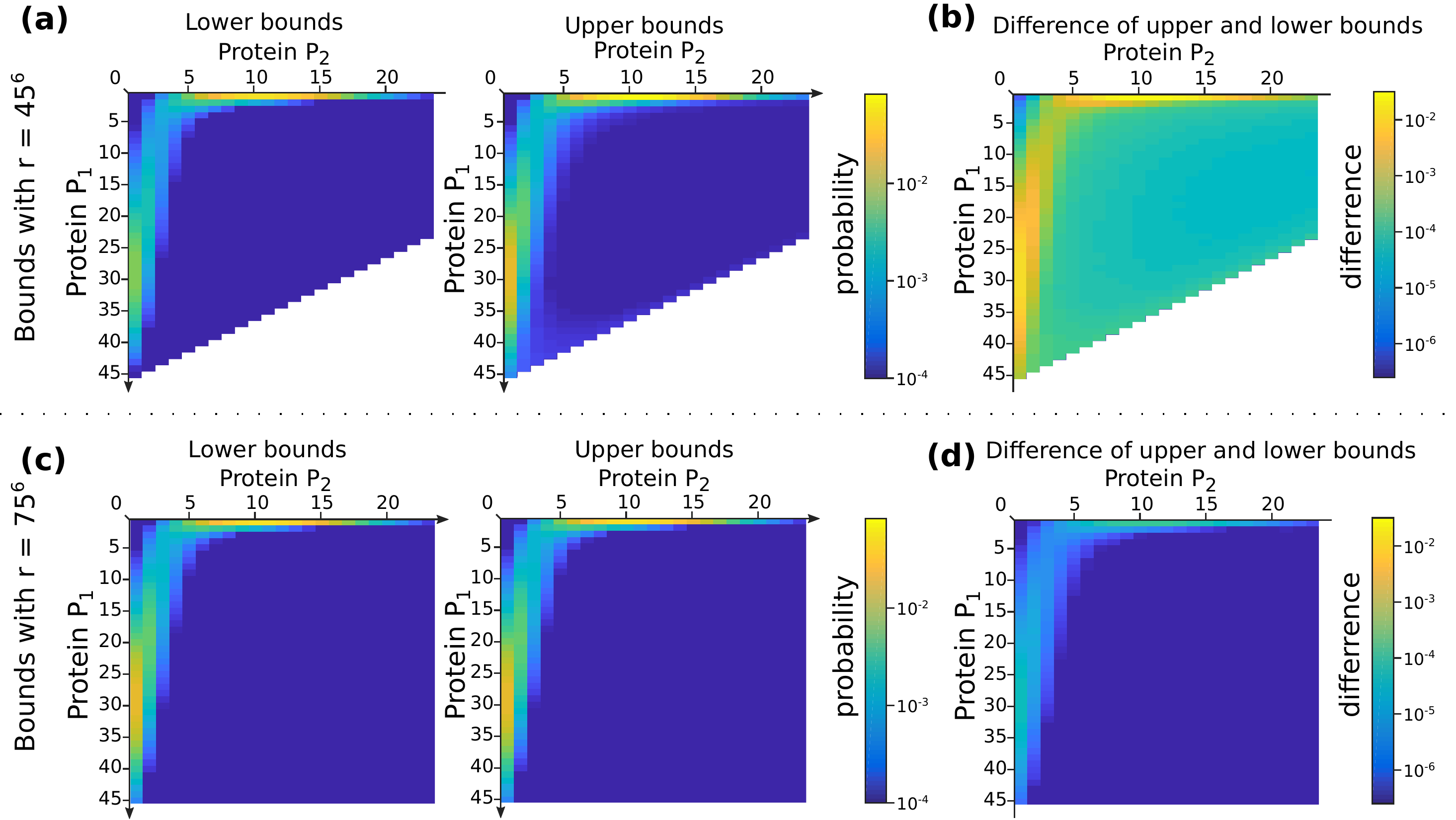} 
	\vspace{-20pt}
	\end{center}
\caption{\textbf{Bounds on the stationary solution of the toggle switch~\eqref{reaction:toggle} as the state space truncation is increased.}  \textbf{(a)} Lower and upper bounds $l^r$ and $u^r$ for the truncation
$\s_{46^6}=\{(x_1+x_2)^6< r=46^6\}$ ($990$ equations involving $1035$ states). The white areas indicate states outside of the truncation. Approximation errors of the lower and upper bounding approximations: $\varepsilon^l_r= 0.31$ 
and $\varepsilon^u_r = 0.41$, 
respectively. 
In total, $2070$ bounds were computed (solver time = $5$ minutes, average of $0.15$ seconds per bound).
\textbf{(b)} Gap between upper and lower bounds: the largest uncertainties occur near the modes.  
\textbf{(c)} Same as (a) but with truncation parameter increased to $r=75^6$ ($2775$ equations involving $2850$ states). The upper and lower bounds are visually indistinguishable, with approximation errors  $\varepsilon^l_r = 2.5\times10^{-3}$ and $\varepsilon^u_r = 2.6\times10^{-3}$.  In total, $5700$ bounds were computed (solver time = $64$ minutes, average of $0.7$ seconds per bound).
\textbf{(d)} The maximum absolute gap between bounds is less than $10^{-4}$.
Parameters: $\theta=1$, $k_1=30$, $k_2=k_4=1$ and $k_3=10$.}
\label{fig:toggle_joint}
\end{figure*}

We now present the application of the methodology to three examples.  First, we showcase how to obtain tight bounds on the stationary solution (and marginals) of a two-dimensional toggle switch. Second, we consider a model of bursty gene expression with negative feedback, through which we explore the capabilities of our method to deal with promoter switching noise. Third, we demonstrate the application of our methods to the non-unique case with a dimerisation network. The code used to compute the approximations and bounds for this last example is available at~\cite{Kuntz2017code}.

\subsection{A toggle switch}
\label{togglesec}
Toggle switches are common motifs in many cell-fate decision genetic circuits\cite{gardner2000,hemberg2007,perez2016}. A simple such circuit consists of two mutually repressing genes\cite{gardner2000}. 
In particular, we consider the asymmetric case with mutual repression modelled via Hill functions and dilution/degradation modelled via linear decay: 
\begin{equation}
\label{reaction:toggle}
\begin{array}{rcl}
\varnothing \xrightarrow{a_1} & P_1 &\xrightarrow{a_2}\varnothing, \\ %\qquad 
\varnothing \xrightarrow{a_3} & P_2 & \xrightarrow{a_4}\varnothing.
\end{array}
\end{equation}
The state space of the CME is $x \in \s = \n^2$ with $x=(x_1,x_2)$, where 
 $x_1$ and $x_2$ denote the  number of protein $P_1$ and $P_2$, respectively, and the propensities of the reactions are:
\begin{align}
 \label{toggle:prop}
 &a_1(x) = \frac{k_1}{1+(x_2/\theta)^3}, \quad
 &a_2(x) = k_2 x_1, \notag\\
 &a_3(x) = \frac{k_3}{1+x_1}, \quad
 &a_4(x) = k_4 x_2.
\end{align}
where the $k_i > 0$ are kinetic constants and $\theta>0$ is the dissociation constant of $P_1$.  

We follow the steps detailed in Section~\ref{sec:lpcom} to obtain bounds and approximations for this reaction network. First, we show that a stationary solution $\pi$ exists and that all of the moments of every solution are finite using a Foster-Lyapunov criterion (App.~\ref{appendix2}).

\begin{figure}[htbp]
	\begin{center}
	\includegraphics[width=.47\textwidth]{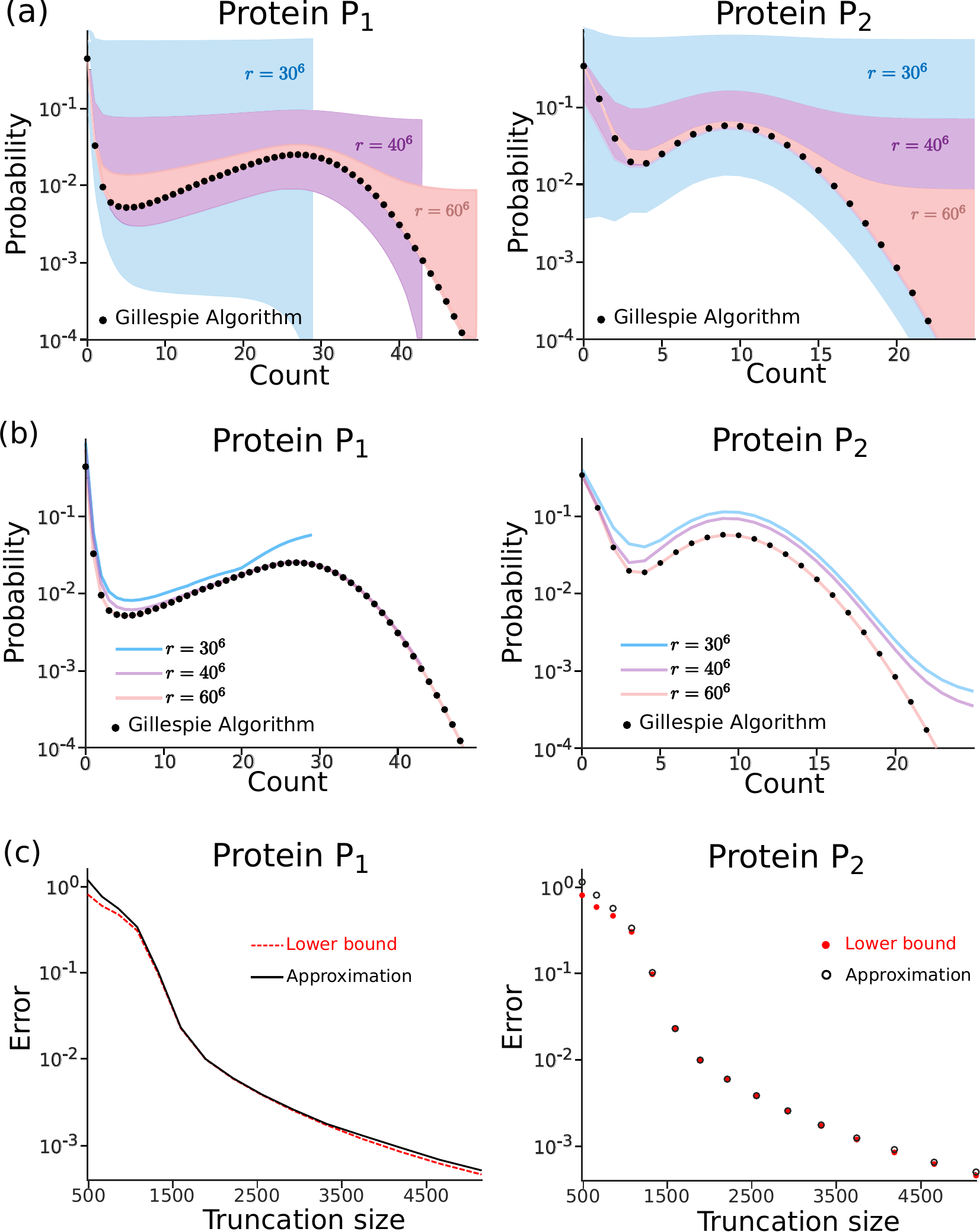}
	\vspace{-20pt}
	\end{center}
\caption{\textbf{Marginal distributions of the toggle switch~\eqref{reaction:toggle}. }
\textbf{(a)} Tight lower bounds $\hat{l}^r$ and loose upper bounds \eqref{eq:marbadbounds} on the marginal distributions of both proteins computed for increasing state space truncations $r=30^6$ (cyan), $40^6$ (purple), and $60^6$ (pink). Overall, $520$ bounds were computed (solver time = $86.6$ seconds, average of $0.17$ seconds per bound). 
For comparison, we show simulations performed using the Gillespie Algorithm with $10^8$ samples (black dots). 
\textbf{(b)} Controlled approximations $\hat{u}^r$ for increasing state space truncations $r=30^6$ (cyan), $40^6$ (purple), and $60^6$ (pink).
\textbf{(c)} Approximation error $\hat{\varepsilon}^l_r$ of the lower bound $\hat{l}^r$ (red) and bound $\hat{\varepsilon}^u_r$ on error of $\hat{u}^r$ (black). In total, $3900$ bounds were computed (solver time = $48.8$ minutes, average of $0.75$ seconds per bound).
Parameters as in Fig.~\ref{fig:toggle_joint}.
}
\label{fig:toggle_marginals}
\end{figure}

We pick the norm-like function
$$w(x):=(x_1+x_2)^6,\qquad\forall x\in\n^2$$
and compute the moment bound 
$$\ave{w}\leq c := 4.48\times10^8$$
by solving the SDP~\eqref{eq:sdp2} with $d:=10$ and $f:=w$ (solver time = $3.6$ minutes). 

We then solve the LPs~\eqref{eq:dn82n2y3rn8qrn1}--\eqref{eq:dn82n2y3rn8qrn2} and compute the bounding approximations $u^r$ and $l^r$ of the stationary solutions.
The fact that the lower bounds $l^r$ are non-zero provides us with a proof of uniqueness of the stationary solution.
Figure~\ref{fig:toggle_joint} shows the bounds for small ($r=45$) and large ($r=75$) state space truncations. 
The maximum absolute discrepancies are found near the modes, and by increasing the size of the truncation, the upper and lower bounds become nearly indistinguishable---the maximum discrepancy drops under $10^{-4}$ (Fig.~\ref{fig:toggle_joint}(d)). Overall, the total approximation error \emph{including the tail} is less than $2.6\times10^{-3}$, as given by \eqref{eq:epl}--\eqref{eq:epu}.

We have also used our method to obtain approximations on the marginal distributions of the number of proteins $P_1$ ($x_1$) and $P_2$ ($x_2$) (Sec.~\ref{sec:mdist}).
The results in Fig.~\ref{fig:toggle_marginals}(a) show that the bounds get tighter for truncations of increasing $r$ (although, as discussed in Remark~\ref{remark:bad_upper_bound}, the upper bound~\eqref{eq:marbadbounds} remains loose). Note, however, that the approximation $\hat{u}^r$ in Fig.~\ref{fig:toggle_marginals}(b) rapidly approaches the Gillespie numerical simulations.  
Fig.~\ref{fig:toggle_marginals}(c) shows that the errors of both $\hat{l}^r$ and $\hat{u^r}$ can be made arbitrarily small by increasing the truncation size (Corollaries~\ref{compstatccm}--\ref{cor:upper_approximation}).

Finally, we apply the method to chart the change of the stationary solution as a function of a parameter.  In particular, the dissociation constant of protein $P_2$ ($\theta$) can be thought of as a bifurcation parameter: increasing $\theta$ allows for higher expression of protein $P_1$. Fig.~\ref{fig:toggle_bifurcation} presents the lower bounds $\hat{l}$ on the marginals of both proteins.
At small values of $\theta$, we observe a single population with high numbers of $P_2$ repressing $P_1$. At large values of $\theta$, the opposite happens: the population we observe has high numbers of $P_1$ repressing $P_2$. For intermediate $\theta$, we observe coexistence
of both populations. Indeed, we find that the modes of the marginal distributions are in good correspondence with the stable solutions of the deterministic steady-state rate equations.

\begin{figure}[htb!]
	\begin{center}
	\includegraphics[width=0.47\textwidth]{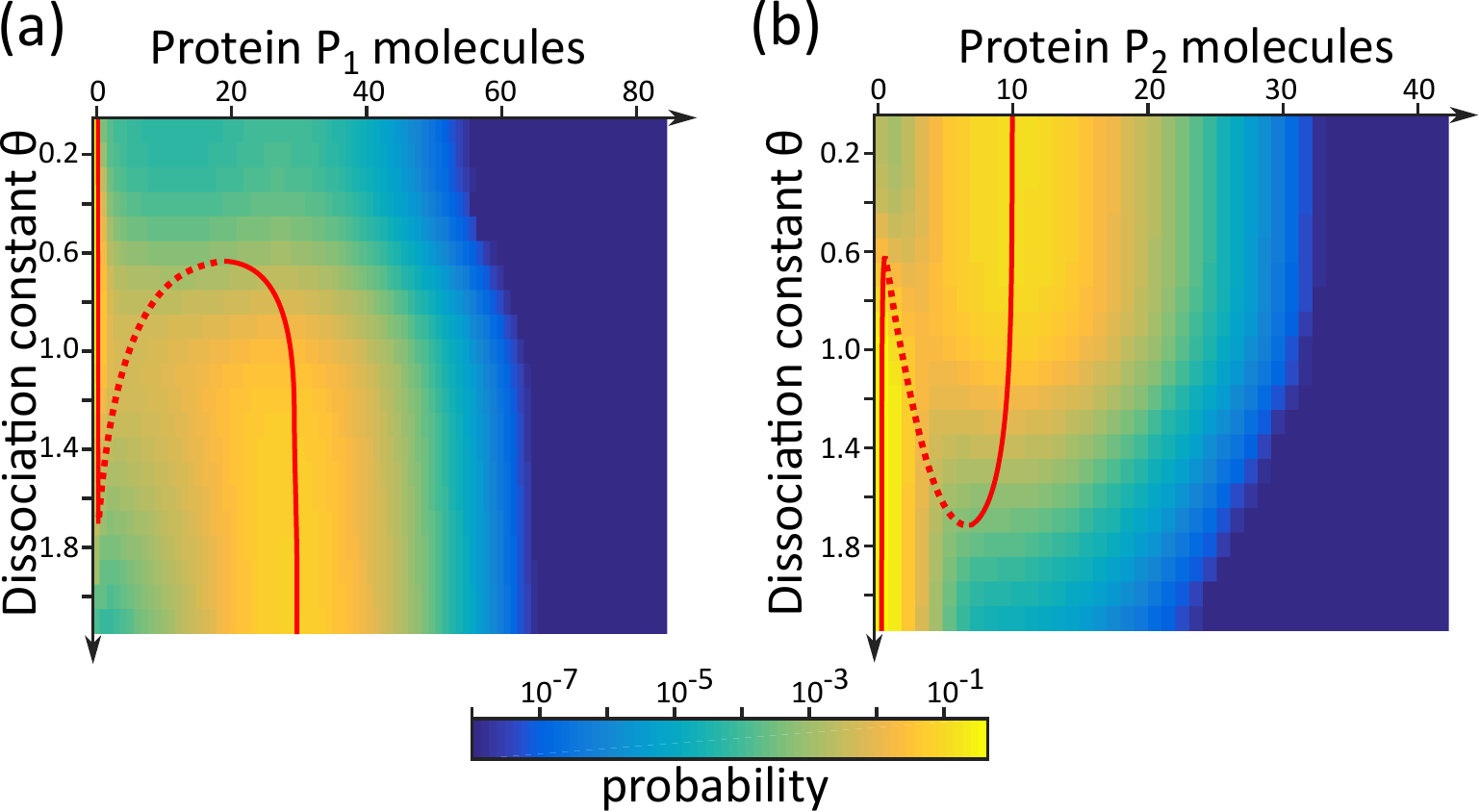}
	\vspace{-20pt}
	\end{center}
\caption{\textbf{Toggling the switch: deterministic \textit{vs} stochastic.} 
As the promoter dissociation constant $\theta$ is increased,
the system switches from a state where overexpression of $P_2$ represses $P_1$ (low $\theta$) to the reverse state where high $P_1$ represses $P_2$ (high $\theta$). In the deterministic case, the two states (stable fixed points, solid red lines) coexist at intermediate $\theta$, and can be reached from different initial conditions separated by  a third unstable steady state (dotted red line). For the stochastic model~\eqref{reaction:toggle}, we compute the marginal stationary probabilities (heatmap) $\hat{\pi}(x_1; \theta)$ in (a) and $\hat{\pi}(x_2; \theta)$ in (b)  for different $\theta$ values,
and observe good correspondence of the modes of the distributions with the deterministic steady states.
Each marginal is approximated by lower bounds computed using $r = 83^6$ ($3403$ equations involving $3486$ states). In total, $3486$ bounds were computed to obtain the full bifurcation diagram ($\theta$ increased in steps of $0.15$): solver time = $62$ minutes, average of 1 second per bound.
Approximation error: $\hat{\varepsilon}^l_r \leq 3\times 10^{-3}, \forall \theta$. 
All parameters (other than $\theta$) as in~Fig.~\ref{fig:toggle_joint}. 
}
\label{fig:toggle_bifurcation}
\end{figure}

\subsection{Bursty gene expression with negative feedback}\label{sec:bursty}
\begin{figure*}
	\begin{center}
	\includegraphics[width=.9\textwidth]{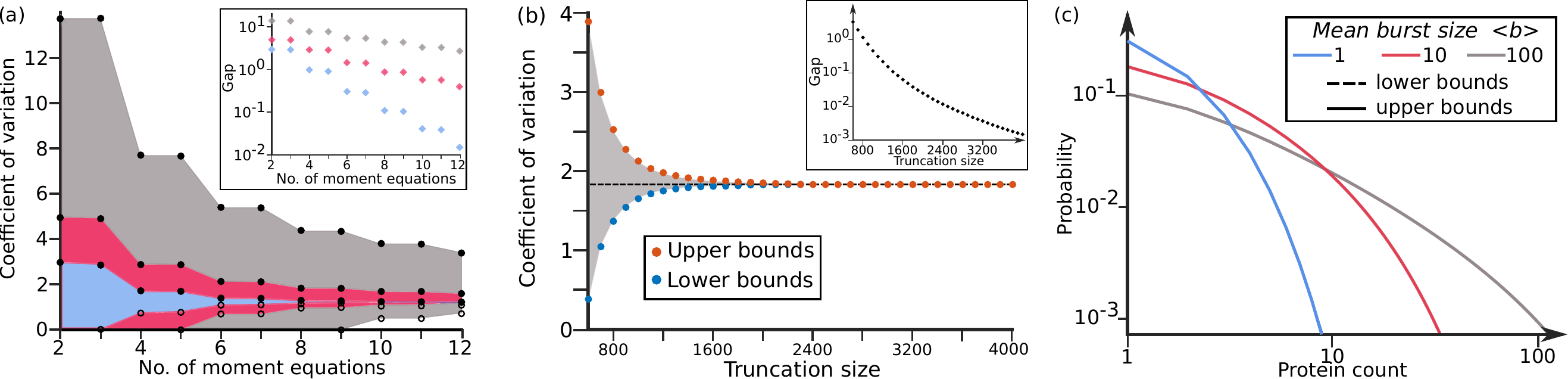}
	\vspace{-20pt}
	\end{center}
\caption{\textbf{Bounds for a bursty gene expression model with negative feedback~\eqref{reaction:gene}}. \textbf{(a)} Lower bounds (open circles) and upper bounds (filled circles) on the CV computed via SDP as the order of the approximation $d$ is increased (No. of moment equations $=d-1$) for different burst sizes, $\langle b\rangle=1$ (cyan), 10 (crimson), 100 (grey).  Colour shadings indicate the gap between bounds for different $\langle b\rangle$. In total, 
$44$ bounds were computed for each $\langle b\rangle$ (solver time=$10$ minutes, average of $13$ seconds per bound). 
Inset: the gap decreases with increasing number of moment equations, albeit more slowly for larger $\langle b\rangle$. 
\textbf{(b)} For large mean burst size ($\langle b\rangle=100$), the upper and lower bounds on the CV can be tightened using LPs with state-space truncations of increasing size (gap shaded in gray).
In total, $140$ bounds were computed (solver time = $11$ minutes, average of $5$ seconds per bound). Inset: the gap can be made arbitrarily small by increasing the truncation size.
\textbf{(c)}  Lower bounds $\hat{l}^r$ and upper bounds $\hat{u}^r$ (visually indistinguishable) on the marginal distribution of the protein for different burst sizes (with $r=10^6,55^6,500^6$ for $\langle b\rangle=1,10,100$, respectively). 
The total variation error is smaller than $4\times10^{-3}$ in all cases. Overall, $1130$ bounds were computed for all burst sies (solver time = $7$ minutes, average of $0.4$ seconds per bound). 
Parameter values as in Ref.~\cite{Kumar2014}: $k_3=k_4=10, k_1=k_2=k_5=1$.
 }
\label{fig:gene}
\end{figure*}
As a second example, consider a model of bursty production of a protein that regulates (negatively) its own expression. The model~\cite{Kumar2014} involves a promoter that switches between active  ($G_\text{on}$) and inactive ($G_\text{off}$) states, and the protein $P$ it encodes.
When the promoter is on, the protein is expressed in bursts of size $b$, a geometrically distributed random variable~\cite{shahrezaei2008} with mean~$\langle b\rangle$. The protein represses its own production by switching off the promoter:
\begin{align}
\begin{array}{rcl}
G_\text{off} &\xrightleftharpoons[a_2]{a_1}& G_\text{on},\\
G_\text{on} + P &\xrightarrow{a_3}& G_\text{off} + P,\\
G_\text{on} &\xrightarrow{a_4}& G_\text{on} +b \enskip P, \\
P &\xrightarrow{a_5}& \varnothing. 
\end{array}
\label{reaction:gene}
\end{align}
The state space of the CME is $x =(x_1,x_2) \in \s = \{0,1\}\times \n$, where $x_1=\{0,1\}$ is a binary variable describing the off/on state of the promoter and $x_2 \in \n$ represents the protein count. The propensities~are
\begin{align*}
&a_1(x)=k_1(1-x_1),\ \
a_2(x)=k_2 x_1,\\
&a_3(x)=k_3 x_2 x_1,\ \
a_4(x)= k_4 x_1,\ \ a_5(x) = k_5 x_2,
\end{align*}
where the $k_i>0$ are reaction rate constants.
In App.~\ref{appendix2} we show that the network has a unique stationary solution $\pi$ and that all of its moments are finite.

This example provides an interesting test case for SDP methods since the protein noise is particularly large: $CV(x_2)$, the coefficient of variation of $x_2$, grows~\cite{Kumar2014} with the burst size $\langle b\rangle$. 
Therefore, we expect that getting tight bounds for the CV will entail the use of a large number of moment equations.
To investigate the effect of such large noise on the efficacy of our SDP method, we compute the following bounds: 
\begin{equation}\label{eq:cvbounds}
\frac{\sqrt{L^d_{x_2^2}-(U^d_{x_2})^2}}{U^d_{x_2}}\leq \text{CV}(x_2) \leq \frac{\sqrt{U^d_{x_2^2}-(L^d_{x_2})^2}}{L^d_{x_2}},
\end{equation}
where we use~\eqref{eq:sdp1}--\eqref{eq:sdp2} and we append the following equalities to our SDP~\eqref{eq:spectrahedron}: 
$$x_1 \in\{0,1\} \implies
\ave{x_1^{\alpha_1}x_2^{\alpha_2}}=\ave{x_1 x_2^{\alpha_2}} \, \alpha_1>0, \alpha_2\geq0.$$
Figure~\ref{fig:gene}(a) shows how the bounds~\eqref{eq:cvbounds} get tighter as we increase the number of moment equations in our SDP calculations. As expected, for small mean burst sizes ($\langle b\rangle=1$), the bounds become tight with $10$ moment equations, but tightening the bounds becomes difficult when the burst size is larger ($\langle b\rangle=10, 100$). 

Computing tight bounds for large $\langle b\rangle$ with the naive SDP approach  would thus require a prohibitive number of moment equations. However, we can apply the LP method of Sec.~\ref{bdist} to overcome this limitation. To do this, use SDP to compute a (cheaper) \emph{loose} upper bound on the sixth moment $L^{13}_{x_2^6}=2.3\times10^{13} \leq \ave{x_2^6} \leq U^{13}_{x_2^6}=4.5\times10^{13}$ (time=1 minute per bound). 
We then set $w(x):=x_2^6$
and $c:=U^{13}_{x_2^6}$ 
in Theorem~\ref{bounds} to obtain:
\begin{align*}
& l^r_{x_2} \leq \ave{x_2}\leq u^r_{x_2}+\frac{c}{r^{5/6}},\quad \\
&l^r_{x_2^2} \leq \ave{x_2^2}\leq u^r_{x_2^2}+\frac{c}{r^{2/3}},
\end{align*}
which we combine as in~\eqref{eq:cvbounds} to obtain much tighter bounds on CV($x_2$). 
Figure~\ref{fig:gene}(b) shows the convergence of these tight bounds for CV($x_2$) with $\langle b\rangle =100$ as the truncation size is increased. These results exemplify the fact that it is enough to obtain loose SDP bounds on a higher order moment in order to obtain arbitrarily tight LP bounds on lower order moments (Theorem \ref{bounds}).

\begin{figure*}[htbp]
	\begin{center}
	\includegraphics[width=0.9\textwidth]{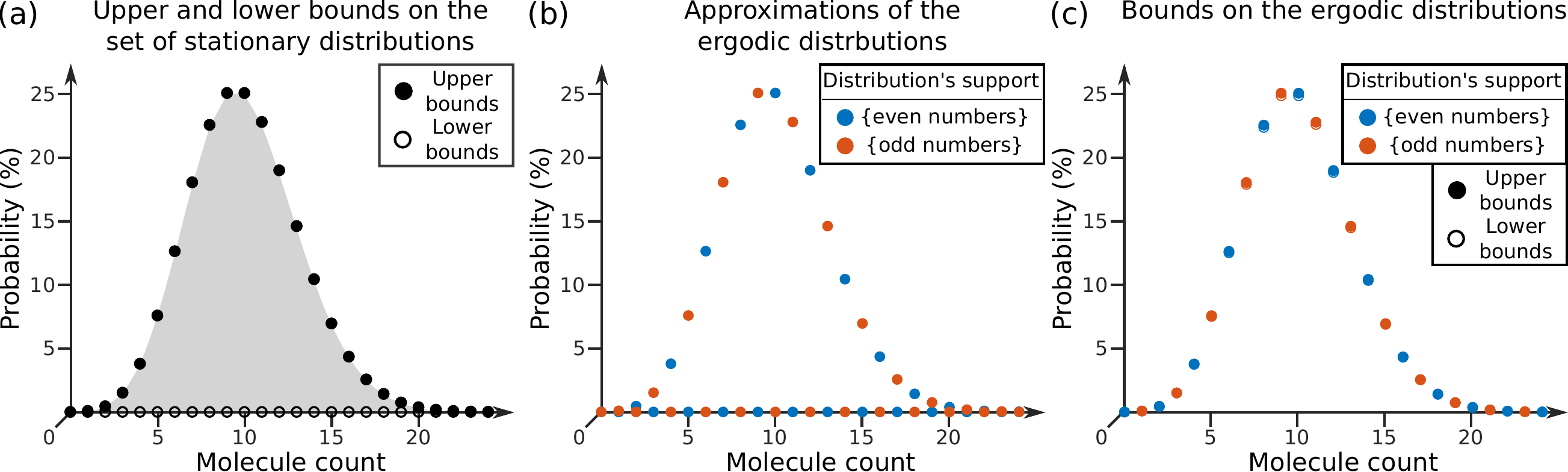} 
	\vspace{-20pt}
	\end{center}
\caption{\textbf{A dimerisation network~\eqref{eq:simple} with
multiple stationary distributions.}
\textbf{(a)} Lower ($l^r$, open circles) and upper ($u^r$, filled circles) bounds on the set of stationary distributions for the truncation $\s_{25^7}=\{x^7<r=25^7\}=\{0,1,\dots,24\}$ ($23$ equations involving $25$ states; solver time = $0.5$ seconds averaging $0.01$ seconds per bound). The gray shading indicates the gap between upper and lower bounds. Note that the lower bounds are zero indicating the presence of multiple stationary distributions.
\textbf{(b)} The optimal points $\pi^{*,r}_0$ and $\pi^{*,r}_1$ of the LPs~\eqref{eq:simplelps} (solver time = $0.01$ seconds per optimal point) provide approximations of the ergodic distributions and indicate that there are two closed communicating classes: the even numbers and the odd numbers. \textbf{(c)} The lower and upper bounds (open and filled circles, visually indistinguishable) computed separately on each of the two ergodic distributions with support on the odd numbers (in red) and even numbers (in blue) (solver time = $0.5$ seconds, average of $0.01$ seconds per bound). Parameter values: $k_1=50$ and $k_2=0.5$. The lower bounds are now non-zero, indicating the uniqueness of each of the ergodic distributions.}
\label{fig:8}
\end{figure*}

Finally, we exemplify in Figure~\ref{fig:gene}(c) another use of our capability to bound marginal distributions following the steps in Section~\ref{sec:lpcom}. 
In this case, we marginalise over the on/off promoter variable ($x_1$) and we compute upper bounds $\hat{u}^r$ and lower bounds $\hat{l}^r$ (visually indistinguishable in Fig.~\ref{fig:gene}(c)) on $\hat{\pi}(x_2)$, the distribution of protein counts, for three different burst sizes, 
$\langle b\rangle=1, 10, 100$.
As expected, the protein distribution widens considerably (yet still with light tails) as the burst size increases.  
To compute these bounds, we set
$$A_i=\{0,1\}\times \{i\},\qquad\forall i\in\{i\in\n:i^6<r\}=:\cal{I}_r,$$
and solve the LPs \eqref{eq:hat_low}--\eqref{eq:hat_up} to obtain  \eqref{eq:lowerdm}--\eqref{eq:upperdm}.
Note that this marginalisation is over the \emph{complete}, untruncated domain of the marginalised variable $x_1=\{0,1\}$. As a result, the $\hat{u}^r$ do provide upper bonds in this case (Remark~\ref{remark:bad_upper_bound}). 
\subsection{Dimerisation network with multiple stationary distributions}
\label{sec:ergodic_example}
To illustrate the use of our method on CMEs with multiple stationary solutions (see Section~\ref{sec:nonunique}), we consider the reversible dimerisation network
\begin{equation}\label{eq:simple}\varnothing \xrightarrow{a_1} 2S\xrightarrow{a_2} \varnothing,\end{equation}
with state space $\s=\n$ and mass action kinetics: $a_1(x):=k_1$ and $a_2(x):=k_2 \, x_1(x_1-1)$. The reactions preserve the parity of the number of molecules;   
% and the propensities are non-zero everywhere (except for $a_2(0)=a_2(1)=0$), 
hence  the even numbers $\cal{C}_0$ and the odd numbers $\cal{C}_1$ are closed communicating classes. Furthermore, because the network obeys detailed balance\cite{vanKampen1976}, it is straightforward to obtain analytical expressions for the two ergodic distributions:
\begin{equation}\label{eq:simpleana}\pi_i(x)=\frac{1}{Z_i}\frac{\mu^x}{x!} \quad \forall x \in \cal{C}_i \quad i=0,1,\end{equation}
where $\mu:=\sqrt{k_1/k_2}$ and the normalising constants are given by $Z_0:=\cosh(\mu)$ and $Z_1:=\sinh(\mu)$.

To illustrate the use of our tools in such a non-unique case, suppose we were not aware of the above facts and instead apply our computational procedure. First, we use a Foster-Lyapunov criterion and find that the rate matrix is regular; at least one stationary solution exists; all of the moments of each stationary solution are finite; and, for each closed communicating class, there exists an ergodic distribution with support in that class (App.~\ref{appendix2}).

Next, we obtain bounds on the stationary solutions. 
Using the norm-like function
$$w(x):=x^{7},\qquad\forall x\in\n, $$
we solve the SDP~\eqref{eq:sdp2} with $d:=7$ and $f:=w$ (solver time = $3$ seconds) to obtain the moment bound 
$$\ave{w}\leq c = 4.8814\times10^{7}.$$
We then solve the LPs~\eqref{eq:dn82n2y3rn8qrn1}--\eqref{eq:dn82n2y3rn8qrn2} to compute upper and lower bounds, $u^r$ and $l^r$ (Fig.~\ref{fig:8}(a)). 
Note that the lower bounds remain trapped at zero, indicating that the stationary solution is non-unique (as given by Theorem~\ref{uniqueness-test}). 
Furthermore, we observe that the mass of the upper bounds approaches two as the truncation grows 
(i.e., $u^{r}(\s_{r}) \to 2$ as $r \to \infty$),
%(e.g., $u^{r}(\s_{r})\approx 2$ for $r=25^7$) 
hinting that there exist \emph{two} ergodic distributions corresponding to two closed communicating classes (c.f.~\eqref{eq:upmasserg}).

To identify the communicating classes, we note that the optimal points $\pi^{*,r}_0$ and $\pi^{*,r}_1$ of the linear programs
\begin{equation}\label{eq:simplelps}
\sup\{\pi^r(0):\pi^r\in\cal{P}^r\},\quad \sup\{\pi^r(1):\pi^r\in\cal{P}^r\}
\end{equation}
approach distributions with support on the even and odd numbers, respectively (Fig.~\ref{fig:8}(b)).
These observations indicate that the communicating classes are the even numbers $\cal{C}_0$ and the odd numbers $\cal{C}_1$ 
(Corollary~\ref{optimalpointsconv2}). 

To verify this claim, we compute two different sets of upper and lower bounds over $\cal{C}_0$ and $\cal{C}_1$ separately. 
When computed over the subsets $\cal{C}_0$ and $\cal{C}_1$ separately, the lower bounds are non-zero, thus
attesting to the uniqueness of each stationary distribution over its communicating class. 
Using $r=25^7$, the total variation error of the lower bounds~\eqref{eq:epl} is approximately $8\times10^{-3}$, while that of the upper bounds~\eqref{eq:epu} is bounded above by $8\times10^{-3}$. In fact, the actual errors of the upper bounds computed from the analytical expressions~\eqref{eq:simpleana} are $<10^{-4}$,
i.e., substantially smaller than the guaranteed bound.

\section{Discussion}
\label{sec:discussion}

We have introduced two mathematical programming approaches that yield bounds on: (i) the stationary moments, and  (ii) the stationary distributions of biochemical reaction networks. 
These statistical quantities typically satisfy infinite sets of coupled equations: (i) the stationary moment equations and (ii) the stationary CME. Both our approaches consider a subset of these equations and employ: (i) semidefinite programming and (ii) linear programming to bound the set of solutions. The bounds we obtain provide converging
estimates of moments and probabilities with quantifiable errors.

Regarding our first method, which provides bounds for stationary moments, recently, and independently of our work, SDP-based procedures  have been proposed by several authors~\cite{Kuntz2017,Dowdy2017,Ghusinga2017a,Sakurai2017}. 
Our work differs from those works in two ways: firstly, our results apply to networks with both polynomial and \emph{rational} propensities, a wider class of networks of interest in biochemistry, beyond the mass action models considered in Refs.~\cite{Dowdy2017,Ghusinga2017a,Sakurai2017}; 
secondly, we give mathematically precise conditions for the validity  of the method (Assumption~\ref{ass:1}) and we explain how these conditions can be verified in practice. 
To the best of our knowledge, our second approach, the LP bounding and approximation procedure for probability distributions, has not appeared in the CME literature  (see Ref.~\cite{Kuntz2018a} for a discussion of related methodologies in the optimisation
literature). 
Importantly, both methods are tightly interlinked: our second method uses SDP moment bounds to formulate the LPs, in order to obtain controlled approximations of stationary solutions and marginals.

For some CMEs, the SDP approach might need to include a large number of moments to obtain accurate estimates of lower order moments (see Figs.~\ref{fig:schomom}~and~\ref{fig:gene}). Such large SDPs pose a computational challenge for larger networks, as the number of moments $\#_d$ is $\binom{n+d}{d}$, where $n$ denotes the number of species in the network and $d$ is the maximum moment order, and thus explodes combinatorially with the number of species. Similar costs are encountered when using moment-closure methods~\cite{lakatos2015,schnoerr2015}.
In contrast with moment closure methods \cite{engblom2006,engblom2014,lakatos2015,schnoerr2015}, however, the proposed SDP method to bound moments yields approximations with quantified errors. Furthermore, we show that repeated applications of our SDP method yield upper (resp. lower) bounds that are monotonically decreasing (resp. increasing) as the number of moment equations and inequalities is increased (Theorem~\ref{inclu}). Although, as mentioned in Sec.~\ref{sec:mommulti}, there are reaction networks for which the bounds do not converge to the exact moments, they often converge in practice (Fig.~\ref{fig:schomom} and other examples in Refs.~\cite{Dowdy2017,Dowdy2018,Ghusinga2017a,Sakurai2017,Sakurai2018}). 
In addition, when tight SDP bounds prove computationally too expensive, our LP approach can be used to tighten bounds on moments of interest employing a loose SDP bound on a higher moment (Fig.~\ref{fig:gene}).
Lastly, it is possible to obtain sharper SDP bounds for restricted state spaces~\cite{Lasserre2009,Laurent2009,Blekherman2013,Dowdy2018}, but these refinements are beyond the scope of this paper.

As stated above, the LP approach produces convergent bounds on the stationary solutions (including their marginals or averages). To do so, it uses a moment bound obtained with the SDP method. 
It is worth remarking that, while we have limited ourselves to rational networks where the moment bound can be obtained using SDPs, the LP approach can be extended beyond rational propensities by using Foster-Lyapunov criteria~\cite{Glynn2008}. 
If the CME has a unique solution, our LP method yields converging lower and upper bounds on this solution and easy-to-evaluate error bounds (Corollaries \ref{compstatcc}, \ref{compstatccm}~and~\ref{cor:upper_approximation}). If the CME has multiple solutions, our method provides bounds over the set of possible solutions. Furthermore, the procedure can be adapted to infer the closed communicating classes and to compute converging approximations of (and bounds on) the corresponding ergodic distributions. Additionally, if we are unsure whether the stationary solution is unique, our method provides a uniqueness test (Corollary \ref{uniqueness-test}) that settles the question.

Although LP solvers are highly mature and scalable, the applicability of our LP approach can  present computational challenges.  
Firstly, as discussed in Sec.~\ref{sec:lpcom} (Step~\ref{step_LP}), the LPs can become ill-conditioned if the truncation is large, although this issue is mitigated by scaling the variables and by ongoing improvements in LP solvers. Secondly, although the computational cost 
of solving an LP depends on the algorithm, the cost per bound is at least $\cal{O}(\mmag{\s_r})$, where $\mmag{\s_r}$ is the size of the truncation.
For the purpose of computing the entire distribution, we need $\cal{O}(\mmag{\s_r})$ such bounds; hence  
the cost is at least  $\cal{O}(\mmag{\s_r}^2)$. 
If computing a marginal distribution where $k$ species remain,  
we need $\cal{O}(\mmag{\s_r}^{k/n})$ bounds with an  
cost of at least $\cal{O}(\mmag{\s_r}^{1+k/n})$.
If only a stationary average is of interest, 
the cost is at least $\cal{O}(\mmag{\s_r})$, since 
only two bounds per average need to be computed. 
Note also that the truncation size typically grows combinatorially in the number of unbounded species, e.g., 
the number of states for a simplex truncation $\{x\in\nn:x_1+\dots+x_n\leq M\}$ is $\binom{n+M}{n}$, where $M$ is an upper cut-off for the species count. Hence the cost of LPs  suffers a combinatorial explosion in the number of species, as for all truncation-based methods~\cite{Tweedie1971,Hart2012,Gupta2017,Dayar2011,Spieler2014,Dayar2011a}.

Several other truncation-based schemes have been proposed to approximate the stationary solutions of the CME\cite{Hart2012,Gupta2017,Dayar2011,Dayar2011a}. In contrast with ours, those schemes typically assume that the CME has a unique stationary distribution, which has to be verified separately \cite{pauleve2014,Gupta2018}. Perhaps most extensively studied is the \emph{truncation-and-augmentation} (TA) scheme, originally proposed by Seneta\cite{Seneta1967} for discrete-time chains. Its continuous-time counterpart\cite{Tweedie1971,Hart2012,Gupta2017} converges in total variation for exponentially ergodic chains, monotone chains, and certain generalisations~\cite{Hart2012}. However,  bounds on the TA approximation error can be conservative and often involve constants that are difficult to compute in practice~\cite{Hart2012,Gupta2017,Tweedie1998,Meyn1994,Liu2015,Masuyama2017,Masuyama2017a,Liu2018,Liu2018a}.
Spieler et al.\cite{Dayar2011,Spieler2014} overcame this issue by iterating the TA scheme and applying a tail bound derived from a Foster-Lyapunov criterion to bound the stationary distribution.  
Spieler's truncation-based scheme is thus closest to ours. However, their scheme entails solving only systems of linear equations which, although cheaper to compute and simpler to implement than LPs, offer no guarantee of convergence and are only applicable in the unique case. 

Another distinct feature of our method is that it enables the direct computation of bounds on the marginal distributions, without the need to compute bounds for each state of the joint distribution. Marginal distributions are of particular interest for the analysis of high-dimensional networks and for inference of model parameters from single cell data. Since our approach yields upper and lower bounds on the marginals, it can be used to bound the likelihood or likelihood ratios from experimental observations. 
This would be useful to extend the work in Ref.~\cite{fox2016} avoiding error redistribution using water filling methods, and aiding  
by selecting the size of truncations that are sufficient for parameter identifiability. Similar bounds on acceptance ratios could be used in Metropolis-Hastings algorithms to extend the applicability of our method to Bayesian inference. We therefore expect that our approach will be valuable not only for estimating distributions, but also for estimating model parameters from noisy single cell data  where accurate approximations with quantified errors are needed.

\section{Acknowledgements}
We thank Justine Dattani and Diego Oyarz\'{u}n for stimulating discussions, and Michela Ottobre and Jure Vogrinc for insights on the stability and long-term behaviour of continuous-time chains. JK gratefully acknowledges support through a BBSRC PhD studentship (BB/F017510/1); PT through a Fellowship of the Royal Commission for the Exhibition of 1851; GBS through the EPSRC Fellowship EP/M002187/1; and MB through the EPSRC grant EP/N014529/1 funding the EPSRC Centre for Mathematics of Precision Healthcare.
\appendix

\section{Minimal continuous-time Markov chains, their long-term behaviour and stationary distributions}\label{appendix1}
In practice, the chain $X=(X(t))_{t\geq0}$ in Sec.~\ref{sec:stationary} is often constructed by running the Gillespie Algorithm\cite{Gillespie1976,Kendall1950,Feller1940}, i.e., one starts the chain from a state $x$ sampled from an \emph{initial distribution} $\lambda:=(\lambda(x))_{x\in\s}$. If $q(x):=-q(x,x)$ in \eqref{eq:qmatrix}--\eqref{eq:qmatrix2} is zero, leave the chain at the $x$ for all time. Otherwise, wait an exponentially distributed amount of time with mean $1/q(x)$, sample $y\neq x$ from the probability distribution $(q(x,y)/q(x))_{y\neq x}$, and update the chain's state to $y$ (we say that the chain \emph{jumps} from $x$ to $y$ and we call the time at which it jumps the \emph{jump time}). Repeat these steps starting from $y$ instead of $x$. All random variables sampled must be independent of each other.

If $T_n$ denotes the $n^{th}$ jump time, then the limit
$$T_\infty:=\lim_{n\to\infty}T_n$$
is known as the \emph{explosion time} of the chain, i.e., the first instant by which the chain has left every finite subset of the state space\cite[Sec.~2.3]{Kuntz2018}. If no such explosion occurs, then $T_\infty = \infty$, and we say that the chain is \emph{non-explosive}: 
\begin{equation}\label{eq:nonexp}\Pbl{\{T_\infty=\infty\}}=1,\end{equation}
where $\Pb_\lambda$ denotes the probability measure underlying the chain  (the subscript $\lambda$ emphasises the fact that the starting state was sampled from the distribution $\lambda$). If \eqref{eq:nonexp} holds for every probability distribution $\lambda$ ($\lambda(x)\geq0, \forall x\in\s, \lambda(\s)=1$), then the rate matrix $Q$ is \emph{regular}.

We collectively refer to the probabilities 
$$(p_t(x))_{x\in\s,t\geq0}=(\Pbl{\{X_t=x,t<T_\infty\}})_{x\in\s,t\geq0}$$
of observing the process in the state $x:=(x_1,\ldots,x_n)\in\s$ at time $t\geq0$ as the \emph{time-varying law} of the chain. The time-veraying law is the minimal non-negative solution of the CME~\eqref{eq:CME} (see Ref.\cite[Cor.~2.21]{Kuntzthe}) 

A probability distribution $\pi:=\{\pi(x)\}_{x\in\s}$ on $\s$ is said to be a \emph{stationary} (or \emph{steady-state} or \emph{invariant}) \emph{distribution} of the chain if sampling the chain's starting position from $\pi$ ensures that it will be distributed according to $\pi$ for all time:
\begin{equation}\label{eq:statdef}\Pbp{\{X_t=x,t<T_\infty\}}=\pi(x),\quad\forall x\in\s,\enskip t\geq0.\end{equation}
Summing both sides of \eqref{eq:statdef} over $x\in\s$ and taking the limit $t\to\infty$, we find that the chain is non-explosive when its starting location is sampled from a stationary distribution:
\begin{equation}\label{eq:statnoexp}\Pbp{\{T_\infty=\infty\}}=1.\end{equation}
Taking the derivative in time of \eqref{eq:statdef}, we find that stationary distributions are stationary solutions of the CME \eqref{eq:CME} (that is, it belongs to \eqref{eq:cmestat}). The reverse direction is slightly more complicated:
\begin{theorem}[Theorem 2.41\cite{Kuntzthe}]\label{Qstateq} Let $X$ be a continuous-time chain with rate matrix $Q$ satisfying \eqref{eq:qmatrix}--\eqref{eq:qmatrix2}. A probability distribution $\pi$ on $\s$ is a stationary distribution of $X$ if and only if it is a stationary solution of the CME  and the chain is non-explosive when initialised with law $\pi$ (i.e., \eqref{eq:statnoexp} holds).
\end{theorem}
In particular, assuming that $Q$ is regular, $\pi$ is a stationary distribution if and only if it is a stationary distribution of the chain. In other words, \eqref{eq:cmestat} is an \emph{analytical} (as in non-probabilistic) \emph{linear programming characterisation}\cite{Kuntzthe,Kurtz1998} of the set of stationary distributions for regular $Q$. The non-explosivity in Theorem \ref{Qstateq} is crucial: a counterexample is the birth-death process \eqref{eq:one_step} with $a_+(x):=2^{2x}$ and $a_-(x):=2^{2x}/2$. In this case, the sum in \eqref{eq:normalising} is finite showing that the CME has a unique stationary solution $\pi$ given by \eqref{eq:1deg}--\eqref{eq:normalising}. However, \cite[Theorem 11]{Reuter1957} shows that the process is explosive for any initial distribution (including $\pi$) and it follows from \eqref{eq:statnoexp} that no stationary distribution exists.

Stationary distributions are of interest because, if the chain is stable, then, regardless of the initial distribution $\lambda$, they determine\cite{Meyn1993a} the chain's long term behaviour. i.e., the time-varying law of the chain converges in total variation to a stationary distribution, $\pi$:
\begin{equation}\label{eq:spaceaverages}\lim_{t\to\infty}\norm{p_t-\pi}=0.\end{equation}
Furthermore, the \emph{empirical distribution} $\rho_{T}$, which
denotes the fraction of the time interval $[0,T]$ that the chain will spend in a state $x$
$$\rho_T(x):=\int_0^{\min\{T,T_\infty\}}1_x(X_t)dt\quad\forall x\in\s,$$
also converges to $\pi$
\begin{equation}\label{eq:timeaverages}
\lim_{T\to\infty}\norm{\rho_T-\pi}=0\quad\Pb_\lambda\text{-almost surely},
\end{equation}
In general, the stationary distributions featuring in \eqref{eq:spaceaverages} and \eqref{eq:timeaverages} depend on the initial distribution $\lambda$ and on the starting location $X_0$, respectively.

If the chain starts in one of the closed communicating classes (defined in Sec.~\ref{nonuniquesec}), then it can never escape the class. The convergence of the empirical distribution in \eqref{eq:timeaverages} then implies that, in the stable case, there must exist at least one stationary distribution per closed communicating class $\cal{C}_j$ and that the stationary distribution must have support contained in $\cal{C}_j$ ($\pi_j(\cal{C}_j)=1$). This distribution $\pi_j$ is unique and, if the initial distribution $\lambda$ has its mass contained in $\cal{C}_j$, then\cite{Meyn1993a} both the time-varying law and the empirical distribution converge to $\pi_j$ (in the sense that the $\pi$ featuring in both \eqref{eq:spaceaverages} and \eqref{eq:timeaverages} is $\pi_j$). For this reason, $\pi_j$ is known as an \emph{ergodic distribution} of the chain. The definition of the set  $\cal{T}:=\s\backslash\cup_{j}\cal{C}_j$ featuring in the decomposition \eqref{eq:ddc} implies that the chain visits any given state $x$ in $\cal{T}$ at most finitely many times (in particular $1_x(X_t)\to0$ as $t\to\infty$ $\Pb_\lambda$-almost surely). It follows that the chain's paths must eventually either leave $\cal{T}$ or diverge to infinity. In the case of a stable chain, tending to infinity is not an option and so the chain eventually enters one of the closed communicating classes. It then follows from \eqref{eq:statdef} that no stationary distribution $\pi$ such that $\pi(x)>0$ exists for a state $x$ in $\cal{T}$. Bringing this discussion together\cite{Meyn1993a}, we have that stationary distribution $\pi$ in \eqref{eq:timeaverages} is the ergodic distribution $\pi_j$ of the closed communicating class $\cal{C}_j$ that the chain's path eventually enters, while that in \eqref{eq:spaceaverages} is a weighted combination of the ergodic distributions where the weight given to $\pi_j$ is the probability that the chain ever enters $\cal{C}_j$. Theorem \ref{doeblinc} in Sec.~\ref{nonuniquesec} follows from these facts.

\section{A Foster-Lyapunov criterion}\label{appendix2}
In practice, verifying whether a chain is stable is done by applying a Foster-Lyapunov criterion. For our examples, we will use the following well-known criterion:
\begin{theorem}[Foster-Lyapunov criterion~\cite{Chen1991,Meyn1993a,Meyn1993b}]
\label{lyap} If there exist constants $K_1\in\r$, $K_2>0$, and a norm-like function $w$ (Definition~\ref{def:norm_like}) that satisfy
$$Qw(x):=\sum_{y\in\s}q(x,y)w(y)\leq K_1-K_2w(x)\qquad \forall x\in\s,$$
then the following hold:
\begin{enumerate}[label=(\roman*)]
\item The rate matrix $Q$ is regular.
\item There exists at least one stationary distribution.
\item For each closed communicating class, there exists an ergodic distribution with support in that class.
\item Every stationary distribution $\pi$ satisfies $\ave{w}\leq K_1/K-2<\infty.$
\item The stationary distributions determine the long-term behaviour of the chain: 
\begin{enumerate}[label=\alph*)]
\item for every deterministic starting distribution there exists a stationary distribution $\pi$ satisfying \eqref{eq:spaceaverages}; 
\item for any starting distribution $\lambda$ and for $\Pb_\lambda$-almost every path, there exists a stationary distribution $\pi$ satisfying \eqref{eq:timeaverages}.
\end{enumerate}
\end{enumerate}
\end{theorem}
\begin{proof}Part $(i)$ is Ref.~\cite[Theorem 1.11]{Chen1991} (see also Ref.~\cite[Theorem 2.1]{Meyn1993b}). Parts $(ii)$--$(v)$ follow from Ref.~\cite[Theorems 8.1 and 8.2]{Meyn1993a} and Ref.~\cite[Theorem 4.6]{Meyn1993b}.\end{proof}

We have used this criterion for our examples in the main text, as follows.
\paragraph{\textbf{Schl\"ogl's model:}}
In the case of  Schl\"ogl's model \eqref{eq:smdl}, fixing $w(x):=x^{d-2}$ for an integer $d>2$, we have
$$Qw(x)=g_{d-1}(x)-k_2(d-2)x^{d},$$
where $g_{d-1}$ is a polynomial of degree $d-1$. Thus,
\begin{align*}Qw(x)&\leq \sup_{x\in\n}\left\{g_{d-1}(x)-\frac{k_2(d-2)}{2}x^{d}\right\}-\frac{k_2(d-2)}{2}x^{d}\\
&\leq \sup_{x\in\n}\left\{g_{d-1}(x)-\frac{k_2(d-2)}{2}x^{d}\right\}-\frac{k_2(d-2)}{2}x^{d-2},\end{align*}
and the supremum is finite. Taking $d\geq 3$, Theorem~\ref{lyap} tells us that~\eqref{eq:smdl} has a regular rate matrix and at least one stationary distribution; that the moments $\ave{x^1},\dots,\ave{x^{d-2}}$ are finite for any stationary distribution $\pi$; and that the limits \eqref{eq:spaceaverages}--\eqref{eq:timeaverages} hold. Since we can choose ever larger $d$, we have finiteness of all moments. Uniqueness of the distribution follows from~\eqref{eq:1deg}--\eqref{eq:normalising}.
 
\paragraph{\textbf{Toggle switch:}}
In the case of the toggle switch chain of Sec.~\ref{togglesec}, setting $w(x):=(x_1+x_2)^d$, we have
$$\begin{array}{ll}Qw(x) &=(a_1(x)+a_3(x))((x_1+x_2+1)^d-(x_1+x_2)^d)\\ 
&+(a_2(x)+a_4(x))((x_1+x_2-1)^d-(x_1+x_2)^d) \\ 
&\leq (k_1+k_3)((x_1+x_2+1)^d-(x_1+x_2)^d)\\ 
&+(a_2(x)+a_4(x))((x_1+x_2-1)^d-(x_1+x_2)^d)\\
&\leq g_{d-1}(x_1+x_2)-d(k_2x_1+k_4x_2)(x_1+x_2)^{d-1}\\
&\leq g_{d-1}(x_1+x_2)-d\min\{k_2,k_4\}(x_1+x_2)^d,\end{array}$$
where $g_{d-1}$ is a polynomial of degree $d-1$. For this reason, proceeding as we did above for Schl\"ogl's model, we have that $Q$ is regular, that the chain does have a stationary distribution, and that all of the moments are finite of each of the stationary distributions are finite. The non-trivial lower bounds in Fig.~\ref{fig:toggle_joint} and Corollary \ref{uniqueness-test} show that it is unique.

\paragraph{\textbf{Bursty gene model:}}
For the bursty gene expression model of Sec.~\ref{sec:bursty}, let $w(x):=x_2^d$ and $(p(k))_{k\in\n}$ denotes the geometric distribution, $p(k)=(1-p(0))^kp(0)$ with $p(0)=1/(1+\langle b\rangle)$), and $\langle b^l\rangle$ denotes its $l^{th}$ moment. We then have
$$\begin{array}{ll}Qw(x) &=k_4x_1\sum_{k=0}^\infty p(k)((x_2+k)^d-x_2^d)
\\&+k_5x_2((x_2-1)^d-x_2^d) \\ 
&\leq k_4\sum_{l=0}^{d-1}\binom{d}{l}\langle b^{d-l}\rangle x_2^{l}
+k_5x_2((x_2-1)^d-x_2^d)\\ 
&=g_{d-1}(x_2)-dk_5x_2^d,\end{array}$$
where $g_{d-1}$ is a polynomial of degree $d-1$ (note that all moments of a geometric random variable are finite). Because the state space is $\{0,1\}\times\n$, $w$ is norm-like. Thus, proceeding as we did above for Schl\"ogl's model, we have that $Q$ is regular, that there exists at least one stationary distribution, and that each stationary distribution has all moments finite. For each set of parameter values, we solved LP \eqref{eq:lowerd} with $x=(0,0)$ and obtained a non-trivial lower on $\pi((0,0))$.  Corollary \ref{uniqueness-test} then showed that the stationary distribution is unique.

\paragraph{\textbf{Dimerisation network with multiple stationary distributions:}} For the network \eqref{eq:simple}, choosing $w(x):=x^{d-1}$ with integer $d>1$, we have
$$Qw(x)=g_{d-1}(x)-2k_2(d-1)x^{d},$$
where $g_{d-1}$ is a polynomial of degree $d-1$. Proceeding as for Schl\"ogl's model above, we have that $Q$ is regular, that there exists at least one stationary distribution, and that each stationary distribution has all moments finite.

\section*{References}
\bibliographystyle{apsrev4-1}
\bibliography{lit}

\begin{thebibliography}{81}%
\makeatletter
\providecommand \@ifxundefined [1]{%
 \@ifx{#1\undefined}
}%
\providecommand \@ifnum [1]{%
 \ifnum #1\expandafter \@firstoftwo
 \else \expandafter \@secondoftwo
 \fi
}%
\providecommand \@ifx [1]{%
 \ifx #1\expandafter \@firstoftwo
 \else \expandafter \@secondoftwo
 \fi
}%
\providecommand \natexlab [1]{#1}%
\providecommand \enquote  [1]{``#1''}%
\providecommand \bibnamefont  [1]{#1}%
\providecommand \bibfnamefont [1]{#1}%
\providecommand \citenamefont [1]{#1}%
\providecommand \href@noop [0]{\@secondoftwo}%
\providecommand \href [0]{\begingroup \@sanitize@url \@href}%
\providecommand \@href[1]{\@@startlink{#1}\@@href}%
\providecommand \@@href[1]{\endgroup#1\@@endlink}%
\providecommand \@sanitize@url [0]{\catcode `\\12\catcode `\$12\catcode
  `\&12\catcode `\#12\catcode `\^12\catcode `\_12\catcode `\%12\relax}%
\providecommand \@@startlink[1]{}%
\providecommand \@@endlink[0]{}%
\providecommand \url  [0]{\begingroup\@sanitize@url \@url }%
\providecommand \@url [1]{\endgroup\@href {#1}{\urlprefix }}%
\providecommand \urlprefix  [0]{URL }%
\providecommand \Eprint [0]{\href }%
\providecommand \doibase [0]{http://dx.doi.org/}%
\providecommand \selectlanguage [0]{\@gobble}%
\providecommand \bibinfo  [0]{\@secondoftwo}%
\providecommand \bibfield  [0]{\@secondoftwo}%
\providecommand \translation [1]{[#1]}%
\providecommand \BibitemOpen [0]{}%
\providecommand \bibitemStop [0]{}%
\providecommand \bibitemNoStop [0]{.\EOS\space}%
\providecommand \EOS [0]{\spacefactor3000\relax}%
\providecommand \BibitemShut  [1]{\csname bibitem#1\endcsname}%
\let\auto@bib@innerbib\@empty
%</preamble>
\bibitem [{\citenamefont {Elowitz}\ \emph {et~al.}(2002)\citenamefont
  {Elowitz}, \citenamefont {Levine}, \citenamefont {Siggia},\ and\
  \citenamefont {Swain}}]{elowitz2002}%
  \BibitemOpen
  \bibfield  {author} {\bibinfo {author} {\bibfnamefont {M.~B.}\ \bibnamefont
  {Elowitz}}, \bibinfo {author} {\bibfnamefont {A.~J.}\ \bibnamefont {Levine}},
  \bibinfo {author} {\bibfnamefont {E.~D.}\ \bibnamefont {Siggia}}, \ and\
  \bibinfo {author} {\bibfnamefont {P.~S.}\ \bibnamefont {Swain}},\ }\href
  {\doibase 10.1126/science.1070919} {\bibfield  {journal} {\bibinfo  {journal}
  {Science}\ }\textbf {\bibinfo {volume} {297}},\ \bibinfo {pages} {1183}
  (\bibinfo {year} {2002})}\BibitemShut {NoStop}%
\bibitem [{\citenamefont {Taniguchi}\ \emph {et~al.}(2010)\citenamefont
  {Taniguchi}, \citenamefont {Choi}, \citenamefont {Li}, \citenamefont {Chen},
  \citenamefont {Babu}, \citenamefont {Hearn}, \citenamefont {Emili},\ and\
  \citenamefont {Xie}}]{taniguchi2010}%
  \BibitemOpen
  \bibfield  {author} {\bibinfo {author} {\bibfnamefont {Y.}~\bibnamefont
  {Taniguchi}}, \bibinfo {author} {\bibfnamefont {P.~J.}\ \bibnamefont {Choi}},
  \bibinfo {author} {\bibfnamefont {G.-W.}\ \bibnamefont {Li}}, \bibinfo
  {author} {\bibfnamefont {H.}~\bibnamefont {Chen}}, \bibinfo {author}
  {\bibfnamefont {M.}~\bibnamefont {Babu}}, \bibinfo {author} {\bibfnamefont
  {J.}~\bibnamefont {Hearn}}, \bibinfo {author} {\bibfnamefont
  {A.}~\bibnamefont {Emili}}, \ and\ \bibinfo {author} {\bibfnamefont {X.~S.}\
  \bibnamefont {Xie}},\ }\href {\doibase 10.1126/science.1188308} {\bibfield
  {journal} {\bibinfo  {journal} {Science}\ }\textbf {\bibinfo {volume}
  {329}},\ \bibinfo {pages} {533} (\bibinfo {year} {2010})}\BibitemShut
  {NoStop}%
\bibitem [{\citenamefont {Uphoff}\ \emph {et~al.}(2016)\citenamefont {Uphoff},
  \citenamefont {Lord}, \citenamefont {Okumus}, \citenamefont
  {Potvin-Trottier}, \citenamefont {Sherratt},\ and\ \citenamefont
  {Paulsson}}]{uphoff2016}%
  \BibitemOpen
  \bibfield  {author} {\bibinfo {author} {\bibfnamefont {S.}~\bibnamefont
  {Uphoff}}, \bibinfo {author} {\bibfnamefont {N.~D.}\ \bibnamefont {Lord}},
  \bibinfo {author} {\bibfnamefont {B.}~\bibnamefont {Okumus}}, \bibinfo
  {author} {\bibfnamefont {L.}~\bibnamefont {Potvin-Trottier}}, \bibinfo
  {author} {\bibfnamefont {D.~J.}\ \bibnamefont {Sherratt}}, \ and\ \bibinfo
  {author} {\bibfnamefont {J.}~\bibnamefont {Paulsson}},\ }\href {\doibase
  10.1126/science.aac9786} {\bibfield  {journal} {\bibinfo  {journal}
  {Science}\ }\textbf {\bibinfo {volume} {351}},\ \bibinfo {pages} {1094}
  (\bibinfo {year} {2016})}\BibitemShut {NoStop}%
\bibitem [{\citenamefont {Arkin}\ \emph {et~al.}(1998)\citenamefont {Arkin},
  \citenamefont {Ross},\ and\ \citenamefont {McAdams}}]{arkin1998}%
  \BibitemOpen
  \bibfield  {author} {\bibinfo {author} {\bibfnamefont {A.}~\bibnamefont
  {Arkin}}, \bibinfo {author} {\bibfnamefont {J.}~\bibnamefont {Ross}}, \ and\
  \bibinfo {author} {\bibfnamefont {H.~H.}\ \bibnamefont {McAdams}},\
  }\href@noop {} {\bibfield  {journal} {\bibinfo  {journal} {Genetics}\
  }\textbf {\bibinfo {volume} {149}},\ \bibinfo {pages} {1633} (\bibinfo {year}
  {1998})}\BibitemShut {NoStop}%
\bibitem [{\citenamefont {Dandach}\ and\ \citenamefont
  {Khammash}(2010)}]{Dandach2010}%
  \BibitemOpen
  \bibfield  {author} {\bibinfo {author} {\bibfnamefont {S.~H.}\ \bibnamefont
  {Dandach}}\ and\ \bibinfo {author} {\bibfnamefont {M.}~\bibnamefont
  {Khammash}},\ }\href {\doibase 10.1371/journal.pcbi.1000985} {\bibfield
  {journal} {\bibinfo  {journal} {PLOS Comput. Biol.}\ }\textbf {\bibinfo
  {volume} {6}},\ \bibinfo {pages} {e1000985} (\bibinfo {year}
  {2010})}\BibitemShut {NoStop}%
\bibitem [{\citenamefont {Thomas}\ \emph {et~al.}(2014)\citenamefont {Thomas},
  \citenamefont {Popovi{\'{c}}},\ and\ \citenamefont {Grima}}]{thomas2014}%
  \BibitemOpen
  \bibfield  {author} {\bibinfo {author} {\bibfnamefont {P.}~\bibnamefont
  {Thomas}}, \bibinfo {author} {\bibfnamefont {N.}~\bibnamefont
  {Popovi{\'{c}}}}, \ and\ \bibinfo {author} {\bibfnamefont {R.}~\bibnamefont
  {Grima}},\ }\href {\doibase 10.1073/pnas.1400049111} {\bibfield  {journal}
  {\bibinfo  {journal} {Proc. Natl. Acad. Sci.}\ }\textbf {\bibinfo {volume}
  {111}},\ \bibinfo {pages} {6994} (\bibinfo {year} {2014})}\BibitemShut
  {NoStop}%
\bibitem [{\citenamefont {Munsky}\ \emph {et~al.}(2015)\citenamefont {Munsky},
  \citenamefont {Fox},\ and\ \citenamefont {Neuert}}]{munsky2015}%
  \BibitemOpen
  \bibfield  {author} {\bibinfo {author} {\bibfnamefont {B.}~\bibnamefont
  {Munsky}}, \bibinfo {author} {\bibfnamefont {Z.}~\bibnamefont {Fox}}, \ and\
  \bibinfo {author} {\bibfnamefont {G.}~\bibnamefont {Neuert}},\ }\href
  {\doibase 10.1016/j.ymeth.2015.06.009} {\bibfield  {journal} {\bibinfo
  {journal} {Methods}\ }\textbf {\bibinfo {volume} {85}},\ \bibinfo {pages}
  {12} (\bibinfo {year} {2015})}\BibitemShut {NoStop}%
\bibitem [{\citenamefont {Fr{\"{o}}hlich}\ \emph {et~al.}(2016)\citenamefont
  {Fr{\"{o}}hlich}, \citenamefont {Thomas}, \citenamefont {Kazeroonian},
  \citenamefont {Theis}, \citenamefont {Grima},\ and\ \citenamefont
  {Hasenauer}}]{frohlich2016}%
  \BibitemOpen
  \bibfield  {author} {\bibinfo {author} {\bibfnamefont {F.}~\bibnamefont
  {Fr{\"{o}}hlich}}, \bibinfo {author} {\bibfnamefont {P.}~\bibnamefont
  {Thomas}}, \bibinfo {author} {\bibfnamefont {A.}~\bibnamefont {Kazeroonian}},
  \bibinfo {author} {\bibfnamefont {F.~J.}\ \bibnamefont {Theis}}, \bibinfo
  {author} {\bibfnamefont {R.}~\bibnamefont {Grima}}, \ and\ \bibinfo {author}
  {\bibfnamefont {J.}~\bibnamefont {Hasenauer}},\ }\href {\doibase
  10.1371/journal.pcbi.1005030} {\bibfield  {journal} {\bibinfo  {journal}
  {PLoS Comput. Biol.}\ }\textbf {\bibinfo {volume} {12}},\ \bibinfo {pages}
  {e1005030} (\bibinfo {year} {2016})}\BibitemShut {NoStop}%
\bibitem [{\citenamefont {Neuert}\ \emph {et~al.}(2013)\citenamefont {Neuert},
  \citenamefont {Munsky}, \citenamefont {Tan}, \citenamefont {Teytelman},
  \citenamefont {Khammash},\ and\ \citenamefont {van
  Oudenaarden}}]{neuert2013}%
  \BibitemOpen
  \bibfield  {author} {\bibinfo {author} {\bibfnamefont {G.}~\bibnamefont
  {Neuert}}, \bibinfo {author} {\bibfnamefont {B.}~\bibnamefont {Munsky}},
  \bibinfo {author} {\bibfnamefont {R.~Z.}\ \bibnamefont {Tan}}, \bibinfo
  {author} {\bibfnamefont {L.}~\bibnamefont {Teytelman}}, \bibinfo {author}
  {\bibfnamefont {M.}~\bibnamefont {Khammash}}, \ and\ \bibinfo {author}
  {\bibfnamefont {A.}~\bibnamefont {van Oudenaarden}},\ }\href {\doibase
  10.1126/science.1231456} {\bibfield  {journal} {\bibinfo  {journal}
  {Science}\ }\textbf {\bibinfo {volume} {339}},\ \bibinfo {pages} {584}
  (\bibinfo {year} {2013})}\BibitemShut {NoStop}%
\bibitem [{\citenamefont {Strelkowa}\ and\ \citenamefont
  {Barahona}(2010)}]{Strelkowa:2010}%
  \BibitemOpen
  \bibfield  {author} {\bibinfo {author} {\bibfnamefont {N.}~\bibnamefont
  {Strelkowa}}\ and\ \bibinfo {author} {\bibfnamefont {M.}~\bibnamefont
  {Barahona}},\ }\href {\doibase 10.1098/rsif.2009.0487} {\bibfield  {journal}
  {\bibinfo  {journal} {Journal of the Royal Society Interface}\ }\textbf
  {\bibinfo {volume} {7}},\ \bibinfo {pages} {1071} (\bibinfo {year}
  {2010})}\BibitemShut {NoStop}%
\bibitem [{\citenamefont {Strelkowa}\ and\ \citenamefont
  {Barahona}(2011)}]{Strelkowa:2011}%
  \BibitemOpen
  \bibfield  {author} {\bibinfo {author} {\bibfnamefont {N.}~\bibnamefont
  {Strelkowa}}\ and\ \bibinfo {author} {\bibfnamefont {M.}~\bibnamefont
  {Barahona}},\ }\href {\doibase 10.1063/1.3574387} {\bibfield  {journal}
  {\bibinfo  {journal} {CHAOS}\ }\textbf {\bibinfo {volume} {21}} (\bibinfo
  {year} {2011}),\ 10.1063/1.3574387}\BibitemShut {NoStop}%
\bibitem [{\citenamefont {Oyarzún}\ \emph {et~al.}(2014)\citenamefont
  {Oyarzún}, \citenamefont {Lugagne},\ and\ \citenamefont
  {Stan}}]{oyarzun2014}%
  \BibitemOpen
  \bibfield  {author} {\bibinfo {author} {\bibfnamefont {D.~A.}\ \bibnamefont
  {Oyarzún}}, \bibinfo {author} {\bibfnamefont {J.-B.}\ \bibnamefont
  {Lugagne}}, \ and\ \bibinfo {author} {\bibfnamefont {G.-B.~V.}\ \bibnamefont
  {Stan}},\ }\href {\doibase 10.1021/sb400126a} {\bibfield  {journal} {\bibinfo
   {journal} {ACS Synth. Biol.}\ }\textbf {\bibinfo {volume} {4}},\ \bibinfo
  {pages} {116} (\bibinfo {year} {2014})}\BibitemShut {NoStop}%
\bibitem [{\citenamefont {Zechner}\ \emph {et~al.}(2016)\citenamefont
  {Zechner}, \citenamefont {Seelig}, \citenamefont {Rullan},\ and\
  \citenamefont {Khammash}}]{zechner2016}%
  \BibitemOpen
  \bibfield  {author} {\bibinfo {author} {\bibfnamefont {C.}~\bibnamefont
  {Zechner}}, \bibinfo {author} {\bibfnamefont {G.}~\bibnamefont {Seelig}},
  \bibinfo {author} {\bibfnamefont {M.}~\bibnamefont {Rullan}}, \ and\ \bibinfo
  {author} {\bibfnamefont {M.}~\bibnamefont {Khammash}},\ }\href {\doibase
  10.1073/pnas.1517109113} {\bibfield  {journal} {\bibinfo  {journal} {Proc.
  Natl. Acad. Sci.}\ }\textbf {\bibinfo {volume} {113}},\ \bibinfo {pages}
  {4729} (\bibinfo {year} {2016})}\BibitemShut {NoStop}%
\bibitem [{\citenamefont {Tomazou}\ \emph {et~al.}(2018)\citenamefont
  {Tomazou}, \citenamefont {Barahona}, \citenamefont {Polizzi},\ and\
  \citenamefont {Stan}}]{TOMAZOU2018}%
  \BibitemOpen
  \bibfield  {author} {\bibinfo {author} {\bibfnamefont {M.}~\bibnamefont
  {Tomazou}}, \bibinfo {author} {\bibfnamefont {M.}~\bibnamefont {Barahona}},
  \bibinfo {author} {\bibfnamefont {K.~M.}\ \bibnamefont {Polizzi}}, \ and\
  \bibinfo {author} {\bibfnamefont {G.-B.}\ \bibnamefont {Stan}},\ }\href
  {\doibase https://doi.org/10.1016/j.cels.2018.03.013} {\bibfield  {journal}
  {\bibinfo  {journal} {Cell Systems}\ }\textbf {\bibinfo {volume} {6}},\
  \bibinfo {pages} {508 } (\bibinfo {year} {2018})}\BibitemShut {NoStop}%
\bibitem [{\citenamefont {Meyn}\ and\ \citenamefont
  {Tweedie}(1993{\natexlab{a}})}]{Meyn1993a}%
  \BibitemOpen
  \bibfield  {author} {\bibinfo {author} {\bibfnamefont {S.~P.}\ \bibnamefont
  {Meyn}}\ and\ \bibinfo {author} {\bibfnamefont {R.~L.}\ \bibnamefont
  {Tweedie}},\ }\href {\doibase 10.2307/1427521} {\bibfield  {journal}
  {\bibinfo  {journal} {Adv. in Appl. Probab.}\ }\textbf {\bibinfo {volume}
  {25}},\ \bibinfo {pages} {487} (\bibinfo {year}
  {1993}{\natexlab{a}})}\BibitemShut {NoStop}%
\bibitem [{\citenamefont {Hemberg}\ and\ \citenamefont
  {Barahona}(2007)}]{hemberg2007}%
  \BibitemOpen
  \bibfield  {author} {\bibinfo {author} {\bibfnamefont {M.}~\bibnamefont
  {Hemberg}}\ and\ \bibinfo {author} {\bibfnamefont {M.}~\bibnamefont
  {Barahona}},\ }\href {\doibase 10.1529/biophysj.106.099390} {\bibfield
  {journal} {\bibinfo  {journal} {Biophys. J.}\ }\textbf {\bibinfo {volume}
  {93}},\ \bibinfo {pages} {401} (\bibinfo {year} {2007})}\BibitemShut
  {NoStop}%
\bibitem [{\citenamefont {Hemberg}\ and\ \citenamefont
  {Barahona}(2008)}]{Hemberg2008}%
  \BibitemOpen
  \bibfield  {author} {\bibinfo {author} {\bibfnamefont {M.}~\bibnamefont
  {Hemberg}}\ and\ \bibinfo {author} {\bibfnamefont {M.}~\bibnamefont
  {Barahona}},\ }\href {\doibase 10.1186/1752-0509-2-42} {\bibfield  {journal}
  {\bibinfo  {journal} {BMC Syst. Biol.}\ }\textbf {\bibinfo {volume} {2}},\
  \bibinfo {pages} {42} (\bibinfo {year} {2008})}\BibitemShut {NoStop}%
\bibitem [{\citenamefont {Schnoerr}\ \emph {et~al.}(2017)\citenamefont
  {Schnoerr}, \citenamefont {Sanguinetti},\ and\ \citenamefont
  {Grima}}]{Schnoerr2017}%
  \BibitemOpen
  \bibfield  {author} {\bibinfo {author} {\bibfnamefont {D.}~\bibnamefont
  {Schnoerr}}, \bibinfo {author} {\bibfnamefont {G.}~\bibnamefont
  {Sanguinetti}}, \ and\ \bibinfo {author} {\bibfnamefont {R.}~\bibnamefont
  {Grima}},\ }\href {\doibase 10.1088/1751-8121/aa54d9} {\bibfield  {journal}
  {\bibinfo  {journal} {‎J. Phys. A}\ }\textbf {\bibinfo {volume} {50}},\
  \bibinfo {pages} {093001} (\bibinfo {year} {2017})}\BibitemShut {NoStop}%
\bibitem [{\citenamefont {Engblom}(2006)}]{engblom2006}%
  \BibitemOpen
  \bibfield  {author} {\bibinfo {author} {\bibfnamefont {S.}~\bibnamefont
  {Engblom}},\ }\href {\doibase 10.1016/j.amc.2005.12.032} {\bibfield
  {journal} {\bibinfo  {journal} {Appl. Math. Comput.}\ }\textbf {\bibinfo
  {volume} {180}},\ \bibinfo {pages} {498} (\bibinfo {year}
  {2006})}\BibitemShut {NoStop}%
\bibitem [{\citenamefont {Engblom}\ and\ \citenamefont
  {Pender}(2014)}]{engblom2014}%
  \BibitemOpen
  \bibfield  {author} {\bibinfo {author} {\bibfnamefont {S.}~\bibnamefont
  {Engblom}}\ and\ \bibinfo {author} {\bibfnamefont {J.}~\bibnamefont
  {Pender}},\ }\href@noop {} {\  (\bibinfo {year} {2014})},\ \Eprint
  {http://arxiv.org/abs/1406.6164} {arXiv:1406.6164} \BibitemShut {NoStop}%
\bibitem [{\citenamefont {Lakatos}\ \emph {et~al.}(2015)\citenamefont
  {Lakatos}, \citenamefont {Ale}, \citenamefont {Kirk},\ and\ \citenamefont
  {Stumpf}}]{lakatos2015}%
  \BibitemOpen
  \bibfield  {author} {\bibinfo {author} {\bibfnamefont {E.}~\bibnamefont
  {Lakatos}}, \bibinfo {author} {\bibfnamefont {A.}~\bibnamefont {Ale}},
  \bibinfo {author} {\bibfnamefont {P.~D.~W.}\ \bibnamefont {Kirk}}, \ and\
  \bibinfo {author} {\bibfnamefont {M.~P.~H.}\ \bibnamefont {Stumpf}},\ }\href
  {\doibase 10.1063/1.4929837} {\bibfield  {journal} {\bibinfo  {journal} {J.
  Chem. Phys.}\ }\textbf {\bibinfo {volume} {143}},\ \bibinfo {pages} {94107}
  (\bibinfo {year} {2015})}\BibitemShut {NoStop}%
\bibitem [{\citenamefont {Schnoerr}\ \emph {et~al.}(2015)\citenamefont
  {Schnoerr}, \citenamefont {Sanguinetti},\ and\ \citenamefont
  {Grima}}]{schnoerr2015}%
  \BibitemOpen
  \bibfield  {author} {\bibinfo {author} {\bibfnamefont {D.}~\bibnamefont
  {Schnoerr}}, \bibinfo {author} {\bibfnamefont {G.}~\bibnamefont
  {Sanguinetti}}, \ and\ \bibinfo {author} {\bibfnamefont {R.}~\bibnamefont
  {Grima}},\ }\href {\doibase 10.1063/1.4934990} {\bibfield  {journal}
  {\bibinfo  {journal} {J. Chem. Phys.}\ }\textbf {\bibinfo {volume} {143}},\
  \bibinfo {pages} {185101} (\bibinfo {year} {2015})}\BibitemShut {NoStop}%
\bibitem [{\citenamefont {{Van Kampen}}(1976)}]{vanKampen1976}%
  \BibitemOpen
  \bibfield  {author} {\bibinfo {author} {\bibfnamefont {N.~G.}\ \bibnamefont
  {{Van Kampen}}},\ }\href {\doibase 10.1016/0375-9601(76)90398-4} {\bibfield
  {journal} {\bibinfo  {journal} {Phys. Lett. A}\ }\textbf {\bibinfo {volume}
  {59}},\ \bibinfo {pages} {333} (\bibinfo {year} {1976})}\BibitemShut
  {NoStop}%
\bibitem [{\citenamefont {Schwerer}(1996)}]{Schwerer1996}%
  \BibitemOpen
  \bibfield  {author} {\bibinfo {author} {\bibfnamefont {E.}~\bibnamefont
  {Schwerer}},\ }\emph {\bibinfo {title} {{A linear programming approach to the
  steady-state analysis of Markov processes}}},\ \href@noop {} {Ph.D. thesis},\
  \bibinfo  {school} {Stanford University} (\bibinfo {year} {1996})\BibitemShut
  {NoStop}%
\bibitem [{\citenamefont {Helmes}\ and\ \citenamefont
  {Stockbridge}(2003)}]{Helmes2003}%
  \BibitemOpen
  \bibfield  {author} {\bibinfo {author} {\bibfnamefont {K.}~\bibnamefont
  {Helmes}}\ and\ \bibinfo {author} {\bibfnamefont {R.~H.}\ \bibnamefont
  {Stockbridge}},\ }\href {\doibase 10.1081/STM-120020389} {\bibfield
  {journal} {\bibinfo  {journal} {Stoch. Models}\ }\textbf {\bibinfo {volume}
  {19}},\ \bibinfo {pages} {255} (\bibinfo {year} {2003})}\BibitemShut
  {NoStop}%
\bibitem [{\citenamefont {Hern{\'{a}}ndez-Lerma}\ and\ \citenamefont
  {Lasserre}(2003)}]{Hernandez-Lerma2003}%
  \BibitemOpen
  \bibfield  {author} {\bibinfo {author} {\bibfnamefont {O.}~\bibnamefont
  {Hern{\'{a}}ndez-Lerma}}\ and\ \bibinfo {author} {\bibfnamefont {J.~B.}\
  \bibnamefont {Lasserre}},\ }\href {\doibase 10.1007/978-3-0348-8024-4} {\emph
  {\bibinfo {title} {{Markov chains and invariant probabilities}}}}\ (\bibinfo
  {publisher} {Birkh{\"{a}}user},\ \bibinfo {year} {2003})\BibitemShut
  {NoStop}%
\bibitem [{\citenamefont {Kuntz}\ \emph {et~al.}(2016)\citenamefont {Kuntz},
  \citenamefont {Ottobre}, \citenamefont {Stan},\ and\ \citenamefont
  {Barahona}}]{Kuntz2016}%
  \BibitemOpen
  \bibfield  {author} {\bibinfo {author} {\bibfnamefont {J.}~\bibnamefont
  {Kuntz}}, \bibinfo {author} {\bibfnamefont {M.}~\bibnamefont {Ottobre}},
  \bibinfo {author} {\bibfnamefont {G.-B.}\ \bibnamefont {Stan}}, \ and\
  \bibinfo {author} {\bibfnamefont {M.}~\bibnamefont {Barahona}},\ }\href
  {\doibase 10.1137/16M107801X} {\bibfield  {journal} {\bibinfo  {journal}
  {SIAM J. Sci. Comput.}\ }\textbf {\bibinfo {volume} {38}},\ \bibinfo {pages}
  {A3891} (\bibinfo {year} {2016})}\BibitemShut {NoStop}%
\bibitem [{\citenamefont {Hart}\ and\ \citenamefont
  {Tweedie}(2012)}]{Hart2012}%
  \BibitemOpen
  \bibfield  {author} {\bibinfo {author} {\bibfnamefont {A.~G.}\ \bibnamefont
  {Hart}}\ and\ \bibinfo {author} {\bibfnamefont {R.~L.}\ \bibnamefont
  {Tweedie}},\ }\href {\doibase 10.4236/am.2012.312A301} {\bibfield  {journal}
  {\bibinfo  {journal} {Appl. Math.}\ }\textbf {\bibinfo {volume} {3}},\
  \bibinfo {pages} {2205} (\bibinfo {year} {2012})}\BibitemShut {NoStop}%
\bibitem [{\citenamefont {Gupta}\ \emph {et~al.}(2017)\citenamefont {Gupta},
  \citenamefont {Mikelson},\ and\ \citenamefont {Khammash}}]{Gupta2017}%
  \BibitemOpen
  \bibfield  {author} {\bibinfo {author} {\bibfnamefont {A.}~\bibnamefont
  {Gupta}}, \bibinfo {author} {\bibfnamefont {J.}~\bibnamefont {Mikelson}}, \
  and\ \bibinfo {author} {\bibfnamefont {M.}~\bibnamefont {Khammash}},\ }\href
  {\doibase 10.1063/1.5006484} {\bibfield  {journal} {\bibinfo  {journal} {J.
  Chem. Phys.}\ }\textbf {\bibinfo {volume} {147}},\ \bibinfo {pages} {154101}
  (\bibinfo {year} {2017})}\BibitemShut {NoStop}%
\bibitem [{\citenamefont {Dayar}\ \emph
  {et~al.}(2011{\natexlab{a}})\citenamefont {Dayar}, \citenamefont {Hermanns},
  \citenamefont {Spieler},\ and\ \citenamefont {Wolf}}]{Dayar2011}%
  \BibitemOpen
  \bibfield  {author} {\bibinfo {author} {\bibfnamefont {T.}~\bibnamefont
  {Dayar}}, \bibinfo {author} {\bibfnamefont {H.}~\bibnamefont {Hermanns}},
  \bibinfo {author} {\bibfnamefont {D.}~\bibnamefont {Spieler}}, \ and\
  \bibinfo {author} {\bibfnamefont {V.}~\bibnamefont {Wolf}},\ }\href {\doibase
  10.1002/nla.795} {\bibfield  {journal} {\bibinfo  {journal} {Numer. Linear
  Algebra Appl.}\ }\textbf {\bibinfo {volume} {18}},\ \bibinfo {pages} {931}
  (\bibinfo {year} {2011}{\natexlab{a}})}\BibitemShut {NoStop}%
\bibitem [{\citenamefont {Spieler}(2014)}]{Spieler2014}%
  \BibitemOpen
  \bibfield  {author} {\bibinfo {author} {\bibfnamefont {D.}~\bibnamefont
  {Spieler}},\ }\emph {\bibinfo {title} {{Numerical analysis of long-run
  properties for Markov population models}}},\ \href@noop {} {Ph.D. thesis},\
  \bibinfo  {school} {Saarland University} (\bibinfo {year} {2014})\BibitemShut
  {NoStop}%
\bibitem [{\citenamefont {Dayar}\ \emph
  {et~al.}(2011{\natexlab{b}})\citenamefont {Dayar}, \citenamefont {Sandmann},
  \citenamefont {Spieler},\ and\ \citenamefont {Wolf}}]{Dayar2011a}%
  \BibitemOpen
  \bibfield  {author} {\bibinfo {author} {\bibfnamefont {T.}~\bibnamefont
  {Dayar}}, \bibinfo {author} {\bibfnamefont {W.}~\bibnamefont {Sandmann}},
  \bibinfo {author} {\bibfnamefont {D.}~\bibnamefont {Spieler}}, \ and\
  \bibinfo {author} {\bibfnamefont {V.}~\bibnamefont {Wolf}},\ }\href {\doibase
  10.1017/S0001867800005279} {\bibfield  {journal} {\bibinfo  {journal} {Adv.
  in Appl. Probab.}\ }\textbf {\bibinfo {volume} {43}},\ \bibinfo {pages}
  {1005} (\bibinfo {year} {2011}{\natexlab{b}})}\BibitemShut {NoStop}%
\bibitem [{\citenamefont {Kuntz}\ \emph {et~al.}(2017)\citenamefont {Kuntz},
  \citenamefont {Thomas}, \citenamefont {Stan},\ and\ \citenamefont
  {Barahona}}]{Kuntz2017}%
  \BibitemOpen
  \bibfield  {author} {\bibinfo {author} {\bibfnamefont {J.}~\bibnamefont
  {Kuntz}}, \bibinfo {author} {\bibfnamefont {P.}~\bibnamefont {Thomas}},
  \bibinfo {author} {\bibfnamefont {G.~B.}\ \bibnamefont {Stan}}, \ and\
  \bibinfo {author} {\bibfnamefont {M.}~\bibnamefont {Barahona}},\ }\href@noop
  {} {\  (\bibinfo {year} {2017})},\ \Eprint
  {http://arxiv.org/abs/1702.05468v1} {arXiv:1702.05468v1} \BibitemShut
  {NoStop}%
\bibitem [{\citenamefont {Kuntz}(2017)}]{Kuntzthe}%
  \BibitemOpen
  \bibfield  {author} {\bibinfo {author} {\bibfnamefont {J.}~\bibnamefont
  {Kuntz}},\ }\emph {\bibinfo {title} {{Deterministic approximation schemes
  with computable errors for the distributions of Markov chains}}},\ \href@noop
  {} {Ph.D. thesis},\ \bibinfo  {school} {Imperial College London} (\bibinfo
  {year} {2017})\BibitemShut {NoStop}%
\bibitem [{\citenamefont {Sakurai}\ and\ \citenamefont
  {Hori}(2017)}]{Sakurai2017}%
  \BibitemOpen
  \bibfield  {author} {\bibinfo {author} {\bibfnamefont {Y.}~\bibnamefont
  {Sakurai}}\ and\ \bibinfo {author} {\bibfnamefont {Y.}~\bibnamefont {Hori}},\
  }in\ \href {\doibase 10.1109/CDC.2017.8263820} {\emph {\bibinfo {booktitle}
  {IEEE 56th Annual Conference on Decision and Control}}}\ (\bibinfo {year}
  {2017})\ pp.\ \bibinfo {pages} {1206--1211}\BibitemShut {NoStop}%
\bibitem [{\citenamefont {Dowdy}\ and\ \citenamefont
  {Barton}(2017)}]{Dowdy2017}%
  \BibitemOpen
  \bibfield  {author} {\bibinfo {author} {\bibfnamefont {G.~R.}\ \bibnamefont
  {Dowdy}}\ and\ \bibinfo {author} {\bibfnamefont {P.~I.}\ \bibnamefont
  {Barton}},\ }\href {\doibase 10.1016/B978-0-444-63965-3.50375-5} {\bibfield
  {journal} {\bibinfo  {journal} {Computer Aided Chemical Engineering}\
  }\textbf {\bibinfo {volume} {40}},\ \bibinfo {pages} {2239} (\bibinfo {year}
  {2017})}\BibitemShut {NoStop}%
\bibitem [{\citenamefont {Ghusinga}\ \emph {et~al.}(2017)\citenamefont
  {Ghusinga}, \citenamefont {Vargas-Garcia}, \citenamefont {Lamperski},\ and\
  \citenamefont {Singh}}]{Ghusinga2017a}%
  \BibitemOpen
  \bibfield  {author} {\bibinfo {author} {\bibfnamefont {K.~R.}\ \bibnamefont
  {Ghusinga}}, \bibinfo {author} {\bibfnamefont {C.~A.}\ \bibnamefont
  {Vargas-Garcia}}, \bibinfo {author} {\bibfnamefont {A.}~\bibnamefont
  {Lamperski}}, \ and\ \bibinfo {author} {\bibfnamefont {A.}~\bibnamefont
  {Singh}},\ }\href {\doibase 10.1088/1478-3975/aa75c6} {\bibfield  {journal}
  {\bibinfo  {journal} {Phys. Biol.}\ }\textbf {\bibinfo {volume} {14}},\
  \bibinfo {pages} {04LT01} (\bibinfo {year} {2017})}\BibitemShut {NoStop}%
\bibitem [{\citenamefont {Kuntz}\ \emph
  {et~al.}(2018{\natexlab{a}})\citenamefont {Kuntz}, \citenamefont {Thomas},
  \citenamefont {Stan},\ and\ \citenamefont {Barahona}}]{Kuntz2018a}%
  \BibitemOpen
  \bibfield  {author} {\bibinfo {author} {\bibfnamefont {J.}~\bibnamefont
  {Kuntz}}, \bibinfo {author} {\bibfnamefont {P.}~\bibnamefont {Thomas}},
  \bibinfo {author} {\bibfnamefont {G.~B.}\ \bibnamefont {Stan}}, \ and\
  \bibinfo {author} {\bibfnamefont {M.}~\bibnamefont {Barahona}},\ }\href@noop
  {} {\  (\bibinfo {year} {2018}{\natexlab{a}})},\ \Eprint
  {http://arxiv.org/abs/1810.03658} {arXiv:1810.03658} \BibitemShut {NoStop}%
\bibitem [{\citenamefont {Norris}(1997)}]{Norris1997}%
  \BibitemOpen
  \bibfield  {author} {\bibinfo {author} {\bibfnamefont {J.~R.}\ \bibnamefont
  {Norris}},\ }\href {\doibase 10.1017/CBO9780511810633} {\emph {\bibinfo
  {title} {{Markov chains}}}}\ (\bibinfo  {publisher} {Cambridge University
  Press},\ \bibinfo {year} {1997})\BibitemShut {NoStop}%
\bibitem [{\citenamefont {Schl{\"{o}}gl}(1972)}]{schlogl1972}%
  \BibitemOpen
  \bibfield  {author} {\bibinfo {author} {\bibfnamefont {F.}~\bibnamefont
  {Schl{\"{o}}gl}},\ }\href {\doibase 10.1007/BF01379769} {\bibfield  {journal}
  {\bibinfo  {journal} {Z. Phys. A}\ }\textbf {\bibinfo {volume} {253}},\
  \bibinfo {pages} {147} (\bibinfo {year} {1972})}\BibitemShut {NoStop}%
\bibitem [{\citenamefont {Glynn}\ and\ \citenamefont
  {Zeevi}(2008)}]{Glynn2008}%
  \BibitemOpen
  \bibfield  {author} {\bibinfo {author} {\bibfnamefont {P.~W.}\ \bibnamefont
  {Glynn}}\ and\ \bibinfo {author} {\bibfnamefont {A.}~\bibnamefont {Zeevi}},\
  }in\ \href {\doibase 10.1214/074921708000000381} {\emph {\bibinfo {booktitle}
  {Markov Processes and Related Topics: A Festschrift for Thomas G. Kurtz}}}\
  (\bibinfo  {publisher} {Institute of Mathematical Statistics},\ \bibinfo
  {year} {2008})\ pp.\ \bibinfo {pages} {195--214}\BibitemShut {NoStop}%
\bibitem [{\citenamefont {Lasserre}(2009)}]{Lasserre2009}%
  \BibitemOpen
  \bibfield  {author} {\bibinfo {author} {\bibfnamefont {J.~B.}\ \bibnamefont
  {Lasserre}},\ }\href {\doibase 10.1142/p665} {\emph {\bibinfo {title}
  {{Moments, positive polynomials and their applications}}}}\ (\bibinfo
  {publisher} {Imperial College Press},\ \bibinfo {year} {2009})\BibitemShut
  {NoStop}%
\bibitem [{\citenamefont {Blekherman}\ \emph {et~al.}(2013)\citenamefont
  {Blekherman}, \citenamefont {Parrilo},\ and\ \citenamefont
  {Thomas}}]{Blekherman2013}%
  \BibitemOpen
  \bibinfo {editor} {\bibfnamefont {G.}~\bibnamefont {Blekherman}}, \bibinfo
  {editor} {\bibfnamefont {P.~A.}\ \bibnamefont {Parrilo}}, \ and\ \bibinfo
  {editor} {\bibfnamefont {R.~R.}\ \bibnamefont {Thomas}},\ eds.,\ \href
  {\doibase 10.1137/1.9781611972290.fm} {\emph {\bibinfo {title} {{Semidefinite
  optimization and convex algebraic geometry}}}}\ (\bibinfo  {publisher}
  {SIAM},\ \bibinfo {year} {2013})\BibitemShut {NoStop}%
\bibitem [{\citenamefont {Dowdy}\ and\ \citenamefont
  {Barton}(2018)}]{Dowdy2018}%
  \BibitemOpen
  \bibfield  {author} {\bibinfo {author} {\bibfnamefont {G.~R.}\ \bibnamefont
  {Dowdy}}\ and\ \bibinfo {author} {\bibfnamefont {P.~I.}\ \bibnamefont
  {Barton}},\ }\href {\doibase 10.1063/1.5009950} {\bibfield  {journal}
  {\bibinfo  {journal} {J. Chem. Phys.}\ }\textbf {\bibinfo {volume} {148}},\
  \bibinfo {pages} {084106} (\bibinfo {year} {2018})}\BibitemShut {NoStop}%
\bibitem [{\citenamefont {Lasserre}(2002)}]{Lasserre2002a}%
  \BibitemOpen
  \bibfield  {author} {\bibinfo {author} {\bibfnamefont {J.~B.}\ \bibnamefont
  {Lasserre}},\ }\href {\doibase 10.1090/S0002-9947-01-02898-7} {\bibfield
  {journal} {\bibinfo  {journal} {Trans. Amer. Math. Soc.}\ }\textbf {\bibinfo
  {volume} {354}},\ \bibinfo {pages} {631} (\bibinfo {year}
  {2002})}\BibitemShut {NoStop}%
\bibitem [{\citenamefont {L{\"{o}}fberg}(2004)}]{Lofberg2004}%
  \BibitemOpen
  \bibfield  {author} {\bibinfo {author} {\bibfnamefont {J.}~\bibnamefont
  {L{\"{o}}fberg}},\ }in\ \href {\doibase 10.1109/CACSD.2004.1393890} {\emph
  {\bibinfo {booktitle} {IEEE International Symposium on Computer Aided Control
  System Design}}}\ (\bibinfo {year} {2004})\ pp.\ \bibinfo {pages}
  {284--289}\BibitemShut {NoStop}%
\bibitem [{\citenamefont {Nakata}(2010)}]{Nakata2010}%
  \BibitemOpen
  \bibfield  {author} {\bibinfo {author} {\bibfnamefont {M.}~\bibnamefont
  {Nakata}},\ }in\ \href {\doibase 10.1109/CACSD.2010.5612693} {\emph {\bibinfo
  {booktitle} {EEE International Symposium on Computer-Aided Control System
  Design}}}\ (\bibinfo {year} {2010})\ pp.\ \bibinfo {pages}
  {29--34}\BibitemShut {NoStop}%
\bibitem [{\citenamefont {Fantuzzi}\ and\ \citenamefont
  {Fuentes}(2016)}]{Fantuzzi2016}%
  \BibitemOpen
  \bibfield  {author} {\bibinfo {author} {\bibfnamefont {G.}~\bibnamefont
  {Fantuzzi}}\ and\ \bibinfo {author} {\bibfnamefont {F.}~\bibnamefont
  {Fuentes}},\ }\href@noop {} {\enquote {\bibinfo {title} {{mpYALMIP (Version
  1.1.2), https://github.com/giofantuzzi/mpYALMIP}},}\ } (\bibinfo {year}
  {2016})\BibitemShut {NoStop}%
\bibitem [{\citenamefont {Fantuzzi}\ \emph {et~al.}(2016)\citenamefont
  {Fantuzzi}, \citenamefont {Goluskin}, \citenamefont {Huang},\ and\
  \citenamefont {Chernyshenko}}]{Fantuzzi2016a}%
  \BibitemOpen
  \bibfield  {author} {\bibinfo {author} {\bibfnamefont {G.}~\bibnamefont
  {Fantuzzi}}, \bibinfo {author} {\bibfnamefont {D.}~\bibnamefont {Goluskin}},
  \bibinfo {author} {\bibfnamefont {D.}~\bibnamefont {Huang}}, \ and\ \bibinfo
  {author} {\bibfnamefont {S.~I.}\ \bibnamefont {Chernyshenko}},\ }\href
  {\doibase 10.1137/15M1053347} {\bibfield  {journal} {\bibinfo  {journal}
  {SIAM J. Appl. Dyn. Syst.}\ }\textbf {\bibinfo {volume} {15}},\ \bibinfo
  {pages} {1962} (\bibinfo {year} {2016})}\BibitemShut {NoStop}%
\bibitem [{\citenamefont {Papp}\ and\ \citenamefont {Yildiz}(2017)}]{Papp2017}%
  \BibitemOpen
  \bibfield  {author} {\bibinfo {author} {\bibfnamefont {D.}~\bibnamefont
  {Papp}}\ and\ \bibinfo {author} {\bibfnamefont {S.}~\bibnamefont {Yildiz}},\
  }\href@noop {} {\  (\bibinfo {year} {2017})},\ \Eprint
  {http://arxiv.org/abs/1712.01792} {arXiv:1712.01792} \BibitemShut {NoStop}%
\bibitem [{\citenamefont {Zheng}\ \emph {et~al.}(2018)\citenamefont {Zheng},
  \citenamefont {Fantuzzi},\ and\ \citenamefont
  {Papachristodoulou}}]{Zheng2018}%
  \BibitemOpen
  \bibfield  {author} {\bibinfo {author} {\bibfnamefont {Y.}~\bibnamefont
  {Zheng}}, \bibinfo {author} {\bibfnamefont {G.}~\bibnamefont {Fantuzzi}}, \
  and\ \bibinfo {author} {\bibfnamefont {A.}~\bibnamefont
  {Papachristodoulou}},\ }\href {\doibase 10.1109/TAC.2018.2886170} {\bibfield
  {journal} {\bibinfo  {journal} {IEEE Trans. Automat. Contr.}\ } (\bibinfo
  {year} {2018}),\ 10.1109/TAC.2018.2886170}\BibitemShut {NoStop}%
\bibitem [{\citenamefont {CPLEX}(2009)}]{CPLEX}%
  \BibitemOpen
  \bibfield  {author} {\bibinfo {author} {\bibfnamefont {I.~I.}\ \bibnamefont
  {CPLEX}},\ }\href@noop {} {\enquote {\bibinfo {title} {{V12.1: User's Manual
  for CPLEX}},}\ } (\bibinfo {year} {2009})\BibitemShut {NoStop}%
\bibitem [{\citenamefont {Gardiner}(2009)}]{Gardiner2009}%
  \BibitemOpen
  \bibfield  {author} {\bibinfo {author} {\bibfnamefont {C.}~\bibnamefont
  {Gardiner}},\ }\href@noop {} {\emph {\bibinfo {title} {{Stochastic
  Methods}}}},\ \bibinfo {edition} {4th}\ ed.\ (\bibinfo  {publisher}
  {Springer-Verlag},\ \bibinfo {year} {2009})\BibitemShut {NoStop}%
\bibitem [{\citenamefont {Vallabhajosyula}\ \emph {et~al.}(2005)\citenamefont
  {Vallabhajosyula}, \citenamefont {Chickarmane},\ and\ \citenamefont
  {Sauro}}]{vallabhajosyula2005}%
  \BibitemOpen
  \bibfield  {author} {\bibinfo {author} {\bibfnamefont {R.~R.}\ \bibnamefont
  {Vallabhajosyula}}, \bibinfo {author} {\bibfnamefont {V.}~\bibnamefont
  {Chickarmane}}, \ and\ \bibinfo {author} {\bibfnamefont {H.~M.}\ \bibnamefont
  {Sauro}},\ }\href@noop {} {\bibfield  {journal} {\bibinfo  {journal}
  {Bioinformatics}\ }\textbf {\bibinfo {volume} {22}},\ \bibinfo {pages} {346}
  (\bibinfo {year} {2005})}\BibitemShut {NoStop}%
\bibitem [{Kun()}]{Kuntz2017code}%
  \BibitemOpen
  \href
  {https://github.com/barahona-research-group/Stationary-bounds-for-continuous-time-chains}
  {\enquote {\bibinfo {title}
  {{https://github.com/barahona-research-group/Stationary-bounds-for-continuous-time-chains}},}\
  }\BibitemShut {NoStop}%
\bibitem [{\citenamefont {Gardner}\ \emph {et~al.}(2000)\citenamefont
  {Gardner}, \citenamefont {Cantor},\ and\ \citenamefont
  {Collins}}]{gardner2000}%
  \BibitemOpen
  \bibfield  {author} {\bibinfo {author} {\bibfnamefont {T.~S.}\ \bibnamefont
  {Gardner}}, \bibinfo {author} {\bibfnamefont {C.~R.}\ \bibnamefont {Cantor}},
  \ and\ \bibinfo {author} {\bibfnamefont {J.~J.}\ \bibnamefont {Collins}},\
  }\href {\doibase 10.1038/35002131} {\bibfield  {journal} {\bibinfo  {journal}
  {Nature}\ }\textbf {\bibinfo {volume} {403}},\ \bibinfo {pages} {339}
  (\bibinfo {year} {2000})}\BibitemShut {NoStop}%
\bibitem [{\citenamefont {Perez-Carrasco}\ \emph {et~al.}(2016)\citenamefont
  {Perez-Carrasco}, \citenamefont {Guerrero}, \citenamefont {Briscoe},\ and\
  \citenamefont {Page}}]{perez2016}%
  \BibitemOpen
  \bibfield  {author} {\bibinfo {author} {\bibfnamefont {R.}~\bibnamefont
  {Perez-Carrasco}}, \bibinfo {author} {\bibfnamefont {P.}~\bibnamefont
  {Guerrero}}, \bibinfo {author} {\bibfnamefont {J.}~\bibnamefont {Briscoe}}, \
  and\ \bibinfo {author} {\bibfnamefont {K.~M.}\ \bibnamefont {Page}},\ }\href
  {\doibase 10.1371/journal.pcbi.1005154} {\bibfield  {journal} {\bibinfo
  {journal} {PLoS Comput. Biol.}\ }\textbf {\bibinfo {volume} {12}},\ \bibinfo
  {pages} {e1005154} (\bibinfo {year} {2016})}\BibitemShut {NoStop}%
\bibitem [{\citenamefont {Kumar}\ \emph {et~al.}(2014)\citenamefont {Kumar},
  \citenamefont {Platini},\ and\ \citenamefont {Kulkarni}}]{Kumar2014}%
  \BibitemOpen
  \bibfield  {author} {\bibinfo {author} {\bibfnamefont {N.}~\bibnamefont
  {Kumar}}, \bibinfo {author} {\bibfnamefont {T.}~\bibnamefont {Platini}}, \
  and\ \bibinfo {author} {\bibfnamefont {R.~V.}\ \bibnamefont {Kulkarni}},\
  }\href {\doibase 10.1103/PhysRevLett.113.268105} {\bibfield  {journal}
  {\bibinfo  {journal} {Phys. Rev. Lett.}\ }\textbf {\bibinfo {volume} {113}},\
  \bibinfo {pages} {268105} (\bibinfo {year} {2014})}\BibitemShut {NoStop}%
\bibitem [{\citenamefont {Shahrezaei}\ and\ \citenamefont
  {Swain}(2008)}]{shahrezaei2008}%
  \BibitemOpen
  \bibfield  {author} {\bibinfo {author} {\bibfnamefont {V.}~\bibnamefont
  {Shahrezaei}}\ and\ \bibinfo {author} {\bibfnamefont {P.~S.}\ \bibnamefont
  {Swain}},\ }\href@noop {} {\bibfield  {journal} {\bibinfo  {journal}
  {Proceedings of the National Academy of Sciences}\ } (\bibinfo {year}
  {2008})}\BibitemShut {NoStop}%
\bibitem [{\citenamefont {Sakurai}\ and\ \citenamefont
  {Hori}(2018)}]{Sakurai2018}%
  \BibitemOpen
  \bibfield  {author} {\bibinfo {author} {\bibfnamefont {Y.}~\bibnamefont
  {Sakurai}}\ and\ \bibinfo {author} {\bibfnamefont {Y.}~\bibnamefont {Hori}},\
  }\href {\doibase https://doi.org/10.1098/rsif.2017.0709} {\bibfield
  {journal} {\bibinfo  {journal} {Adv. Appl. Probab.}\ }\textbf {\bibinfo
  {volume} {15}},\ \bibinfo {pages} {20170709} (\bibinfo {year}
  {2018})}\BibitemShut {NoStop}%
\bibitem [{\citenamefont {Laurent}(2009)}]{Laurent2009}%
  \BibitemOpen
  \bibfield  {author} {\bibinfo {author} {\bibfnamefont {M.}~\bibnamefont
  {Laurent}},\ }in\ \href {\doibase 10.1007/978-0-387-09686-5_7} {\emph
  {\bibinfo {booktitle} {Emerging applications of algebraic geometry}}}\
  (\bibinfo  {publisher} {Springer-Verlag},\ \bibinfo {year} {2009})\ pp.\
  \bibinfo {pages} {157--270}\BibitemShut {NoStop}%
\bibitem [{\citenamefont {Tweedie}(1971)}]{Tweedie1971}%
  \BibitemOpen
  \bibfield  {author} {\bibinfo {author} {\bibfnamefont {R.~L.}\ \bibnamefont
  {Tweedie}},\ }\href {\doibase 10.2307/3211901} {\bibfield  {journal}
  {\bibinfo  {journal} {J. Appl. Probab.}\ }\textbf {\bibinfo {volume} {8}},\
  \bibinfo {pages} {311} (\bibinfo {year} {1971})}\BibitemShut {NoStop}%
\bibitem [{\citenamefont {Paulev{\'{e}}}\ \emph {et~al.}(2014)\citenamefont
  {Paulev{\'{e}}}, \citenamefont {Craciun},\ and\ \citenamefont
  {Koeppl}}]{pauleve2014}%
  \BibitemOpen
  \bibfield  {author} {\bibinfo {author} {\bibfnamefont {L.}~\bibnamefont
  {Paulev{\'{e}}}}, \bibinfo {author} {\bibfnamefont {G.}~\bibnamefont
  {Craciun}}, \ and\ \bibinfo {author} {\bibfnamefont {H.}~\bibnamefont
  {Koeppl}},\ }\href {\doibase 10.1007/s00285-013-0686-2} {\bibfield  {journal}
  {\bibinfo  {journal} {J. Math. Biol.}\ }\textbf {\bibinfo {volume} {69}},\
  \bibinfo {pages} {55} (\bibinfo {year} {2014})}\BibitemShut {NoStop}%
\bibitem [{\citenamefont {Gupta}\ and\ \citenamefont
  {Khammash}(2018)}]{Gupta2018}%
  \BibitemOpen
  \bibfield  {author} {\bibinfo {author} {\bibfnamefont {A.}~\bibnamefont
  {Gupta}}\ and\ \bibinfo {author} {\bibfnamefont {M.}~\bibnamefont
  {Khammash}},\ }\href {\doibase 10.1137/17M1134299} {\bibfield  {journal}
  {\bibinfo  {journal} {SIAM J. Appl. Dyn. Syst}\ }\textbf {\bibinfo {volume}
  {17}},\ \bibinfo {pages} {1213} (\bibinfo {year} {2018})}\BibitemShut
  {NoStop}%
\bibitem [{\citenamefont {Seneta}(1967)}]{Seneta1967}%
  \BibitemOpen
  \bibfield  {author} {\bibinfo {author} {\bibfnamefont {E.}~\bibnamefont
  {Seneta}},\ }\href {\doibase 10.1017/S0305004100042006} {\bibfield  {journal}
  {\bibinfo  {journal} {Proc. Camb. Phil. Soc.}\ }\textbf {\bibinfo {volume}
  {63}},\ \bibinfo {pages} {983} (\bibinfo {year} {1967})}\BibitemShut
  {NoStop}%
\bibitem [{\citenamefont {Tweedie}(1998)}]{Tweedie1998}%
  \BibitemOpen
  \bibfield  {author} {\bibinfo {author} {\bibfnamefont {R.~L.}\ \bibnamefont
  {Tweedie}},\ }\href {\doibase 10.1017/S0021900200016181} {\bibfield
  {journal} {\bibinfo  {journal} {J. Appl. Probab.}\ }\textbf {\bibinfo
  {volume} {35}},\ \bibinfo {pages} {517} (\bibinfo {year} {1998})}\BibitemShut
  {NoStop}%
\bibitem [{\citenamefont {Meyn}\ and\ \citenamefont
  {Tweedie}(1994)}]{Meyn1994}%
  \BibitemOpen
  \bibfield  {author} {\bibinfo {author} {\bibfnamefont {S.~P.}\ \bibnamefont
  {Meyn}}\ and\ \bibinfo {author} {\bibfnamefont {R.~L.}\ \bibnamefont
  {Tweedie}},\ }\href {\doibase 10.1214/aoap/1177004900} {\bibfield  {journal}
  {\bibinfo  {journal} {Ann. Appl. Probab.}\ }\textbf {\bibinfo {volume} {4}},\
  \bibinfo {pages} {981} (\bibinfo {year} {1994})}\BibitemShut {NoStop}%
\bibitem [{\citenamefont {Liu}(2015)}]{Liu2015}%
  \BibitemOpen
  \bibfield  {author} {\bibinfo {author} {\bibfnamefont {Y.}~\bibnamefont
  {Liu}},\ }\href {\doibase 10.1007/s11425-015-5019-z} {\bibfield  {journal}
  {\bibinfo  {journal} {Sci. China Math.}\ }\textbf {\bibinfo {volume} {58}},\
  \bibinfo {pages} {2633} (\bibinfo {year} {2015})}\BibitemShut {NoStop}%
\bibitem [{\citenamefont {Masuyama}(2017{\natexlab{a}})}]{Masuyama2017}%
  \BibitemOpen
  \bibfield  {author} {\bibinfo {author} {\bibfnamefont {H.}~\bibnamefont
  {Masuyama}},\ }\href {\doibase 10.15807/jorsj.60.271} {\bibfield  {journal}
  {\bibinfo  {journal} {J. Oper. Res. Soc. Japan}\ }\textbf {\bibinfo {volume}
  {60}},\ \bibinfo {pages} {271} (\bibinfo {year}
  {2017}{\natexlab{a}})}\BibitemShut {NoStop}%
\bibitem [{\citenamefont {Masuyama}(2017{\natexlab{b}})}]{Masuyama2017a}%
  \BibitemOpen
  \bibfield  {author} {\bibinfo {author} {\bibfnamefont {H.}~\bibnamefont
  {Masuyama}},\ }\href {\doibase 10.1016/J.LAA.2016.10.014} {\bibfield
  {journal} {\bibinfo  {journal} {Linear Algebra Appl.}\ }\textbf {\bibinfo
  {volume} {514}},\ \bibinfo {pages} {105} (\bibinfo {year}
  {2017}{\natexlab{b}})}\BibitemShut {NoStop}%
\bibitem [{\citenamefont {Liu}\ \emph {et~al.}(2018)\citenamefont {Liu},
  \citenamefont {Li},\ and\ \citenamefont {Masuyama}}]{Liu2018}%
  \BibitemOpen
  \bibfield  {author} {\bibinfo {author} {\bibfnamefont {Y.}~\bibnamefont
  {Liu}}, \bibinfo {author} {\bibfnamefont {W.}~\bibnamefont {Li}}, \ and\
  \bibinfo {author} {\bibfnamefont {H.}~\bibnamefont {Masuyama}},\ }\href
  {\doibase 10.1016/J.ORL.2018.05.001} {\bibfield  {journal} {\bibinfo
  {journal} {Oper. Res. Lett.}\ }\textbf {\bibinfo {volume} {46}},\ \bibinfo
  {pages} {409} (\bibinfo {year} {2018})}\BibitemShut {NoStop}%
\bibitem [{\citenamefont {Liu}\ and\ \citenamefont {Li}(2018)}]{Liu2018a}%
  \BibitemOpen
  \bibfield  {author} {\bibinfo {author} {\bibfnamefont {Y.}~\bibnamefont
  {Liu}}\ and\ \bibinfo {author} {\bibfnamefont {W.}~\bibnamefont {Li}},\
  }\href {\doibase 10.1017/apr.2018.28} {\bibfield  {journal} {\bibinfo
  {journal} {Adv. Appl. Probab.}\ }\textbf {\bibinfo {volume} {50}},\ \bibinfo
  {pages} {645} (\bibinfo {year} {2018})}\BibitemShut {NoStop}%
\bibitem [{\citenamefont {Fox}\ \emph {et~al.}(2016)\citenamefont {Fox},
  \citenamefont {Neuert},\ and\ \citenamefont {Munsky}}]{fox2016}%
  \BibitemOpen
  \bibfield  {author} {\bibinfo {author} {\bibfnamefont {Z.}~\bibnamefont
  {Fox}}, \bibinfo {author} {\bibfnamefont {G.}~\bibnamefont {Neuert}}, \ and\
  \bibinfo {author} {\bibfnamefont {B.}~\bibnamefont {Munsky}},\ }\href
  {\doibase 10.1063/1.4960505} {\bibfield  {journal} {\bibinfo  {journal} {J.
  Chem. Phys.}\ }\textbf {\bibinfo {volume} {145}},\ \bibinfo {pages} {74101}
  (\bibinfo {year} {2016})}\BibitemShut {NoStop}%
\bibitem [{\citenamefont {Gillespie}(1976)}]{Gillespie1976}%
  \BibitemOpen
  \bibfield  {author} {\bibinfo {author} {\bibfnamefont {D.~T.}\ \bibnamefont
  {Gillespie}},\ }\href {\doibase 10.1016/0021-9991(76)90041-3} {\bibfield
  {journal} {\bibinfo  {journal} {J. Comput. Phys.}\ }\textbf {\bibinfo
  {volume} {22}},\ \bibinfo {pages} {403} (\bibinfo {year} {1976})}\BibitemShut
  {NoStop}%
\bibitem [{\citenamefont {Kendall}(1950)}]{Kendall1950}%
  \BibitemOpen
  \bibfield  {author} {\bibinfo {author} {\bibfnamefont {D.~G.}\ \bibnamefont
  {Kendall}},\ }\href@noop {} {\bibfield  {journal} {\bibinfo  {journal} {J. R.
  Stat. Soc. Ser. B Stat. Methodol.}\ }\textbf {\bibinfo {volume} {12}},\
  \bibinfo {pages} {116} (\bibinfo {year} {1950})}\BibitemShut {NoStop}%
\bibitem [{\citenamefont {Feller}(1940)}]{Feller1940}%
  \BibitemOpen
  \bibfield  {author} {\bibinfo {author} {\bibfnamefont {W.}~\bibnamefont
  {Feller}},\ }\href {\doibase 10.1090/S0002-9947-1940-0002697-3} {\bibfield
  {journal} {\bibinfo  {journal} {Trans. Amer. Math. Soc.}\ }\textbf {\bibinfo
  {volume} {48}},\ \bibinfo {pages} {488} (\bibinfo {year} {1940})}\BibitemShut
  {NoStop}%
\bibitem [{\citenamefont {Kuntz}\ \emph
  {et~al.}(2018{\natexlab{b}})\citenamefont {Kuntz}, \citenamefont {Thomas},
  \citenamefont {Stan},\ and\ \citenamefont {Barahona}}]{Kuntz2018}%
  \BibitemOpen
  \bibfield  {author} {\bibinfo {author} {\bibfnamefont {J.}~\bibnamefont
  {Kuntz}}, \bibinfo {author} {\bibfnamefont {P.}~\bibnamefont {Thomas}},
  \bibinfo {author} {\bibfnamefont {G.~B.}\ \bibnamefont {Stan}}, \ and\
  \bibinfo {author} {\bibfnamefont {M.}~\bibnamefont {Barahona}},\ }\href@noop
  {} {\  (\bibinfo {year} {2018}{\natexlab{b}})},\ \Eprint
  {http://arxiv.org/abs/1801.09507} {arXiv:1801.09507} \BibitemShut {NoStop}%
\bibitem [{\citenamefont {Kurtz}\ and\ \citenamefont
  {Stockbridge}(1998)}]{Kurtz1998}%
  \BibitemOpen
  \bibfield  {author} {\bibinfo {author} {\bibfnamefont {T.~G.}\ \bibnamefont
  {Kurtz}}\ and\ \bibinfo {author} {\bibfnamefont {R.~H.}\ \bibnamefont
  {Stockbridge}},\ }\href {\doibase 10.1137/S0363012995295516} {\bibfield
  {journal} {\bibinfo  {journal} {SIAM J. Control Optim.}\ }\textbf {\bibinfo
  {volume} {36}},\ \bibinfo {pages} {609} (\bibinfo {year} {1998})}\BibitemShut
  {NoStop}%
\bibitem [{\citenamefont {Reuter}(1957)}]{Reuter1957}%
  \BibitemOpen
  \bibfield  {author} {\bibinfo {author} {\bibfnamefont {G.~E.~H.}\
  \bibnamefont {Reuter}},\ }\href {\doibase 10.1007/BF02392391} {\bibfield
  {journal} {\bibinfo  {journal} {Acta Math.}\ }\textbf {\bibinfo {volume}
  {97}},\ \bibinfo {pages} {1} (\bibinfo {year} {1957})}\BibitemShut {NoStop}%
\bibitem [{\citenamefont {Chen}(1991)}]{Chen1991}%
  \BibitemOpen
  \bibfield  {author} {\bibinfo {author} {\bibfnamefont {M.}~\bibnamefont
  {Chen}},\ }\href {\doibase 10.2307/3214868} {\bibfield  {journal} {\bibinfo
  {journal} {J. Appl. Probab.}\ }\textbf {\bibinfo {volume} {28}},\ \bibinfo
  {pages} {305} (\bibinfo {year} {1991})}\BibitemShut {NoStop}%
\bibitem [{\citenamefont {Meyn}\ and\ \citenamefont
  {Tweedie}(1993{\natexlab{b}})}]{Meyn1993b}%
  \BibitemOpen
  \bibfield  {author} {\bibinfo {author} {\bibfnamefont {S.~P.}\ \bibnamefont
  {Meyn}}\ and\ \bibinfo {author} {\bibfnamefont {R.~L.}\ \bibnamefont
  {Tweedie}},\ }\href {\doibase 10.2307/1427522} {\bibfield  {journal}
  {\bibinfo  {journal} {Adv. Appl. Probab.}\ }\textbf {\bibinfo {volume}
  {25}},\ \bibinfo {pages} {518} (\bibinfo {year}
  {1993}{\natexlab{b}})}\BibitemShut {NoStop}%
\end{thebibliography}%

\end{document}